\documentclass[12pt]{amsart}
\usepackage{latexsym}
\usepackage{amssymb} 
\usepackage{mathrsfs}
\usepackage{amsmath}
\usepackage{latexsym}
\usepackage{delarray}
\usepackage{amssymb,amsmath,amsfonts,amsthm,
mathrsfs}

\usepackage{mathtools}
\usepackage[section]{placeins}

\usepackage{hyperref}  
\renewcommand\eqref[1]{(\ref{#1})} 

\usepackage{xcolor}

\newtheorem{theorem}{Theorem}[section]
\newtheorem{corollary}[theorem]{Corollary}
\newtheorem{lemma}[theorem]{Lemma}
\newtheorem{proposition}[theorem]{Proposition}
\newtheorem{definition}[theorem]{Definition}
\theoremstyle{definition}
\newtheorem{remark}[theorem]{Remark}

\newcommand{\wt}[1]{\widetilde{#1}}

\newcommand\Rn{{\mathbb R}^n}

\renewcommand{\S}{\mathscr{S}}
\newcommand{\E}{\ensuremath{{\mathcal E}}}

\newcommand{\mb}[1]{\ensuremath{\mathbb{#1}}}
\newcommand{\N}{\mb{N}}

\newcommand{\R}{\mb{R}}
\newcommand{\C}{\mb{C}}

\renewcommand\N{{\mathbb N}_0}

\newcommand{\beq}{\begin{equation}}
\newcommand{\eeq}{\end{equation}}

\newcommand{\eps}{\varepsilon}

\newcommand{\lara}[1]{\langle #1 \rangle}

\newcommand{\esp}{\mathrm{e}}

\newcommand{\supp}{\mathrm{supp}}

\title[Inhomogeneous hyperbolic equations]
{Inhomogeneous wave equation with $t$-dependent singular coefficients}
\author[Marco Discacciati]{Marco Discacciati}
\address{
  Marco Discacciati:
  \endgraf
  Department of Mathematical Sciences
  \endgraf
  Loughborough University
  \endgraf
  Loughborough, Leicestershire, LE11 3TU
  \endgraf
  United Kingdom
  \endgraf
  {\it E-mail address} {\rm m.discacciati@lboro.ac.uk}
  }

\author[Claudia Garetto]{Claudia Garetto}
\address{
  Claudia Garetto:
  \endgraf
  Department of Mathematical Sciences
  \endgraf
  Loughborough University
  \endgraf
  Loughborough, Leicestershire, LE11 3TU
  \endgraf
  United Kingdom
  \endgraf
  {\it E-mail address} {\rm c.garetto@lboro.ac.uk}
  }
  
\author[Costas Loizou]{Costas Loizou}
\address{
  Costas Loizou:
  \endgraf
  Department of Mathematical Sciences
  \endgraf
  Loughborough University
  \endgraf
  Loughborough, Leicestershire, LE11 3TU
  \endgraf
  United Kingdom
  \endgraf  {\it E-mail address} {\rm c.loizou@lboro.ac.uk}
  }

\thanks{The second author was supported by the
EPSRC grant  EP/V005529/1}

\date{}


\begin{document}

\maketitle

\begin{abstract}
This paper is devoted to the study of the inhomogeneous wave equation with singular (less than continuous) time dependent coefficients. 
Particular attention is given to the role of the lower order terms and suitable Levi conditions are formulated in order to obtain a very weak solution as introduced in \cite{GR:14}.
Very weak solutions for this kind of equations are also investigated from a numerical point of view in two toy models: the wave equation with a Heaviside function and a delta distribution, respectively, as coefficient in its principal part.
\end{abstract}

\section{Introduction}
In this paper we want to study the well-posedness of the Cauchy problem on $[0,T]\times\R^n$ for wave equations of the type
\beq
\label{intro_eq_1}
\partial_t^2u(t,x)-\sum_{i=1}^n a_i(t)
\partial_{x_i}^2u(t,x)+l(t,\partial_t,\partial_x)u(t,x)=f(t,x),
\eeq
where
\[
l(t,\partial_t,\partial_x)=\sum_{i=1}^n c_i(t)\partial_{x_i}+d(t)\partial_t+e(t).
\]
We assume that the coefficients of the principal part are real and that $a_i(t)\ge 0$ for all $t\in[0,T]$ and $i=1,\dots,n$. It follows that this equation is hyperbolic
but not necessarily strictly hyperbolic. This equation has been widely investigated when the coefficients are regular (for instance, continuous, $C^k$ or $C^\infty$ \cite{ColKi:02, CDS}) but this is not the case when the coefficients are discontinuous or more in general singular (distributional). Note that, in view of
the famous Schwartz impossibility result on multiplication of distributions
\cite{Schwartz:impossibility-1954}, a fundamental question arises immediately: how to interpret the equation
\eqref{intro_eq_1} when both coefficients and $u$ are distributions? This question has been answered in \cite{GR:14} by introducing the notion of a \emph{very weak solution}.
However, the results obtained in \cite{GR:14} hold only for homogeneous hyperbolic equations ($l=0$ and $f=0$). In this paper we want to complete the study started in \cite{GR:14} by formulating suitable Levi conditions on the lower order terms adapted to the case of singular coefficients.

Hyperbolic equations model different phenomena in physics from propagation of waves into a multi-layered medium to refraction in crystals. From a purely analytical point of view they present very interesting properties and their study requires a careful combination of different techniques, from linear algebra (symmetrisation, quasi-symmetrisation, diagonalisation, etc.) to functional analysis and operator theory (energy estimates, pseudo-differential and Fourier-integral operators, microlocal analysis, etc.). Every hyperbolic equation of order $m$ is associated to its characteristic polynomial which has exactly $m$ real roots. When they are distinct the equation is called strictly hyperbolic, when the roots have multiplicities the equation is called weakly hyperbolic. While strictly hyperbolic have been widely investigated by the mathematical community (see \cite{CDS, Hoer:71} and references therein) the same is not true for weakly hyperbolic equations. For this last class of equations the well-posedness of the corresponding Cauchy problem heavily depends on the regularity of the coefficients, on the assumptions on the lower order terms (Levi conditions) and the choice of the function space where to work. Usually, one obtains results of Gevrey well-posedness when $C^\infty$ well-posedness cannot be achieved. For a non exhausting literature overview on the subject we refer the reader to \cite{ColKi:02, CS, CDS, G:20, GarJ, GR:11, GR:12, GarRuz:3, GarRuz:7, GJR1, GJR2, Hoer:71, Hord, JT, KY:06, KR2, KS, O70, PP}. In this paper we address a problem which presents two level of difficulty: its coefficients are singular (so the majority of the known analytical methods cannot be directly applied) and it has multiplicities. 

In the sequel we summarise the state of the art on the Cauchy problem 
\beq
\label{intro_eq}
\begin{split}
\partial_t^2u(t,x)-\sum_{i=1}^n a_i(t)
\partial_{x_i}^2u(t,x)+l(t,\partial_t,\partial_x)u(t,x)&=f(t,x),\\
u(0,x)&=g_0(x),\\
\partial_t u(0,x)&=g_1(x).
\end{split}
\eeq
\vspace{0.2cm}

We rewrite the equation in \eqref{intro_eq}, using the notation $D_t=-{\rm i}\partial_t$ and $D_{x_i}=-{\rm i}\partial_{x_i}$ to get 
\beq
\label{intro_eq_Dt}
\begin{split}
D_t^2u(t,x)-\sum_{i=1}^n a_i(t)
D_{x_i}^2u(t,x)-l(t,{\rm i}D_t,{\rm i}D_x)u(t,x)&=-f(t,x),\\
u(0,x)&=g_0(x),\\
D_t u(0,x)&=-{\rm i}g_1(x).
\end{split}
\eeq
Note that this equation has characteristic polynomial
\[
P(t,\tau,\xi)=\tau^2-\sum_{i=1}^n a_i(t)\xi_i^2,
\]
with real roots 
\begin{align} \label{eigenvalues_regular}
\lambda_{1,2} = \pm \sqrt{\sum_{i=1}^n a_i(t)\xi_i^2},
\end{align}
which coincide, for $\xi\neq 0$, when 
\[
a_i(t)=0 \quad \forall i=1,\,\dots,\,n.
\]
Now we adopt the notation introduced in \cite{KS} and \cite{GR:12} to denote by  $A_{2-j+1}$  the operator of order $2-j+1$ and by $A_{(2-j+1)}$ its principal part, for $j=1,\,2$. We rewrite \eqref{intro_eq_Dt} as

\beq
\label{intro_eq_Dt_Aj}
\begin{split}
D_t^2u(t,x) + A_{1}(t,D_x)D_t u(t,x)+ A_{2}(t,D_x)u(t,x)&=-f(t,x),\\
u(0,x)&=g_0(x),\\
D_t u(0,x)&=-{\rm i}g_1(x),
\end{split}
\eeq

\begin{align*}
&\text{where } A_{1}=- {\rm i}d(t) \text{ and hence }A_{(1)}=0 \\
&\text{and } A_{2}=-\sum_{i=1}^n a_i(t)D_{x_i}^2 -{\rm i}\sum_{i=1}^n c_i(t)D_{x_i}-e(t) \text{ and hence }A_{(2)}= -\sum_{i=1}^n a_i(t)D_{x_i}^2.
\end{align*}

In \cite{GR:12} it has been proven that when the equation coefficients are regular enough, namely of class $C^k$ with $k\ge 2$, and $f$ is identically zero then the Cauchy problem above is well-posed in Gevrey classes, provided that {\bf two sets of hypotheses} are fulfilled: on the {\bf roots} $\lambda_j$ and on the {\bf lower order terms}. 

In detail, we assume that
\begin{itemize}
\item the roots $\lambda_j(t,\xi)$ fulfil the condition introduced by Kinoshita and Spagnolo in  \cite{KS}:
\begin{align}
\label{KS_condition}
\exists M>0:\quad
\lambda_1(t,\xi)^2+\lambda_2(t,\xi)^2\le M(\lambda_1(t,\xi)-\lambda_2(t,\xi))^2,\quad t\in[0,T],  \textrm{ for all }\xi.
\end{align}
\item and the lower order terms satisfy the Levi conditions stated in \cite{GR:12}: for all $j=1,\,2$ there exists $C_j>0$ such that
\begin{equation}\label{EQ:lot2}
\begin{split}
\left|(A_{2}-A_{(2)})(t,\xi)\right|^2 &\leq
C_1 (\lambda^2_1(t,\xi)+\lambda_2^2(t,\xi)),\\
\left|(A_{1}-A_{(1)})(t,\xi)\right|^2 &\leq
C_2,
\end{split}
\end{equation} 
for all $t\in [0,T]$ and for $\xi$ away from $0$ (i.e., for $|\xi|\ge R$ for some $R>0$).
\end{itemize}
Note that the Kinoshita-Spagnolo condition on the roots can be stated more in general for any hyperbolic equation of order $m\ge 2$ and it is trivially fulfilled when $m=2$. Indeed, for $\lambda_{1,2}$ as in \eqref{eigenvalues_regular} we have 
\begin{align*}
\exists M>0:\quad
&2\sum_{i=1}^n a_i(t)\xi_i^2 \le M \left( 4\sum_{i=1}^n a_i(t)\xi_i^2\right),
\end{align*}
which is satisfied for $M=1$ since $a_i(t)\ge 0$ for all $t\in[0,T]$ and $i=1,\,\dots,\, n$. The Levi conditions \eqref{EQ:lot2} can be reformulated as follows: $\exists C_1,\,C_2>0$ such that
\begin{align}
\label{EQ:lot2_general}
\left|{\rm i}\sum_{i=1}^n c_i(t)\xi_{i}+e(t) \right|^2 &\le C_1 \left(  2\sum_{i=1}^n a_i(t)\xi_i^2 \right), \\ \nonumber
|d(t)|^2&\le C_2,
\end{align}
for $t\in [0,T]$ and $\xi$ away from $0$. The condition on $d$ is trivial if this coefficient is at least continuous. Summarising, making use of the notations introduced so far, we can state the well-posedness result proven in \cite{GR:12}.
\begin{theorem}[Theorem 2 in \cite{GR:12}]
\label{theo_GR}
Let
\beq
\begin{split}
D_t^2u(t,x) + A_{1}(t,D_x)D_t u(t,x)+ A_{2}(t,D_x)u(t,x)&=0,\\
u(0,x)&=g_0(x),\\
D_t u(0,x)&=-{\rm i}g_1(x),
\end{split}
\eeq
where the operators $A_1$ and $A_2$ are defined in \eqref{intro_eq_Dt_Aj}. Assume that the equation coefficients are continuous with the ones of the principal part of class $C^k$ with $k\ge 2$. Let $g_0,\,g_1\in \gamma^s(\R^n)$. If there exists $C_1>0$ such that 
\[
\left|(A_{2}-A_{(2)})(t,\xi)\right|^2 \leq
C_1 (\lambda^2_1(t,\xi)+\lambda_2^2(t,\xi)),
\]
for all $t\in[0,T]$ and for $\xi$ away from $0$, then the Cauchy problem above has a unique solution $u\in C^2([0,T]; \gamma^s(\R^n))$ provided that 
\[
1\le s\le 1+\frac{k}{2}.
\]

\end{theorem}
It is our aim in this paper to generalise Theorem \ref{theo_GR}  in two different ways:
\begin{enumerate}
\item by introducing a non-identically zero right-hand side $f$;
\item by lowering the regularity of the coefficients.
\end{enumerate}
This will involve a mix of techniques from quasi-symmetriser to regularisations and very weak solutions. We will first extend Theorem \ref{theo_GR} to $f\neq 0$ and we will then pass to consider non-regular coefficients. For a better understanding of very weak solutions and regularisation methods we refer the reader to \cite{GR:14} and to the recent work by Ruzhansky and collaborators in \cite{ART:19, ARST:20, ARST:20b, ARST:20c, RT, RY}.
 
The paper is organised as follows: In Section 2 we introduce the right-hand side $f$ and we prove well-posedness of the Cauchy problem for the weakly hyperbolic inhomogeneous wave equation with smooth coefficients. In Section 3 we pass to consider singular coefficients and we summarise the elements of the theory of very weak solutions that are needed for this paper. Levi conditions on the lower order terms are formulated in Section 4 as a generalisation of the ones introduced in \cite{GR:12} for regular coefficients. Our study of the Cauchy problem for the inhomogeneous wave equation with singular coefficients is organised in two sections: Section 5 where we assume that right-hand side $f$ and initial data are Gevrey functions (Case 1) and Section 6 where the right-hand side and initial data are compactly supported distributions (Case 2). Uniqueness in the very weak sense and consistency with classical results are discussed in Section7. Finally, in Section 8 we focus on two toy models: the wave equation in dimension 1 with a Heaviside function and a delta distribution, respectively, as a coefficient in the principal part. We recall the results known for these equations (with a particular focus on Oleinik's analysis of second order hyperbolic equations in \cite{O70}) and we prove consistency with our very weak approach. Note that our result on the wave equation with a Heaviside function as a coefficient extends a previous result in \cite{DHO13} known only when the jump is between two positive constants. Here we allow one of the constants to be $0$. The analysis of the toy models is supported by some numerical experiments where we investigate the limiting behaviour of very weak solutions and we show that it is independent of the employed regularisation method. The paper ends with an appendix where we collect the needed background on the quasi-symmetriser employed throughout the paper.

\section{Introduction of the right-hand side $f$}
We begin by working on the case of regular coefficients as in \cite{KS} and \cite{GR:12} and we show that the well-posedness result obtained in these papers holds also in presence of a right-hand side $f$ which is not identically zero.
This requires to employ the quasi-symmetriser associated to the equation whose general definition and main properties for matrices in Sylvester form of any size are collected in the appendix at the end of the paper. For the sake of simplicity, we write here only the quasi-symmetriser employed for equations of the second order and we refer the reader to the appendix for an exhaustive survey on the topic. Making use of the transformation
\[	
u_{1}=\lara{D_x}u, \quad u_{2}=D_t u,
\]
where  $\lara{D_x}$ is the pseudo-differential operator with symbol $\lara{\xi}=(1+|\xi|^2)^{\frac{1}{2}}$, we can transform the second order equation 
\[
D_t^2u(t,x)-\sum_{i=1}^n a_i(t)
D_{x_i}^2u(t,x)-l(t,{\rm i}D_t,{\rm i}D_x)u(t,x)=-f(t,x)
\]
or equivalently 
\[
D_t^2u(t,x) + A_{1}(t,D_x)D_t u(t,x)+ A_{2}(t,D_x)u(t,x)=-f(t,x)
\]
into the first order system
\[
D_t\left(
                             \begin{array}{c}
                               u_{1} \\                          
                               u_{2} \\
                             \end{array}
                           \right)
= M
  \left(\begin{array}{c}
                               u_{1} \\                             
                               u_{2} \\
                             \end{array}
                           \right)
                           -
\left(\begin{array}{c}
                               0 \\                               
                               f \\
                             \end{array}
                           \right),
\]
where
\[
\begin{split}
M &=\left(
    \begin{array}{ccccc}
      0 & \lara{D_x} \\
    -A_2(t,D_x)\lara{D_x}^{-1} & -A_1(t,D_x) \\
    \end{array}
  \right)\\
  &=\left(
    \begin{array}{ccccc}
      0 & \lara{D_x} \\
      \left( \sum_{i=1}^n a_{i}D_{x_i}^2 + {\rm i}\sum_{i=1}^n c_{i} D_{x_i} + e \right)\lara{D_x}^{-1} & {\rm i}d \\
    \end{array}
  \right).
  \end{split}
\]
The matrix above can be written as $M=\mathbb{A}_{1}+B$ with
\[
\mathbb{A}_{1}=\left(
    \begin{array}{ccccc}
      0 & \lara{D_x} \\
      \sum_{i=1}^n a_{i}D_{x_i}^2\lara{D_x}^{-1} & 0 \\
    \end{array}
  \right)
  \]
  and 
  \[
B=\left(
    \begin{array}{ccccc}
      0 & 0 \\
      \left({\rm i}\sum_{i=1}^n c_{i} D_{x_i} +e \right)\lara{D_x}^{-1} & {\rm i}d \\
    \end{array}
  \right).  
\]
Analogously, the initial conditions 
\[
u(0,x)=g_0(x),  \quad  D_t u(0,x)=-{\rm i}g_1(x),
\]
are transformed into 
\[
u_1(0,x)=\lara{D_x}g_0(x), \quad u_2(0,x)=-{\rm i}g_1(x).
\]
Setting $U=(u_1, u_2)^T$  and $F=(0,-f)^T$ we will therefore concentrate on the following reformulation of the Cauchy problem \eqref{intro_eq_Dt_Aj}:
\[
\begin{split}
D_t U&=\mathbb{A}_{1}(t,D_x)U+B(t,D_x)U+F,\\
U(0)&=(\lara{D_x}g_0, -{\rm i}g_1)^T.
\end{split}
\]
By Fourier transforming in $x$ (see \cite{GR:12}) we can equivalently study the system of ODEs
\beq
\label{system_new_V}
\begin{split}
D_t V&=\mathbb{A}_1(t,\xi)V+B(t,\xi)V+\widehat{F},\\
V|_{t=0}(\xi)&=V_0(\xi),
\end{split}
\eeq
where $V$ is the $2$-column with entries $v_j=\widehat{u}_j$ for $j=1,\,2$ and $V_0$ is the $2$-column with entries
$v_{0,1}=\lara{\xi}\widehat{g_0}$ and $v_{0,2}=-{\rm i}\widehat{g_1}$.

\subsection{Quasi-symmetriser for a second order hyperbolic equation}
In this subsection we investigate the matrix $\mathbb{A}_{1}(t,D_x)$. This is a matrix of first order pseudo-differential operators with symbol 
\begin{align*}
\mathbb{A}_1(t,\xi)&=\left(
    \begin{array}{ccccc}
      0 & 1\\
    -A_{(2)}(t,\xi)\lara{\xi}^{-2} & -A_{(1)}(t,\xi)\lara{\xi}^{-1}\\
    \end{array}
  \right)\lara{\xi}\\
  &=\left(
    \begin{array}{ccccc}
      0 & 1\\
      \sum_{i=1}^n a_{i}(t){\xi_i}^2\lara{\xi}^{-2} & 0 \\
    \end{array}
  \right)\lara{\xi}.
\end{align*}
The matrix 
\[A(t,\xi)\coloneqq\lara{\xi}^{-1}\mathbb{A}_{1}(t,\xi)=
\left(
    \begin{array}{ccccc}
      0 & 1\\
      \sum_{i=1}^n a_{i}(t){\xi_i}^2\lara{\xi}^{-2} & 0 \\
    \end{array}
  \right)
\]
is in Sylvester form and has eigenvalues 
\[ \tilde{\lambda}_{1}(t,\xi)=\lambda_1(t,\xi)\lara{\xi}^{-1} \text{ and }  \tilde{\lambda}_{2}(t,\xi)=\lambda_2(t,\xi)\lara{\xi}^{-1}.\]
 This matrix has a standard quasi-symmetriser (see the appendix at the end of the paper) which is an Hermitian matrix defined by 
\[
Q^{(2)}_{\delta}(\tilde{\lambda}_{1},\tilde{\lambda}_{2})=\left(
    \begin{array}{cc}
      \tilde{\lambda}_{1}^2+\tilde{\lambda}_{2}^2 & -(\tilde{\lambda}_{1}+\tilde{\lambda}_{2})\\
      -(\tilde{\lambda}_{1}+\tilde{\lambda}_{2}) & 2 \\
           \end{array}
  \right) + 2\delta^2 \left(
    \begin{array}{cc}
      1 & 0\\
      0 & 0 \\
           \end{array}
  \right),
\]
where $\delta\in (0,1]$ is a parameter. Note that for $2\times 2$ matrices $M_1$ and $M_2$ the notation $M_1\le M_2$ means $(M_1v,v)\le (M_2v,v)$ for all $v\in\C^2$ with $(\cdot,\cdot)$ the scalar product in $\C^2$. When the matrices $M_2$ and $M_2$ depend on variable or parameters we intend the inequality uniform with the respect to all the variables and parameters involved.
The quasi-symmetriser $Q^{(2)}_\delta$ has the following properties which we will employ in the rest of the paper:
\begin{itemize}
\item[(i)] there exists a constant $C>0$ such that 
\[
C\delta^2I\le Q^{(2)}_\delta\le C I;
\]
\item[(ii)] there exists a constant $C>0$ such that 
\[
|Q_\delta^{(2)}A-A^\ast Q_\delta^{(m)}(\lambda)|\le C\delta Q_\delta^{(2)};
\]
\item[(iii)] the matrix $Q^{(2)}_\delta$ is \emph{nearly diagonal}, i.e., there exists a constant $c_0>0$ such that 
\[
Q^{(2)}_\delta\ge c_0,{\rm diag}\,Q^{(2)}_\delta,
\]
where $${\rm diag}\,Q^{(2)}_\delta ={\rm diag}\{q_{\delta,11},\,q_{\delta, 22}\}.$$
\end{itemize}
These properties have been proven in \cite{KS, GR:12} and are discussed in more detail in the appendix which provides a short survey on the quasi-symmetriser for Sylvester type matrices of any size.
\subsection{Well-posedness result}
We are now ready to extend Theorem \ref{theo_GR} to equations with a non-identically zero right-hand side $f$. We recall that the operators $A_1$ and $A_2$ are defined by
\[
\begin{split}
A_{1}&=- {\rm i}d(t),\\
A_2&=-\sum_{i=1}^n a_i(t)D_{x_i}^2 -{\rm i}\sum_{i=1}^n c_i(t)D_{x_i}-e(t),
\end{split}
\]
with real coefficients $a_i$ and $a_i\ge 0$ for all $i=1,\,\dots,\,n$. In the proof we will make use of the so-called \emph{Fourier transform characterisation of Gevrey functions}:
\begin{itemize}
\item[(i)] Let $v\in \gamma^s_c(\R^n)$. Then, there exist constants $c>0$ and $\nu>0$ such that
\beq
\label{fou_gey_1}
|\widehat{v}(\xi)|\le c\,\esp^{-\nu\lara{\xi}^{\frac{1}{s}}}
\eeq
for all $\xi\in\R^n$.
\item[(ii)] Let $v\in\S'(\R^n)$. If there exist constants $c>0$ and $\nu>0$ such that \eqref{fou_gey_1} holds then $v\in \gamma^s(\R^n)$.
    \end{itemize}

\begin{theorem}
\label{theo_GRi}
Let 
\beq
\begin{split} \nonumber
D_t^2u(t,x) + A_{1}(t,D_x)D_t u(t,x)+ A_{2}(t,D_x)u(t,x)&=-f(t,x),\\
u(0,x)&=g_0(x),\\
D_t u(0,x)&=-{\rm i}g_1(x).
\end{split}
\eeq
Assume that the equation coefficients are continuous with the ones of the principal part of class $C^k$ with $k\ge 2$. Let $g_0,\,g_1\in \gamma^s_c(\R^n)$ and $f \in C([0,T];\gamma^s_c(\R^n))$. If there exists $C_1>0$ such that 
\beq
\label{LCT}
\left|(A_{2}-A_{(2)})(t,\xi)\right|^2 \leq
C_1 (\lambda^2_1(t,\xi)+\lambda_2^2(t,\xi))
\eeq
for all $t\in[0,T]$ and for $\xi$ away from $0$, then the Cauchy problem above has a unique solution $u\in C^2([0,T]; \gamma^s(\R^n))$ provided that 
\[
1\le s\le 1+\frac{k}{2}.
\] 
If the equation coefficients are continuous and the ones of the principal part are of class $C^\infty$ then under the Levi condition \eqref{LCT} the Cauchy problem is well-posed in every Gevrey class.
\end{theorem}

\begin{proof}
This proof is based on the proof of Theorem 6 in \cite{GR:12}. We recall that the quasi-symmetriser of $A(\lambda)$ is
\[
Q^{(2)}_{\delta}(\tilde{\lambda}_{1},\tilde{\lambda}_{2})=\left(
    \begin{array}{cc}
      \tilde{\lambda}_{1}^2+\tilde{\lambda}_{2}^2 & -(\tilde{\lambda}_{1}+\tilde{\lambda}_{2})\\
      -(\tilde{\lambda}_{1}+\tilde{\lambda}_{2}) & 2 \\
           \end{array}
  \right) + 2\delta^2 \left(
    \begin{array}{cc}
      1 & 0\\
      0 & 0 \\
           \end{array}
  \right),
\]
where $\tilde{\lambda}_{1,2}\coloneqq \lara{\xi}^{-1}\lambda_{1,2}(t,\xi)$, 
with $\lambda_{1,2}$ as in \eqref{eigenvalues_regular}. Hence the quasi-symmetriser becomes
\[
Q^{(2)}_{\delta}(\tilde{\lambda}_{1},\tilde{\lambda}_{2})=\left(
    \begin{array}{cc}
    2\sum_{i=1}^n a_{i}\xi_i^2\lara{\xi}^{-2} & 0\\[0.2cm]
      0 & 2 \\
           \end{array}
  \right)
   + 2\delta^2 \left(
    \begin{array}{cc}
      1 & 0\\
      0 & 0 \\
           \end{array}
  \right).
\]

Following Sections 4 and 5 of \cite{GR:12} and working on the system \eqref{system_new_V}, we define the energy 
\[
E_\delta(t,\xi)=(Q^{(2)}_\delta(t,\xi) V(t,\xi), V(t,\xi)).
\]
Following the proof of Theorem 6 in \cite{GR:12} we end up with the estimate
\begin{equation}\label{est-rev-regular}
\partial_t E_\delta(t,\xi)\le (K_\delta(t,\xi)+C_2\delta\lara{\xi}+C_3)E_\delta(t,\xi)+|(Q^{(2)}_{\delta} \widehat{F},V)- (Q^{(2)}_{\delta} V,\widehat{F})|,
\end{equation}
where $K_\delta(t,\xi)$ is defined in \cite[Theorem 6]{GR:12} and satisfies
\[
\int_0^T K_\delta(t,\xi)\, dt\le C_1\delta^{-2/k},
\]
and $C_1$, $C_2$, $C_3$ are positive constants. We will now deal with the additional last term which was not present in \cite{GR:12}. By direct computations we have that

\begin{align}\label{RHS_estimate_regular}
|(Q^{(2)}_{\delta} \widehat{F},V)- (Q^{(2)}_{\delta} V,\widehat{F})| 
&=2|\Im(Q^{(2)}_{\delta} \widehat{F},V)|\le 2|-2\widehat{f}\overline{V_2}|\\
	&= 4|V_{2}||\widehat{f}|	\le 2(|V_{2}|^2+|\widehat{F}|^2)\nonumber
\end{align}
Making now use of the fact that the quasi-symmetriser is (nearly) diagonal (see Corollary \ref{coroll_nearly_diagonal} in the appendix) we have that 
\begin{align*}
E_{\delta}(t,\xi)&=(Q^{(2)}_{\delta} V,V) = (\text{diag }Q^{(2)}_{\delta} V,V)	=\left(2\sum_{i=1}^n a_{i}\xi_i^2\lara{\xi}^{-2}+2\delta^2\right)|V_{1}|^2	+2|V_{2}|^2\\
	&\ge 2|V_{2}|^2,
\end{align*}
since $ a_i\ge0$. 
Using this inequality in \eqref{RHS_estimate_regular} we get that
\begin{align*}
|(Q^{(2)}_{\delta} \widehat{F},V)- (Q^{(2)}_{\delta} V,\widehat{F})| 	
	&\le 2|\widehat{F}|^2 + E_{\delta}(t,\xi).
\end{align*}

Therefore \eqref{est-rev-regular} becomes
\begin{align*}
\partial_t E_{\delta}(t,\xi) &\le (K_{\delta}(t,\xi)+C_2\delta\lara{\xi}+2C_3+1)E_{\delta}(t,\xi)+2|\widehat{F}(t,\xi)|^2 \\
&\le (K_{\delta}(t,\xi)+C_2\delta\lara{\xi}+2C_3+1)E_{\delta}(t,\xi)+2\underset{t \in [0,T]}{\sup}|\widehat{F}(t,\xi)|^2.
\end{align*}

By Gronwall's lemma we obtain 
\begin{align*}
E_{\delta}(t,\xi)&\le \left( E_{\delta}(0,\xi)+2T\underset{t \in [0,T]}{\sup}|\widehat{F}(t,\xi)|^2 \right) \esp^{C_1\delta^{-\frac{2}{k}}+C_2T\delta \lara{\xi}+2TC_3+T} \\
	&\le \left(  E_{\delta}(0,\xi)+2T\underset{t \in [0,T]}{\sup}|\widehat{F}(t,\xi)|^2 \right) \esp^{C_T(\delta^{-\frac{2}{k}}+\delta \lara{\xi}+2C_3+1)}\\
	&\le \left(  E_{\delta}(0,\xi)+2T\underset{t \in [0,T]}{\sup}|\widehat{F}(t,\xi)|^2 \right) C'_T\esp^{C'_T(\delta^{-\frac{2}{k}}+\delta \lara{\xi})}.
\end{align*}
Note that in the last term we have the same kind of inequality already analysed in the proof of Theorem 6 in \cite{GR:12}, if $g_0,\, g_1\in \gamma_c^s(\R^n)$  and $f \in C([0,T];\gamma_c^s(\R^n))$. Hence, by setting $\delta^{-2/k}=\delta\lara{\xi}$ and arguing as 
in the proof of Theorem 6 in \cite{GR:12} pages 430-431, we conclude that $V$ possesses the typical estimates
\[
|V(t,\xi)|\le c \esp^{-\nu\lara{\xi}^{\frac{1}{s}}}
\]
of a Gevrey function of order $s$ (uniformly with respect to $t\in[0,T]$) provided that 
\[
1\le s\le 1+\frac{k}{2}.
\] 
This entails the desired well-posedness result and well-posedness in every Gevrey class if $k$ can be chosen arbitrarily (as in the case of smooth coefficients $a_i$ and $b_i$).
\end{proof}

\begin{remark}
In the above theorem we choose  
\[g_0,\,g_1\in \gamma^s_c(\R^n) \text{ and } f \in C([0,T];\gamma^s_c(\R^n)).\]
 However, due to the finite speed of propagation, the previous result can be extended to $g_0,\,g_1\in \gamma^s(\R^n)$ and $f \in C([0,T];\gamma^s(\R^n))$.
\end{remark}

\section{Inhomogeneous wave equation with singular coefficients}
The rest of the paper is devoted to the Cauchy problem 
\beq
\begin{split} \nonumber
D_t^2u(t,x) + A_{1}(t,D_x)D_t u(t,x)+ A_{2}(t,D_x)u(t,x)&=-f(t,x),\\
u(0,x)&=g_0(x),\\
D_t u(0,x)&=-{\rm i}g_1(x).
\end{split}
\eeq
when the coefficients of the operators 
\[
\begin{split}
A_{1}&=- {\rm i}d(t),\\
A_2&=-\sum_{i=1}^n a_i(t)D_{x_i}^2 -{\rm i}\sum_{i=1}^n c_i(t)D_{x_i}-e(t),
\end{split}
\]
are singular, i.e., compactly supported distributions. In the next subsection, we recall what is known about this kind of problem.

\subsection{Starting point}
In \cite{GR:14} Garetto and Ruzhansky studied second order homogeneous hyperbolic equations of the type
\beq
\label{intro_CP_ARMA}
\begin{split}
\partial_t^2u(t,x)-\sum_{i=1}^n b_i(t)\partial_t\partial_{x_i}u(t,x)-\sum_{i=1}^n a_i(t)\partial_{x_i}^2u(t,x)&=0,\\
u(0,x)&=g_0,\\
\partial_t u(0,x)&=g_1,\\
\end{split}
\eeq
where the coefficients $a_i$, $b_i$  are distributions with compact support included in $[0,T]$, such that $a_i$, $b_i$ are real-valued and $a_i\ge 0$ for all $i=1,\,\dots,\,n$. 
The main result of \cite{GR:14} is the following theorem.
\begin{theorem}[Theorem 2.6 in \cite{GR:14}]
\label{theo_vws_ARMA}
Let the coefficients $a_i$, $b_i$ of the Cauchy problem \eqref{intro_CP_ARMA} be distributions 
with compact support included in $[0,T]$, such that $a_i$, $b_i$ are real-valued and $a_i\ge 0$ for all $i=1,\,\dots,\,n$. 
Let the Cauchy data $g_0$, $g_1$ be compactly supported distributions.
Then, the Cauchy problem \eqref{intro_CP_ARMA} has a very weak solution of order $s$, for all $s>1$.
\end{theorem}
The concept of a very weak solution has been introduced in \cite{GR:14} to be able to handle equations which might not have a meaningful solution in the usual distributional context. For the sake of the reader we recall the main points of the theory of very weak solutions in the next subsection.

\subsection{Regularisation techniques and very weak solutions}

The main idea in our approach for equations of the type 
\[
\partial_t^2u(t,x)-\sum_{i=1}^n a_i(t)
\partial_{x_i}^2u(t,x)+l(t,\partial_t,\partial_x)u(t,x)=f(t,x),
\]
where
\[
l(t,\partial_t,\partial_x)=\sum_{i=1}^n c_i(t)\partial_{x_i}+d(t)\partial_t+e(t).
\]
is to regularise the distributional coefficients $a_i$,  $c_i$, $d$ and $e$, $i=1,\,\dots,\, n$, via convolution with a suitable mollifier ($\psi\in C^\infty_c(\R)$, $\psi\ge 0$ with $\int\psi=1$) obtaining families of smooth functions $(a_{i,\eps})_\eps$,  $(c_{i,\eps})_\eps$, $(d_{\eps})_\eps$, $(e_{\eps})_\eps$, namely 
\begin{equation}\label{EQ:regs}
a_{i,\eps}=a_i\ast\psi_{\omega(\eps)},\, c_{i,\eps}=c_i\ast\psi_{\omega(\eps)},\, d_{\eps}=d\ast\psi_{\omega(\eps)} \textrm{ and } e_{\eps}=e\ast\psi_{\omega(\eps)} , \nonumber 
\end{equation}
where  $$\psi_{\omega(\eps)}(t)=\omega(\eps)^{-1}\psi(t/\omega(\eps))$$ and $\omega(\eps)$ is a positive function converging to $0$ as $\eps\to 0$. Analogously, we will regularise the right-hand side $f$ by choosing two different mollifiers: one to regularise with respect to the variable $t$ and one to regularise with respect to the variable $x$. This generates \emph{moderate} nets of smooth functions in the sense explained below. In the sequel, the notation $K\Subset\R^n$ stands for $K$ is a compact set in $\R^n$.
\begin{definition}
\label{def_mod_intro}
\leavevmode
\begin{itemize}
\item[(i)] A net of functions $(f_\eps)_\eps\in C^\infty(\R^n)^{(0,1]}$ is $C^\infty$-moderate if for all $K\Subset\R^n$ and for all $\alpha\in\N^n$ there exist $N\in\N$ and $c>0$ such that
\[
\sup_{x\in K}|\partial^\alpha f_\eps(x)|\le c\eps^{-N},
\]
for all $\eps\in(0,1]$. 
\item[(ii)] A net of functions $(f_\eps)_\eps\in \gamma^s(\R^n)^{(0,1]}$ is $\gamma^s$-moderate if for all $K\Subset\R^n$ there exists a constant $c_K>0$ and there exists $N\in\N$ such that 
\[
|\partial^\alpha f_\eps(x)|\le c_K^{|\alpha|+1}(\alpha !)^s \eps^{-N-|\alpha|},
\]
for all $\alpha\in\N^n$, $x\in K$ and $\eps\in(0,1]$.
\item[(iii)] A net of functions $(f_\eps)_\eps\in C^\infty([0,T];\gamma^s(\R^n))^{(0,1]}$ is $C^\infty([0,T];\gamma^s(\R^n))$-mo\-de\-rate 
if for all $K\Subset\R^n$ there exist $N\in\N$, $c>0$ and, for all $k\in\N$ there exist $N_k \in \N$ and $c_k>0$ such that
\[
|\partial_t^k\partial^\alpha_x u_\eps(t,x)|\le c_k\eps^{-N_k} c^{|\alpha|+1}(\alpha !)^s \eps^{-N-|\alpha|},
\]
for all $\alpha\in\N^n$, for all $t\in[0,T]$, $x\in K$ and $\eps\in(0,1]$.
\end{itemize}
\end{definition}
\begin{remark}
\label{rem_our_nets}
Note that the definition of $C^\infty$-moderateness above is natural in the sense that regularisations of distributions are moderate. Indeed, one can prove (see Proposition 2.1 in \cite{G:20} and references therein) that if $u\in\mathcal{E}'(\R^n)$ then there exists $N\in\N$ and for all $\alpha\in\N^n$ there exists $c>0$ such that 
\[
|\partial^\alpha(u\ast\psi_{\omega(\eps)})(x)|\le c\,\omega(\eps)^{-N-|\alpha|},
\]
for all $x\in\R^n$ and $\eps\in(0,1]$. Since $\omega(\eps)$ tends to $0$ as $\eps\to 0$ it is not restrictive to assume that it is bounded. If there exists $c_1,c_2>0$ and $r>0$ such that 
\beq
\label{pos_net}
c_2\eps^r\le \omega(\eps)\le c_1,
\eeq
for all $\eps\in(0,1]$ then $\omega(\eps)$ can be replaced with $\eps$ in the estimates above and we have that the net $(u\ast \psi_{\omega(\eps)})_\eps$ is moderate in the sense of Definition \ref{def_mod_intro}(i).
\end{remark}
In the sequel we will work with \emph{positive nets}, i.e. nets $\omega(\eps)\to 0$ such that \eqref{pos_net} holds. 

For the kind of Cauchy problems we want to deal with (hyperbolic problems with multiplicities), well-posedness is classically obtained in spaces of Gevrey type. This means that we will use $C^\infty$-moderate nets at the level of coefficients but Gevrey-moderate nets at the level of initial data and solutions. More precisely, we introduce the following notion of a `very weak solution' for the Cauchy problem \eqref{intro_eq_Dt}. For simplicity, we will also call $(u_\eps)_\eps$ a Gevrey very weak solution (of order $s$), or simply a very weak solution.

\begin{definition}
\label{def_vws}
Let $s\ge1$. The net $(u_\eps)_\eps\in C^\infty([0,T];\gamma^s(\R^n))$ is a {\bf very weak solution} of order $s$ of the 
Cauchy problem \eqref{intro_eq_Dt} if there exist 
\begin{itemize}
\item[(i)] $C^\infty$-moderate regularisations $a_{i,\eps}$, $c_{i,\eps}$, $d_{\eps}$, $e_{\eps}$ of the coefficients $a_i$,  $c_i$, $d$, $e$, respectively, for $i=1,\,\dots,\,n$, 
\item[(ii)] $C^\infty([0,T];\gamma^s(\R^n))$-moderate regularisation $(f_\eps)_\eps$ of the right-hand side $f$,
\item[(iii)] $\gamma^s$-moderate regularisations $g_{0,\eps}$ and $g_{1,\eps}$ of the initial data $g_0$ and $g_1$, 
respectively,
\end{itemize}
such that $(u_\eps)_\eps$ solves the regularised problem
\[
\begin{split}
\partial_t^2u_{\eps}(t,x)-\sum_{i=1}^n a_{i,\eps}(t)
\partial_{x_i}^2u_{\eps}(t,x)+l_{\eps}(t,\partial_t,\partial_x)u_{\eps}(t,x)&=f_{\eps}(t,x),\\
u_{\eps}(0,x)&=g_{0,\eps}(x),\\
\partial_t u_{\eps}(0,x)&=g_{1,\eps}(x).
\end{split}
\]

where $t\in[0,T]$, $x\in\R^n$ and 
\[
l_\eps(t,\partial_t,\partial_x)=\sum_{i=1}^n c_{i,\eps}(t)\partial_{x_i}+d_{\eps}(t)\partial_t+e_{\eps}(t).
\]
for all $\eps\in(0,1]$, and is $C^\infty([0,T];\gamma^s(\R^n))$-moderate.
\end{definition}
In this paper we will make use of three different types of mollifiers and corresponding nets.
\begin{enumerate}
\item[(1)] \label{mollifier1} Classical Friedrichs mollifier: $\psi\in C^\infty_c(\R^n)$ with $\int\psi=1$. It can be chosen positive if needed. We set
\[
\psi_\eps(x)=\frac{1}{\eps^n}\psi(\frac{x}{\eps}).
\]
\item[(2)]\label{mollifier2} Mollifier with all the moments vanishing: $\varphi\in\S(\R^n)$ with $\int\varphi(x)\, dx=1$ and $\int x^\alpha\varphi(x)\, dx=0$ for all $\alpha\neq 0$. As above
\[
\varphi_\eps(x)=\frac{1}{\eps^n}\varphi(\frac{x}{\eps}).
\]
\item[(3)]\label{mollifier3} Mollifier of order $\sigma>1$: Let $\phi$ be a mollifier in the Gelfand-Shilov space $\mathcal{S}^{(\sigma)}(\R)$ with all the moments vanishing (see e.g. \cite[Chapter 6]{NiRo:10} and \cite{T06}). Let $\chi\in \gamma^\sigma(\R)$ with $0\le\chi\le 1$, $\chi(x)=0$ for $|x|\ge 2$ and $\chi(x)=1$ for $|x|\le 2$. Hence
\[
\rho_\eps(x)\coloneqq\eps^{-1}\phi\biggl(\frac{x}{\eps}\biggr)\chi(x|\ln\eps|)
\]
is a net of Gevrey functions of order $\sigma$.
\end{enumerate}
For more details about the construction of a mollifier of order $\sigma$ we refer the reader to \cite{BenBou:09, GR:14}.

\subsection{Negligible nets and characterisation via Fourier transform}
We can now introduce the notion of a negligible net and show how moderate and negligible nets can be characterised at the level of the Fourier transform. All the results mentioned in the sequel have been proven in \cite{GR:14} so we recall only the statements relevant to this paper.
\begin{definition}
\label{def_neg_net}
\leavevmode
\begin{itemize}
\item[(i)] A net of functions $(f_\eps)_\eps\in C^\infty(\R^n)^{(0,1]}$ is $C^\infty$-negligible if for all $K\Subset\R^n$, for all $\alpha\in\N^n$ and $q\in \N$ there exist $c>0$ such that
\[
\sup_{x\in K}|\partial^\alpha f_\eps(x)|\le c\eps^{q},
\]
for all $\eps\in(0,1]$. 
\item[(ii)] A net of functions $(f_\eps)_\eps\in \gamma^s(\R^n)^{(0,1]}$ is $\gamma^s$-negligible if for all $K\Subset\R^n$ there exists a constant $c=c_K>0$ and for all $q\in\N$ a constant $c_q>0$ such that
\[
|\partial^\alpha f_\eps(x)|\le c_q c^{|\alpha|+1}(\alpha!)^\sigma \eps^q,
\]
for all $\alpha\in\N^n$, $x\in K$ and $\eps\in(0,1]$.
\item[(iii)] A net of functions $(f_\eps)_\eps\in C^\infty([0,T];\gamma^s(\R^n))^{(0,1]}$ is $C^\infty([0,T];\gamma^s(\R^n))$-neg\-li\-gi\-ble if for all $K\Subset\R^n$ there exists a constant $c=c_K>0$ and for all $q\in\N$ a constant $c_q>0$ such that
\[
|\partial_t^k\partial^\alpha_x u_\eps(t,x)|\le c_qc^{|\alpha|+1}(\alpha !)^s \eps^{q},
\]
for all $\alpha\in\N^n$, for all $t\in[0,T]$, $x\in K$ and $\eps\in(0,1]$.
\end{itemize}
\end{definition}
Taking into consideration Definition \ref{def_mod_intro} and Definition \ref{def_neg_net} we now pass to analyse three different types of nets: 
\begin{itemize}
\item $u\ast\varphi_\eps$ where $u\in\gamma^\sigma_c(\R^n)$ and $\varphi_\eps$ is a mollifier of type (2);
\item $u\ast\rho_\eps$ where $u\in C^\infty_c(\R^n)$ and $\rho_\eps$ is a mollifier of type (3);
\item $u\ast\rho_\eps$ where $u\in \E'(\R^n)$ and $\rho_\eps$ is a mollifier of type (3).
\end{itemize}

\begin{proposition}
\label{prop_reg_gevrey_sim}[Proposition 4.1 in \cite{GR:14}]
Let $\sigma> 1$. Let $u\in \gamma_c^\sigma(\R^n)$ and let $\varphi$ be a mollifier of type (2). Then
\begin{itemize}
\item[(i)] there exists $c>0$ such that 
\[ 	
|\partial^\alpha(u\ast\varphi_\eps)(x)|\le c^{|\alpha|+1}(\alpha!)^\sigma
\]
for all $\alpha\in\N^n$, $x\in\R^n$ and $\eps\in(0,1]$;
\item[(ii)] there exists $c>0$ and for all $q\in\N$ a constant $c_q>0$ such that
\[
|\partial^\alpha(u\ast\varphi_\eps-u)(x)|\le c_q c^{|\alpha|+1}(\alpha!)^\sigma \eps^q,
\]
for all $\alpha\in\N^n$, $x\in\R^n$ and $\eps\in(0,1]$;
\item[(iii)] there exist $c,\, c'>0$ such that 
\[
|\widehat {u\ast\varphi_\eps}(\xi)|\le c'\,\esp^{-c\lara{\xi}^{\frac{1}{\sigma}}},
\]
for all $\xi\in\R^n$ and $\eps\in(0,1]$.
\end{itemize}
\end{proposition}
Clearly, (i) and (ii) show that the corresponding nets are $\gamma^\sigma$-moderate and $\gamma^\sigma$-negligible. The following proposition provides a Fourier characterisation of Gevrey-moderate and Gevrey-negligible nets.
\begin{proposition}
\label{prop_reg_gevrey_mod}[Proposition 4.3 in \cite{GR:14}]
\leavevmode
\begin{itemize}
\item[(i)] If $(u_\eps)_\eps$ is $\gamma^\sigma$-moderate and there exists $K\Subset\R^n$ such that $\supp\, u_\eps\subseteq K$ for all $\eps\in(0,1]$ then there exist $c,\,c'>0$ and $N\in\N$ such that
\beq
\label{star1}
|\widehat{u_\eps}(\xi)|\le c'\eps^{-N}\esp^{-c\eps^{\frac{1}{\sigma}}\lara{\xi}^{\frac{1}{\sigma}}},
\eeq
for all $\xi\in\R^n$ and $\eps\in(0,1]$.
\item[(ii)] If $(u_\eps)_\eps$ is $\gamma^\sigma$-negligible and there exists $K\Subset\R^n$ such that $\supp\, u_\eps\subseteq K$ for all $\eps\in(0,1]$ then there exists $c>0$ and for all $q>0$ there exists $c_q>0$  such that
\beq
\label{star2}
|\widehat{u_\eps}(\xi)|\le c_q\eps^{q}\esp^{-c\eps^{\frac{1}{\sigma}}\lara{\xi}^{\frac{1}{\sigma}}},
\eeq
for all $\xi\in\R^n$ and $\eps\in(0,1]$.

\item[(iii)] If $(u_\eps)_\eps$ is a net of tempered distributions with $(\widehat{u_\eps})_\eps$ satisfying \eqref{star1} then $(u_\eps)_\eps$ is $\gamma^s$-moderate.
\item[(iv)] If $(u_\eps)_\eps$ is a net of tempered distributions with $(\widehat{u_\eps})_\eps$ satisfying \eqref{star2} then $(u_\eps)_\eps$ is $\gamma^s$-negligible.
\end{itemize}
\end{proposition}
We now pass to consider $u\ast\rho_\eps$ where $u\in C^\infty_c(\R^n)$. We recall a statement proved in Propositions 5.1 and 5.2 in \cite{GR:14}. Note that the following estimates are valid for $\eps$ small enough, i.e., for all $\eps\in(0,\eta]$ with $\eta\in(0,1]$. Without loss of generality we can assume $\eta=1$.
\begin{proposition}
\label{prop_reg_smooth}
Let $u\in C^\infty_c(\R^n)$ and $\rho_\eps$ be a mollifier of type (3) with $\sigma>1$. Then, there exists $K\Subset\R^n$ such that $\supp(u\ast\rho_\eps)\subseteq K$ for all $\eps$ small enough and 
\begin{itemize}
\item[(i)] there exists $c>0$ and $\eta\in(0,1]$ such that
\[
|\partial^\alpha(u\ast\rho_\eps)(x)|\le c^{|\alpha|+1} (\alpha !)^\sigma \eps^{-|\alpha|}
\]
for all $\alpha\in\N^n$, $x\in\R^n$ and $\eps\in(0,\eta]$, or in other words, $(u\ast\rho_\eps)_\eps$ is $\gamma^\sigma_c$-moderate.
\item[(ii)] The net $(u\ast\rho_\eps-u)_\eps$ is compactly supported uniformly in $\eps$ and $C^\infty$-negligible.
\item[(iii)] There exist $c,\,c'>0$ and $\eta\in(0,1]$ such that
\[
|\widehat{u_\eps}(\xi)|\le c'\,\esp^{-c\,\eps^{\frac{1}{\sigma}}\lara{\xi}^{\frac{1}{\sigma}}},
\]
for all $\xi\in\R^n$ and $\eps\in(0,\eta]$.
\end{itemize}
\end{proposition}
Finally, we recall a statement proved in Propositions 6.1 in \cite{GR:14}.
\begin{proposition}
\label{prop_reg_distr}[Proposition 6.1 in \cite{GR:14}]
Let $u\in\E'(\R^n)$ and $\rho_\eps$ be a mollifier of type (3) and order $\sigma>1$. Then, there exists $K\Subset\R^n$ such that $\supp(u\ast\rho_\eps)\subseteq K$ for all $\eps$ small enough and there exist $C>0$, $N\in\N$ and $\eta\in(0,1]$ such that
\[
|\partial^\alpha(u\ast\rho_\eps)(x)|\le C^{|\alpha|+1} (\alpha !)^\sigma \eps^{-|\alpha|-N}
\]
for all $\alpha\in\N^n$, $x\in\R^n$ and $\eps\in(0,\eta]$.
\end{proposition}

From this we can conclude that the net $(u\ast\rho_\eps)_\eps$ is $\gamma^\sigma_c$-moderate and therefore from Proposition \ref{prop_reg_gevrey_mod}(i) we have that there exists $c>0$ and $N\in\N$ such that
\[
|\widehat{u\ast\rho_\eps}(\xi)|\le c\eps^{-N}\esp^{-c\eps^{\frac{1}{\sigma}}\lara{\xi}^{\frac{1}{\sigma}}},
\]
for all $\xi\in\R^n$ and $\eps$ small enough.
\begin{remark}
Note that in the previous results one could replace $\eps$ with a positive net $\omega(\eps)$. In addition, if  $u\in\E'(\R\times\R^n)$ with $\supp u\subseteq [0,T]\times K$, $K\Subset\R^n$ then by convolution with $\psi_\eps$ and $\rho_\eps$ one can obtain the estimates above formulated for $(u\ast\psi_\eps\rho_\eps)(t,x)$ uniformly valid for $t\in[0,T]$ and $x\in\R^n$. Note that the net
\[
 (u\ast\psi_\eps\rho_\eps)(t,x)=u_{s,y}(\psi_\eps(t-s)\rho_\eps(x-y))
 \]
 is compactly supported uniformly with respect the the parameter $\eps$ when chosen small enough. More precisely, by the structure theorem for distributions, we have that there exist $c>0$, $N_1, N_2\in\N$ and $\eta\in(0,1]$ and for all $k\in \N$ there exists $c_k>0$ such that
\[
|\partial^k_t\partial^\alpha_x(u\ast\psi_\eps\rho_\eps)(t,x)|\le c_k\eps^{-N_1-k}c^{|\alpha|+1} (\alpha !)^\sigma \eps^{-N_2-|\alpha|}
\]
for all $\alpha\in\N^n$, $x\in\R^n$, $t\in[0,T]$ and $\eps\in(0,\eta]$. In other words we get a $C^\infty([0,T], \gamma^\sigma(\R^n))$-moderate net. This is due to the fact that by the structure theorem for distributions, we have that $\exists g \in C(\R \times \R^n)$ with $\supp g\subseteq K_1\times K_2\Subset \R\times\R^n$, $N_1\in\N$, $\beta\in\N^n$ such that
\begin{align*}
|\partial^k_t\partial^\alpha_x(u\ast\psi_\eps\rho_\eps)(t,x)| &=|\partial^k_t\partial^\alpha_x(\partial^{N_1}_t\partial^\beta_x g \ast\psi_\eps\rho_\eps)(t,x)|= |(g \ast\partial^{N_1+k}_t \psi_\eps \partial^{\beta+\alpha}_x\rho_\eps)(t,x)| \\
& = \left|\int_{\R^n} \int_{\R} g(\tau,{x}) \partial^{N_1+k}_t \psi_\eps (t-\tau) \partial^{\beta+\alpha}_x\rho_\eps (x-y) d\tau dy \right|\\
& \le \int_{K_2} \int_{K_1} |g(\tau,y)| |\partial^{N_1+k}_t \psi_\eps (t-\tau)| d\tau |\partial^{\beta+\alpha}_x\rho_\eps (x-y)| dy \\
& \le c_k\eps^{-N_1-k} \int_{K_2}  \Vert g(\cdot,y)\Vert_{L^{\infty}(K_1)} |\partial^{\beta+\alpha}_x\rho_\eps (x-y)| dy.
\end{align*}
Proceeding now as in the proof of Proposition 6.1 in \cite{GR:14} we get the desired estimate with $N_2=|\beta|$ and $\eta$ sufficiently small.

\end{remark}
\subsection{Structure of the proof: what is known and what is unknown}
Let us now consider the regularised Cauchy problem 
\beq
\label{CP_new_reg}
\begin{split}
D_t^2u_{\eps}-\sum_{i=1}^n b_{i,\eps}(t)D_t D_{x_i}u_{\eps}-\sum_{i=1}^n a_{i,\eps}(t)
D_{x_i}^2u_{\eps}-l_{\eps}(t,{\rm i}D_t,{\rm i}D_x)u_{\eps}&=-f_{\eps}(t,x),\\
u_{\eps}(0,x)&=g_{0,\eps}(x),\\
D_t u_{\eps}(0,x)&=-{\rm i}g_{1,\eps}(x).
\end{split}
\eeq
where $t\in[0,T]$, $x\in\R^n$ and 
\[
l_{\eps}(t,{\rm i}D_t,{\rm i}D_x)=\sum_{i=1}^n c_{i,\eps}(t){\rm i}D_{x_i}+d_{\eps}(t){\rm i}D_t+e_{\eps}(t)
\]
and compare it with the one investigated in \cite{GR:14}, i.e., 
\beq
\label{CP_ARMA}
\begin{split}
D_t^2u_{\eps}-\sum_{i=1}^n b_{i,\eps}(t)D_t D_{x_i}u_{\eps}-\sum_{i=1}^n a_{i,\eps}(t)
D_{x_i}^2u_{\eps} &=0,\\
u_{\eps}(0,x)&=g_{0,\eps}(x),\\
D_t u_{\eps}(0,x)&=-{\rm i}g_{1,\eps}(x).
\end{split}
\eeq
Let the coefficients $a_i$, $b_i$  be distributions with compact support included in $[0,T]$, such that $a_i$, $b_i$ are real-valued and $a_i\ge 0$ for all $i=1,\,\dots,\,n$. Analysing the proof of Theorem \ref{theo_vws_ARMA} in \cite{GR:14} for the homogeneous Cauchy problem \eqref{CP_ARMA} we see that a very weak solution $(u_\eps)_\eps$ of order $s$ exists if we have

\begin{itemize}
\item coefficients $a_{i,\eps}=a_i\ast\psi_{\omega(\eps)}$, $b_{i,\eps}=b_i\ast\psi_{\omega(\eps)}$ with $\psi\ge 0$ mollifier in $C^\infty_c(\R)$ of type (1) and 
\[
\omega^{-1}(\eps)=c(\ln(\eps^{-1}))^r,
\]
for some constants $c,\,r>0$ and initial data $g_0,\,g_1\in \gamma_c^s(\R^n)$ with $g_{i,\eps}=g_i\ast\varphi_\eps$\footnote{The regularisation of the initial data is not essential in this case because of the regularity of the initial data but it is necessary as soon as the initial data are less than Gevrey-regular.}, $i=0,\,1$, with $\varphi\in\S(\R^n)$ mollifier of type (2).
\end{itemize}

 
The existence of a $C^\infty([0,T];\gamma^s(\R^n))$-moderate net of solutions $(u_\eps)_\eps$ is proven in \cite{GR:14} by working on the homogeneous system 
\begin{equation*}
\label{system_new_V_ARMA}
\begin{split}
D_t V_\eps&=\mathbb{A}_{1,\eps}(t,\xi)V_\eps,\\
V_\eps|_{t=0}(\xi)&=V_{0,\eps}(\xi),
\end{split}
\end{equation*}
where
\[
V_\eps=\left(
\begin{array}{c}
\widehat{u_{1,\eps}}(t,\xi) \\
 \widehat{u_{2,\eps}}(t,\xi) 
 \end{array}
 \right)
 =\left(
\begin{array}{c}
\lara{\xi}\widehat{u_\eps}(t,\xi) \\
 \widehat{D_tu}_\eps(t,\xi) 
 \end{array}
 \right),
\]
\[
\mathbb{A}_{1,\eps}(t,\xi)= \left(
    \begin{array}{ccccc}
      0 & 1\\
      \sum_{i=1}^n a_{i,\eps}(t){\xi_i}^2\lara{\xi}^{-2} & \sum_{i=1}^n b_{i,\eps}(t)\xi_i\lara{\xi}^{-1} \\
    \end{array}
  \right)\lara{\xi},
\]
and
\[
V_{0,\eps}=\left(
\begin{array}{c}
 \lara{\xi}\widehat{g_{0,\eps}}\\
 -{\rm i}\widehat{g_{1,\eps}}
 \end{array}
 \right)
 \]
and defining the energy 
\[
E_{\delta,\eps}(V)= (Q^{(2)}_{\delta,\eps}(t,\xi) V, V),
\]
by means of the quasi-symmetriser
\begin{align*}
Q^{(2)}_{\delta,\eps}(\tilde{\lambda}_{1,\eps},\tilde{\lambda}_{2,\eps})=&\left(
    \begin{array}{cc}
     \big(\sum_{i=1}^n b_{i,\eps}\xi_i\big)^2\lara{\xi}^{-2}+2\sum_{i=1}^n a_{i,\eps}\xi_i^2\lara{\xi}^{-2} & -\sum_{i=1}^n b_{i,\eps}\xi_i\lara{\xi}^{-1}\\[0.2cm]
      -\sum_{i=1}^n b_{i,\eps}\xi_i\lara{\xi}^{-1} & 2 \\
           \end{array}
  \right)\\
   &+ 2\delta^2 \left(
    \begin{array}{cc}
      1 & 0\\
      0 & 0 \\
           \end{array}
  \right).
\end{align*}
Note that in this case the quasi-symmetriser is depending on two parameters: the regularising parameter $\eps$ and the standard parameter $\delta$.

Now, when we pass to the inhomogeneous Cauchy problem \eqref{CP_new_reg} we end up with the system
\beq
\label{system_inhom_reg}
\begin{split}
D_t V_\eps&=\mathbb{A}_{1,\eps}(t,\xi)V_\eps+B_\eps(t,\xi)V_\eps+\widehat{F_\eps}, \nonumber\\
V_\eps|_{t=0}(\xi)&=V_{0,\eps}(\xi),
\end{split}
\eeq
where   
 \[
B_\eps(t,\xi)=\left(
    \begin{array}{ccccc}
      0 & 0 \\
      \left({\rm i}\sum_{i=1}^n c_{i,\eps}(t)\xi_i +e_{\eps}(t) \right)\lara{\xi}^{-1} & {\rm i}d_{\eps}(t) \\
    \end{array}
  \right) \quad \text{and} \quad 
  \widehat{F_\eps}=\left(
\begin{array}{c}
0\\
 -\widehat{f_\eps}(t,\xi)
 \end{array}
 \right).
\]

The energy is still defined by the quasi-symmetriser $Q^{(2)}_{\delta,\eps}$ as in the homogenous case but differently from \cite{GR:14} we need to handle the matrix of the lower order terms $B_\eps$ and the right-hand side $F_\eps$. This will require
\begin{itemize}
\item the formulation of suitable Levi conditions on $B_{\eps}$ extending \eqref{EQ:lot2_general} to the case of singular coefficients;
\item the inclusion of right-hand side $F_\eps$ in the energy estimates as in Theorem \ref{theo_GRi}.
\end{itemize}
We will achieve our purpose by suitably combining ideas and techniques initiated in \cite{GR:12} (Levi conditions), \cite{GR:14} (singular coefficients) and Theorem \ref{theo_GRi} (right-hand side).

 \subsection{Hypotheses and different cases}
We work under the assumptions that 
\begin{itemize}
\item[(H)] the equation coefficients $a_i$,  $c_i$, $d$, $e$ are distributions with compact support contained in $[0,T]$, $a_i$ are real-valued and $a_i\ge 0$ for all $i=1,\,\dots,\,n$. 
\end{itemize}
Our analysis will distinguish between two cases defined as follows:
\begin{itemize}
\item[{\bf Case 1}] $f\in C^\infty([0,T], \gamma^s_c(\R^n))$ and $g_0,\, g_1\in\gamma_c^s(\R^n)$.
\item[{\bf Case 2}] $f\in \E'(\R\times\R^n)$ with $\supp f\Subset[0,T]\times\R^n$ and  $g_0, \,g_1\in\gamma_c^s(\R^n)$.
\end{itemize}

\section{How to formulate the Levi conditions: motivating example}
In this section we focus on some simple example in order to deduce how to formulate the Levi conditions on the lower order terms. For the sake of simplicity, and without loss of generality, we set the right-hand side equal to $0$. Let us consider the equation 
\[
D_t^2u(t,x)-a(t)D_x^2u(t,x)+c(t)D_x u+d(t)D_tu+e(t)u=0,\qquad x\in\R, t\in[0,T]
\]
where $a(t)=\mu(t)H(t-t_0)$, $t_0\in(0,T)$ and $\mu$ is a positive cut-off function with support contained in $[0,T]$ and identically equal to $1$ around $t_0$. We assume that $c$, $d$, $e$ are compactly supported distributions with support contained in $[0,T]$ as well. In addition, we assume the initial conditions $u|_{t=0}=g_0\in\gamma^s_c(\R)$, $D_t u|_{t=0}=g_1\in\gamma^s_c(\R)$ with $\supp g_i\Subset[0,T]$ for $i=0,\,1$.

Let $\psi\ge 0$ be a mollifier of type (1) and let $\omega(\eps)$ be a positive net. Regularising by convolution we obtain the coefficients
\[
a_\eps= a\ast \psi_{\omega(\eps)},\, c_\eps= c\ast \psi_{\omega(\eps)},\, d_\eps= d\ast \psi_{\omega(\eps)},\, e_\eps= e\ast \psi_{\omega(\eps)}
\]
and the equation 
\[
D_t^2u_{\eps}-a_{\eps}(t)D_x^2u_{\eps}+c_{\eps}(t)D_x u_{\eps}+d_{\eps}(t)D_tu_{\eps}+e_{\eps}(t)u_{\eps}=0,\qquad x\in\R, t\in[0,T],
\]
with initial conditions $u_\eps(0)=g_{0,\eps}=g_0\ast \varphi_\eps$ and $D_t u_\eps(0)=g_{1,\eps}=g_1\ast\varphi_\eps$. Note that $\varphi$ is a mollifier or type (2) so that the nets $(g_0-g_{0,\eps})_\eps$ and $(g_1-g_{1,\eps})_\eps$ are $\gamma^s$-negligible.
%

For each $\eps>0$, the Levi conditions \eqref{EQ:lot2}, reformulated as in \eqref{EQ:lot2_general}, correspond to $\exists C_{1,\eps},\,C_{2,\eps}>0$ such that
\begin{align}\label{EQ:lot2_toy}
\left|c_{\eps}(t)\xi+e_{\eps}(t) \right|^2 &\le 2C_{1,\eps}a_{\eps}(t)|\xi|^2 , \\
|d_{\eps}(t)|^2&\le C_{2,\eps},\nonumber
\end{align}
for $t\in [0,T]$ and $\xi$ away from $0$ (i.e., for $|\xi|\ge R$ for some $R>0$ independent of $\eps$).

%

We now argue as in the proof of Theorem \ref{theo_GRi} and we reduce  the second order equation to a first order system of pseudo-differential operators by setting 
\[
u_{1,\eps}=\lara{D_x}u_{\eps}, \quad u_{2,\eps}=D_t u_{\eps}.
\]
We get the  system
\beq
\label{syst_Taylor}
D_t\left(
                             \begin{array}{c}
                               u_{1,\eps} \\                          
                               u_{2,\eps} \\
                             \end{array}
                           \right)
= \left(
    \begin{array}{ccccc}
      0 & \lara{D_x} \\
      \left(a_{\eps}D_x^2 - c_{\eps}D_x - e_{\eps} \right)\lara{D_x}^{-1} & -d_{\eps} \\
    \end{array}
  \right)
  \left(\begin{array}{c}
                               u_{1,\eps} \\                             
                               u_{2,\eps} \\
                             \end{array}
                           \right), 
                           \eeq
where the matrix above can be written as $\mathbb{A}_{1,\eps}+B_{\eps}$ with
\[
\mathbb{A}_{1,\eps}=\left(
    \begin{array}{ccccc}
      0 & \lara{D_x} \\
      a_{\eps}D_x^2\lara{D_x}^{-1} & 0 \\
    \end{array}
  \right), \quad 
B_{\eps}=\left(
    \begin{array}{ccccc}
      0 & 0 \\
      -\left( c_{\eps}D_x +e_{\eps} \right)\lara{D_x}^{-1} & -d_{\eps} \\
    \end{array}
  \right).  
\]

By Fourier transforming both sides of \eqref{syst_Taylor} in $x$, we obtain 
\beq
\label{system_new}
\begin{split}
D_t V_\eps&=\mathbb{A}_{1,\eps}(t,\xi)V_\eps+B_{\eps}(t,\xi)V_\eps,\\
V_\eps|_{t=0}(\xi)&=V_{0,\eps}(\xi),
\end{split}
\eeq
where $V_\eps$ is the $2$-column with entries $v_{j,\eps}=\widehat{u}_{j,\eps}$, $V_{0,\eps}$ is the $2$-column vector \\ $(\lara{\xi}\widehat{g_{0,\eps}}, \widehat{g_{1,\eps}})^T$ and 
\begin{multline*}
\label{mA}
\mathbb{A}_{1,\eps}(t,\xi)=\left(
    \begin{array}{ccccc}
      0 & \lara{\xi} \\
      a_{\eps}\xi^2\lara{\xi}^{-1} & 0 \\
    \end{array}
  \right), \quad 
B_{\eps}(t,\xi)=\left(
    \begin{array}{ccccc}
      0 & 0 \\
      -\left( c_{\eps}\xi +e_{\eps} \right)\lara{\xi}^{-1} & -d_{\eps} \\
    \end{array}
  \right).
\end{multline*}
 
Henceforth, we will focus on the system \eqref{system_new} and on the matrix
 $$A_{\eps}(t,\xi)\coloneqq\lara{\xi}^{-1}\mathbb{A}_{1,\eps}(t,\xi)$$
 for which we will construct a quasi-symmetriser. Note that the eigenvalues of the matrix $\mathbb{A}_{1,\eps}$
 are exactly the roots $\lambda_{1,\eps}(t,\xi),\,\lambda_{2,\eps}(t,\xi)$.
Furthermore, the Kinoshita-Spagnolo condition \eqref{KS_condition} holds for the eigenvalues
$\lara{\xi}^{-1}\lambda_{1,\eps}(t,\xi),\,\lara{\xi}^{-1}\lambda_{2,\eps}(t,\xi)$
of the $0$-order matrix $A_{\eps}(t,\xi)$ as well. We have that the eigenvalues of $A_{\eps}(t,\xi)$ are
\[
\tilde{\lambda}_{1,2,\eps}(t,\xi)\coloneqq \lara{\xi}^{-1}\lambda_{1,2,\eps}(t,\xi)=\pm \sqrt{a_{\eps}}|\xi| \lara{\xi}^{-1}
\]
and hence the quasi-symmetriser 
\[
Q^{(2)}_{\delta}(\tilde{\lambda}_{1,\eps},\tilde{\lambda}_{2,\eps})=\left(
    \begin{array}{cc}
      \tilde{\lambda}_{1,\eps}^2+\tilde{\lambda}_{2,\eps}^2 & -(\tilde{\lambda}_{1,\eps}+\tilde{\lambda}_{2,\eps})\\
      -(\tilde{\lambda}_{1,\eps}+\tilde{\lambda}_{2,\eps}) & 2 \\
           \end{array}
  \right) + 2\delta^2 \left(
    \begin{array}{cc}
      1 & 0\\
      0 & 0 \\
           \end{array}
  \right),
\]
becomes
\[
Q^{(2)}_{\delta}(\tilde{\lambda}_{1,\eps},\tilde{\lambda}_{2,\eps})=\left(
    \begin{array}{cc}
      2a_{\eps} \xi^2 \lara{\xi}^{-2} & 0\\
      0 & 2 \\
           \end{array}
  \right) + 2\delta^2 \left(
    \begin{array}{cc}
      1 & 0\\
      0 & 0 \\
           \end{array}
  \right).
\] 

This allows us to define the energy
\[
E_{\delta,\eps}(t,\xi)\coloneqq(Q^{(2)}_{\delta,\eps}(t,\xi) V_\eps(t,\xi), V_\eps(t,\xi)).
\] 

We hence have
\begin{align} \label{EQ:energy_deriv}
\partial_t E_{\delta,\eps}(t,\xi)=&(\partial_tQ^{(2)}_{\delta,\eps} V_\eps,V_\eps)+ (Q^{(2)}_{\delta,\eps} \partial_tV_\eps,V_\eps)+(Q^{(2)}_{\delta,\eps} V_\eps,\partial_tV_\eps)\nonumber \\ \nonumber
=&(\partial_tQ^{(2)}_{\delta,\eps} V_\eps,V_\eps)+ i(Q^{(2)}_{\delta,\eps} D_tV_\eps,V_\eps)-i(Q^{(2)}_{\delta,\eps} V_\eps,D_tV_\eps)\\ \nonumber
=&(\partial_tQ^{(2)}_{\delta,\eps} V_\eps,V_\eps)+i(Q^{(2)}_{\delta,\eps}(\lara{\xi}A_{\eps}+B_{\eps})V_\eps,V_\eps)-i(Q^{(2)}_{\delta,\eps} V_\eps,(\lara{\xi}A_{\eps}+B_{\eps})V_\eps)\\ 
=&(\partial_tQ^{(2)}_{\delta,\eps} V_\eps,V_\eps)+i\lara{\xi}((Q^{(2)}_{\delta,\eps} A_\eps-A_\eps^\ast Q^{(2)}_{\delta,\eps})V_\eps,V_\eps)\\
&+i((Q^{(2)}_{\delta,\eps} B_{\eps}-B_{\eps}^\ast Q^{(2)}_{\delta,\eps})V_\eps,V_\eps). \nonumber
\end{align} 
If for a moment we neglect the matrix of lower order terms $B_\eps$, we are dealing with the kind of Cauchy problem and energy estimates studied already in \ref{theo_vws_ARMA} in \cite{GR:14}. So we know that, if $B_\eps\equiv 0$ and the initial data $g_0,\,g_1\in \gamma_c^s(\R)$ with $\supp(g_i)\Subset[0,T]$, a very weak solution $(u_\eps)_\eps$ of order $s$ exists if we take 
\begin{itemize}
\item coefficients $a_{i,\eps}=a_i\ast\psi_{\omega(\eps)}$ with $\psi\ge 0$ mollifier in $C^\infty_c(\R^n)$ of type (1) and 
\[
\omega^{-1}(\eps)=c(\ln(\eps^{-1}))^{r},
\]
for some constants $c,\,r>0$.  
\end{itemize}
Our aim is 
\begin{itemize}
\item to understand which type of nets $C_{1,\eps}$ and $C_{2,\eps}$ are needed in the Levi conditions \eqref{EQ:lot2_toy} in order to get a moderate net $(u_\eps)_\eps$.
\end{itemize}
This is possible by analysing and estimating the term
\[
(Q^{(2)}_{\delta,\eps} B_{\eps}-B_{\eps}^\ast Q^{(2)}_{\delta,\eps})
\]
as in \cite{GR:12} Section 5. That is by making sure that, similarly to the term 
\[\lara{\xi}((Q^{(2)}_{\delta,\eps} A_\eps-A_\eps^\ast Q^{(2)}_{\delta,\eps})V_\eps,V_\eps),\]
it can be estimated by the energy $E_{\delta,\eps}$. 

We recall that, for arbitrary $V\in \C^2$, 
\begin{itemize}
\item 
$
((Q^{(2)}_{\delta,\eps} B_\eps-B_\eps^\ast Q^{(2)}_{\delta,\eps})V,V)=((Q_{0,\eps}^{(2)} B_\eps-B_\eps^\ast Q_{0,\eps}^{(2)})V,V),
$
where
\[
Q_{0,\eps}^{(2)}=\left(
    \begin{array}{cc}
      2a_{\eps}(t) \xi^2 \lara{\xi}^{-2} & 0\\
      0 & 2 \\
           \end{array}
  \right).
\]
\item Since by construction $(Q^{(2)}_{0,\eps} V,V)\le E_{\delta,\eps}$, we want to find a net $C_\eps>0$ such that 
\begin{equation}
\label{cBeps}
|((Q_{0,\eps}^{(2)} B_\eps-B_\eps^\ast Q_{0,\eps}^{(2)})V,V)|\le C_\eps (Q^{(2)}_{0,\eps} V,V)\le C_\eps E_{\delta,\eps},
\end{equation}
for all $V\in \C^2$, $t\in[0,T]$, $\xi\in\R^2$ and $\eps\in(0,1]$. 
\item Note that we can write
\begin{align*}
((Q_{0,\eps}^{(2)} B_\eps-B_\eps^\ast Q_{0,\eps}^{(2)})V,V)&=((\mathcal{W}_\eps B_\eps V,\mathcal{W}_\eps V)-(\mathcal{W}_\eps V,\mathcal{W}_\eps B_\eps V))\\
&=2i\Im (\mathcal{W}_\eps B_\eps V,\mathcal{W}_\eps V),
\end{align*}
where
\[
\mathcal{W}_{\eps}=\left(
    \begin{array}{cc}
      -\tilde{\lambda}_{2,\eps} & 1\\
      -\tilde{\lambda}_{1,\eps} & 1\\
      \end{array}
  \right)  
 =\left(
    \begin{array}{cc}
      \sqrt{a_{\eps}}|\xi| \lara{\xi}^{-1} & 1\\
      -\sqrt{a_{\eps}}|\xi| \lara{\xi}^{-1} & 1\\
      \end{array}
  \right).
\]
So, 
\[
|((Q_{0,\eps}^{(2)} B_\eps-B_\eps^\ast Q_{0,\eps}^{(2)})V,V)|\le 2|\mathcal{W}_\eps B_\eps V||\mathcal{W}_\eps V|.
\]
\item Since
\[
(Q^{(2)}_{0,\eps} V,V)=|\mathcal{W}_\eps V|^2
\]
we have that if
\beq
\label{cB2eps}
|\mathcal{W}_\eps B_\eps V|\le \frac{C_{\eps}}{2}|\mathcal{W}_\eps V|
\eeq
for some net $C_\eps>0$ independent of $t$, $\xi$ and $V$, then the condition \eqref{cBeps} will hold.
\end{itemize}
By direct computations we have
\begin{align*}
\mathcal{W}_{\eps}B_{\eps}V
  &=\left(
    \begin{array}{cc}
      +\sqrt{a_{\eps}}|\xi| \lara{\xi}^{-1} & 1\\
      -\sqrt{a_{\eps}}|\xi| \lara{\xi}^{-1} & 1\\
      \end{array}
  \right)
  \left(
    \begin{array}{cc}
      0 & 0\\
      -\left( c_{\eps}\xi +e_{\eps} \right)\lara{\xi}^{-1} & -d_{\eps} \\
           \end{array}
  \right) 
  \left(
    \begin{array}{c}
      V_{1}\\
      V_{2} \\
           \end{array}
  \right) \\
  &=\left(
    \begin{array}{c}
      -\left(c_{\eps}\xi +e_{\eps} \right)\lara{\xi}^{-1}V_{1}-d_{\eps}V_{2}\\
      -\left(c_{\eps}\xi +e_{\eps} \right)\lara{\xi}^{-1}V_{1}-d_{\eps}V_{2}\\
      \end{array}
  \right).
\end{align*}
Therefore,
\[
|\mathcal{W}_{\eps}B_{\eps}V|^2=2|\left(c_{\eps}\xi +e_{\eps} \right)\lara{\xi}^{-1}V_{1}+d_{\eps}V_{2}|^2.
\]
We also have that
\[
\mathcal{W}_{\eps}V
	=\left(
    \begin{array}{cc}
      +\sqrt{a_{\eps}}|\xi| \lara{\xi}^{-1} & 1\\
      -\sqrt{a_{\eps}}|\xi| \lara{\xi}^{-1} & 1\\
      \end{array}
  \right)
  \left(
    \begin{array}{c}
      V_{1}\\
      V_{2} \\
           \end{array}
  \right)
  =\left(
    \begin{array}{cc}
      \sqrt{a_{\eps}}|\xi| \lara{\xi}^{-1}V_{1}+V_{2} \\
      -\sqrt{a_{\eps}}|\xi| \lara{\xi}^{-1}V_{1}+V_{2}\\
      \end{array}
  \right)
\]
and hence
\[
|\mathcal{W}_{\eps}V|^2=|\sqrt{a_{\eps}}|\xi| \lara{\xi}^{-1}V_{1}+V_{2} |^2+ |-\sqrt{a_{\eps}}|\xi| \lara{\xi}^{-1}V_{1}+V_{2}|^2.
\] 
We can therefore rewrite \eqref{cB2eps} as  
\begin{align} \label{m21}
2|\left(c_{\eps}\xi +e_{\eps} \right)\lara{\xi}^{-1}V_{1}+d_{\eps}V_{2}|^2 \le \frac{C_{\eps}^2}{4}  
& \left( |\sqrt{a_{\eps}}|\xi| \lara{\xi}^{-1}V_{1}+V_{2} |^2 \right.\\
&\left. \,+ |-\sqrt{a_{\eps}}|\xi| \lara{\xi}^{-1}V_{1}+V_{2}|^2 \right) \nonumber
.
\end{align}

Using the inequality $|x-y|^2\le 2|x|^2+2|y|^2$ we have that
\begin{align*}
|\sqrt{a_{\eps}}|\xi| \lara{\xi}^{-1}V_{1}+V_{2} |^2+ |-\sqrt{a_{\eps}}|\xi| \lara{\xi}^{-1}V_{1}+V_{2}|^2 &\geq \frac{1}{2}|2\sqrt{a_{\eps}}|\xi| \lara{\xi}^{-1}V_1|^2 \\
&=2 a_{\eps}|\xi|^2 \lara{\xi}^{-2}|V_1|^2.
\end{align*}

On the other hand
\[
|\left(c_{\eps}\xi +e_{\eps} \right)\lara{\xi}^{-1}V_{1}+d_{\eps}V_{2}|^2 \leq 2 \left(|\left(c_{\eps}\xi +e_{\eps} \right)\lara{\xi}^{-1}|^2|V_1|^2+|d_{\eps}|^2|V_2|^2\right)
\]
and therefore it suffices to show that, under the Levi conditions \eqref{EQ:lot2_toy},
\[
\begin{split}
|c_{\eps}(t)\xi+e_{\eps}(t)|^2 &\le 2C_{1,\eps}  a_{\eps}(t)|\xi|^2,\\
|d_{\eps}(t)|^2&\le C_{2,\eps}, 
\end{split}
\]
there exists $C_\eps>0$ such that
\beq
\label{m22}
4(|\left(c_{\eps}\xi +e_{\eps} \right)\lara{\xi}^{-1}|^2|V_1|^2+|d_{\eps}|^2|V_2|^2 )\le \frac{C_{\eps}^2}{4}(2 a_{\eps}|\xi|^2 \lara{\xi}^{-2}|V_1|^2).
\eeq 
In order to prove that the inequality above holds, we argue as in \cite{GR:12} on two different areas of $\C^2$ that we will denote by Area 1 and Area 2. In detail:
 \[
 \begin{split}
 \text{Area 1}&=\{V\in C^2:\, |V_2|^2\le  2\gamma  a_{\eps}|\xi|^2 \lara{\xi}^{-2}|V_1|^2\}\\
 \text{Area 2}&=\{V\in C^2:\, |V_2|^2>  2\gamma  a_{\eps}|\xi|^2 \lara{\xi}^{-2}|V_1|^2\},
 \end{split}
 \]
 for $\gamma>0$.

{\bf Area 1.} Let $|V_2|^2\le  2\gamma  a_{\eps}|\xi|^2 \lara{\xi}^{-2}|V_1|^2$. Note that if
\beq
\label{B1B2}
4(|\left(c_{\eps}\xi +e_{\eps} \right)\lara{\xi}^{-1}|^2+\gamma a_{\eps}|\xi|^2 \lara{\xi}^{-2}|d_{\eps}|^2 )\le \frac{C_{\eps}^2}{4}(2 a_{\eps}|\xi|^2 \lara{\xi}^{-2})
\eeq
then \eqref{m22} holds. Using the Levi conditions \eqref{EQ:lot2_toy}, we have that \eqref{B1B2} holds independently of $\gamma>0$ with $\frac{C_{\eps}^2}{4}=4C_{1,\eps}+4\gamma C_{2,\eps}$.

{\bf Area 2.} Let  $|V_2|^2>  2\gamma  a_{\eps}|\xi|^2 \lara{\xi}^{-2}|V_1|^2$.

Using the Levi conditions \eqref{EQ:lot2_toy},
\begin{align*}
&4(|\left(c_{\eps}\xi +e_{\eps} \right)\lara{\xi}^{-1}|^2|V_1|^2+|d_{\eps}|^2|V_2|^2) \le 4(2C_{1,\eps}a_{\eps}|\xi|^2 \lara{\xi}^{-2}|V_1|^2+C_{2,\eps}|V_2|^2)\\
&\le 4(C_{2,\eps}+\frac{C_{1,\eps}}{2\gamma})|V_2|^2
\le 4\max(C_{1,\eps},C_{2,\eps})(1+\frac{1}{2\gamma})|V_2|^2.
\end{align*}
Since $|\sqrt{a_{\eps}}|\xi| \lara{\xi}^{-1}V_1|<1/\sqrt{2\gamma} |V_2|$ we have that
\begin{align*}
&\frac{C_{\eps}^2}{4}\left( |\sqrt{a_{\eps}}|\xi| \lara{\xi}^{-1}V_{1}+V_{2} |^2+ |-\sqrt{a_{\eps}}|\xi| \lara{\xi}^{-1}V_{1}+V_{2}|^2 \right) \\
&\ge \frac{C_{\eps}^2}{4}\left((|V_2|-|\sqrt{a_{\eps}}|\xi| \lara{\xi}^{-1} V_1|)^2+(|V_2|-|\sqrt{a_{\eps}}|\xi| \lara{\xi}^{-1}V_1|)^2\right)\\
&\ge 2\frac{C_{\eps}^2}{4}(1-\frac{1}{\sqrt{2\gamma}})^2|V_2|^2 \ge\frac{C_{\eps}^2}{4} |V_2|^2,
\end{align*}
for $\gamma$ big enough. Hence, \eqref{m21} holds  if $\gamma$ is big enough with $\frac{C_{\eps}^2}{4}\geq 4 \max(C_{1,\eps},C_{2,\eps})$. Note that we can choose $C_{\eps}^2$ as in the Area 1 and that $\gamma$ is independent of $\eps$ and it is a constant once it is fixed. We have therefore proven that under the Levi conditions  
\[
\begin{split}
|c_{\eps}(t)\xi+e_{\eps}(t)|^2 &\le 2C_{1,\eps}  a_{\eps}(t)|\xi|^2,\\
|d_{\eps}(t)|^2&\le C_{2,\eps}, 
\end{split}
\]
we have 
\beq
\label{est_B_toy}
|((Q_{\delta,\eps}^{(2)} B_\eps-B_\eps^\ast Q_{\delta,\eps}^{(2)})V,V)|\le C_\eps E_{\delta,\eps},
\eeq
with $C_{\eps}=\sqrt{16C_{1,\eps}+16\gamma C_{2,\eps}}$ for some $\gamma>0$ independent of $\eps$.

This shows that our Levi conditions are sufficient to get the energy estimates we want but it is still unclear which kind of nets $C_{1,\eps}$ and $C_{2,\eps}$ we need to get moderateness of the net of solutions $(u_\eps)_\eps$. For this reason, we need to go back to \eqref{EQ:energy_deriv}. We have
\begin{align*} \label{EQ:energy_deriv_2}
\partial_t E_{\delta,\eps}(t,\xi) \le& |(\partial_tQ^{(2)}_{\delta,\eps} V_{\eps},V_{\eps})|+|\lara{\xi}((Q^{(2)}_{\delta,\eps} A_\eps-A_\eps^\ast Q^{(2)}_{\delta,\eps})V_{\eps},V_{\eps})|\\ \nonumber
&+|((Q^{(2)}_{\delta,\eps} B_{\eps}-B_{\eps}^\ast Q^{(2)}_{\delta,\eps})V_{\eps},V_{\eps})|.
\end{align*}
By combining the arguments of Section 4.2 in \cite{GR:14} with \eqref{est_B_toy} we obtain
\begin{equation*}\label{est-rev}
\partial_t E_{\delta,\eps}(t,\xi) \le (K_{\delta,\eps}(t,\xi)+C_2\delta\lara{\xi}+C_{\eps})E_{\delta,\eps}(t,\xi) ,
\end{equation*}
where $K_{\delta,\eps}(t,\xi)$ has the property
\[
\int_0^T K_{\delta,\eps}(t,\xi)\, dt \le C_1\delta^{-\frac{2}{k}}\omega(\eps)^{-1},
\]
with $k\ge 2$ depending on $a$ and $C_1,\, C_2>0$ and $C_{\eps}$ as above.
By Gronwall's lemma we obtain 
\[
E_{\delta,\eps}(t,\xi)\le E_{\delta,\eps}(0,\xi)\esp^{C_1\delta^{-\frac{2}{k}}\omega(\eps)^{-1}+C_2T\delta \lara{\xi}+TC_{\eps}}. 
\]
As in \cite{GR:12}, we set $\delta^{-\frac{2}{k}}=\delta\lara{\xi}$. It follows that $\delta^{-\frac{2}{k}}=\lara{\xi}^{\frac{1}{\sigma}}$, where
\[
\sigma= 1+\frac{k}{2}.
\] 
It follows that for $\delta^{-\frac{2}{k}}=\delta\lara{\xi}$,
\beq
\label{En_toy}
\begin{split}
E_{\delta,\eps}(t,\xi)&\le E_{\delta,\eps}(0,\xi)\esp^{TC_\eps}\esp^{C_1\lara{\xi}^{\frac{1}{\sigma}}\omega(\eps)^{-1}}\esp^{C_2T\lara{\xi}^{\frac{1}{\sigma}}}\\
&\le E_{\delta,\eps}(0,\xi)\esp^{TC_\eps} \esp^{C_T\lara{\xi}^{\frac{1}{\sigma}}\omega(\eps)^{-1}},
\end{split}
\eeq
for some constant $C_T>0$. We recall that the energy $E_{\delta,\eps}$ above is defined as 
\[
(Q^{(2)}_\delta(\lambda_\eps) V_\eps,V_\eps)=(2a_{\eps} \xi^2 \lara{\xi}^{-2} +2\delta^2)|V_{1,\eps}|^2+2|V_{2,\eps}|^2.
\]
Since 
\[
0\le 2a_{\eps} \xi^2 \lara{\xi}^{-2}\le 2\Vert a_\eps\Vert_\infty\le 2\Vert \psi_{\omega(\eps)}\Vert_1 \Vert a\Vert_\infty\le 2
\]
we get the inequalities
\[
\begin{split}
E_{\delta,\eps}(t,\xi)&\ge 2\delta^2|V_{1,\eps}|^2+2|V_{2,\eps}|^2\ge 2\delta^2|V_\eps|^2,\\
E_{\delta,\eps}(t,\xi)&\le(2\Vert a_\eps\Vert_\infty+2)|V_{1,\eps}|^2+2|V_{2,\eps}|^2\le (2+2\Vert a_\eps\Vert_\infty)|V_\eps|^2 \le 4 |V_\eps|^2.
\end{split}
\]
Making use of these bounds in \eqref{En_toy} we can write, for $\delta^{-\frac{2}{k}}=\lara{\xi}^{\frac{1}{\sigma}}$,
\[
\begin{split}
2\delta^2|V_\eps|^2&\le E_{\delta,\eps}(0,\xi)\esp^{TC_\eps} \esp^{C_T\lara{\xi}^{\frac{1}{\sigma}}\omega(\eps)^{-1}},\\
2|V_\eps|^2&\le \lara{\xi}^{\frac{k}{\sigma}}E_{\delta,\eps}(0,\xi)\esp^{TC_\eps} \esp^{C_T\lara{\xi}^{\frac{1}{\sigma}}\omega(\eps)^{-1}},\\
|V_\eps|^2&\le 2 \lara{\xi}^{\frac{k}{\sigma}}|V_\eps(0)|^2\esp^{TC_\eps} \esp^{C_T\lara{\xi}^{\frac{1}{\sigma}}\omega(\eps)^{-1}}.
\end{split}
\]
Let us now focus on the last estimate for $|V_\eps|$:
\[
|V_\eps|^2\le 2 \lara{\xi}^{\frac{k}{\sigma}}|V_\eps(0)|^2\esp^{TC_\eps} \esp^{C_T\lara{\xi}^{\frac{1}{\sigma}}\omega(\eps)^{-1}}.
\]
In order to get moderate estimates for $V_\eps$ we need specific assumptions on both the nets $(C_\eps)_\eps$ and $(\omega(\eps)^{-1})_\eps$. In detail, arguing as in \cite{GR:14} Section 4 and by direct computations we have 
\[
\begin{split}
|V_\eps|^2&\le 2 \lara{\xi}^{\frac{k}{\sigma}}c_0\esp^{-c_0\lara{\xi}^{\frac{1}{s}}}\esp^{TC_\eps} \esp^{C_T\lara{\xi}^{\frac{1}{\sigma}}\omega(\eps)^{-1}}\\
&= 2 \lara{\xi}^{\frac{k}{\sigma}}c_0\esp^{-\frac{c_0}{2}\lara{\xi}^{\frac{1}{s}}}\esp^{-\frac{c_0}{2}\lara{\xi}^{\frac{1}{s}}}\esp^{TC_\eps} \esp^{C_T\lara{\xi}^{\frac{1}{\sigma}}\omega(\eps)^{-1}}.
\end{split}
\]
This leads, for some new constant $c>0$, to 
\[
|V_\eps|^2\le c\esp^{-c\lara{\xi}^{\frac{1}{s}}}\esp^{-\frac{c_0}{2}\lara{\xi}^{\frac{1}{s}}}\esp^{TC_\eps} \esp^{C_T\lara{\xi}^{\frac{1}{\sigma}}\omega(\eps)^{-1}}.
\]
If we want $\gamma^s$-moderate estimates we need to control the terms 
\beq
\label{term2c}
\esp^{TC_\eps} \nonumber
\eeq
and 
\beq
\label{term3c}
\biggl(\esp^{-\frac{c_0}{2}\lara{\xi}^{\frac{1}{s}}}\esp^{C_T\lara{\xi}^{\frac{1}{\sigma}}\omega(\eps)^{-1}}\biggr).
\eeq
This means to require that 
\[
C_\eps=O(\ln(\eps^{-1}))
\]
and to argue as in \cite{GR:14} (see (33) in Section 4.2) for \eqref{term3c}, i.e., 
\[
\omega^{-1}(\eps)=O((\ln(\eps^{-1}))^{r}),
\]
with 
\[
r=\frac{\frac{1}{s}-\frac{1}{\sigma}}{\frac{1}{s}}, \qquad 1<s<\sigma=1+\frac{k}{2},
\]
for a fixed $k\ge 2$. This allows us to prove that, for some $c'>0$ and $N\in\N$, 
\[
|V_\eps(t,\xi)|\le c'\eps^{-N}\esp^{-c'\lara{\xi}^{\frac{1}{s}}},
\]
uniformly in $t,\xi$ and $\eps$. Concluding, this example shows us that since $$C_{\eps}=\sqrt{16C_{1,\eps}+16\gamma C_{2,\eps}}=4\sqrt{C_{1,\eps}+\gamma C_{2,\eps}}$$ for some $\gamma>0$ independent of $\eps$, both the nets $C_{1,\eps}$ and $C_{2,\eps}$ appearing in the Levi conditions need to be $O((\ln(\eps^{-1}))^2)$ as $\eps\to 0$. In other words, we need Levi conditions \eqref{EQ:lot2_toy} of the type
\[
\begin{split}
|c_{\eps}(t)\xi+e_{\eps}(t)|^2 &\le c_1(\ln(\eps^{-1}))^2 a_{\eps}(t)|\xi|^2,\\
|d_{\eps}(t)|^2&\le c_2(\ln(\eps^{-1}))^2, 
\end{split}
\]
for $\eps\in(0,1]$, $t\in[0,T]$ and $\xi$ away from $0$.

For instance, these conditions are fulfilled by equations of the type
\[
D_t^2u_{\eps}-a_{\eps}(t)D_x^2u_{\eps}+d_{\eps}(t)D_tu_{\eps}=0,\qquad x\in\R, t\in[0,T],
\]
where $a_\eps(t)=a\ast\psi_{\omega(\eps)}$, with $a(t)=\mu(t)H(t-t_0)$ and $\omega(\eps)^{-1}=O((\ln(\eps^{-1}))^{r})$ and $d_\eps=d\ast\psi_{\nu(\eps)}$, where $d\in\E'(\R)$ with $\supp\, d\Subset\R$ and $\nu(\eps)=O((\ln(\eps^{-1}))^{-\frac{1}{N}})$, for some $N$ depending on $d$. Note that if $d=\delta$ then we can take $N=1$.

\section{Levi conditions and very weak solutions: Case 1}
We now consider the general Cauchy problem 
\beq
\label{CP_gen}
\begin{split}
\partial_t^2u(t,x)-\sum_{i=1}^n a_i(t)\partial_{x_i}^2u(t,x)+l(t,\partial_t,\partial_x)u(t,x)&=f(t,x),\\
u(0,x)&=g_0(x),\\
\partial_t u(0,x)&=g_1(x),\\
\end{split}
\eeq
where $t\in[0,T]$, $x\in\R^n$ and 
\[
l(t,\partial_t,\partial_x)=\sum_{i=1}^n c_i(t)\partial_{x_i}+d(t)\partial_t+e(t).
\]

In {\bf Case 1} we work under the following set of hypotheses:

\begin{itemize}
\item[(i)] the coefficients $a_i$ are real-valued distributions with compact support contained in $[0,T]$ with $a_i\ge 0$ for all $t\in[0,T]$,
 \item[(ii)] $f\in C([0,T], \gamma_c^s(\R^n))$,
\item[(iii)] $g_0,\,g_1\in \gamma_c^s(\R^n)$.
\end{itemize}

We first regularise the coefficients and initial data of the Cauchy problem \eqref{CP_gen} and then we rewrite it with $D_t$ and $D_x$ derivatives. We get  
\beq
\label{intro_eq_Dt_regularised}
\begin{split}
D_t^2u_{\eps}-\sum_{i=1}^n a_{i,\eps}(t)
D_{x_i}^2u_{\eps}-l_{\eps}(t,{\rm i}D_t,{\rm i}D_x)u_{\eps}&=-f_{\eps}(t,x),\\
u_{\eps}(0,x)&=g_{0,\eps}(x),\\
D_t u_{\eps}(0,x)&=-{\rm i}g_{1,\eps}(x).
\end{split}
\eeq
where $t\in[0,T]$, $x\in\R^n$ and 
\[
l_{\eps}(t,{\rm i}D_t,{\rm i}D_x)=\sum_{i=1}^n c_{i,\eps}(t){\rm i}D_{x_i}+d_{\eps}(t){\rm i}D_t+e_{\eps}(t)
\]
with $a_{i,\eps}=a_{i} \ast \psi_{\omega(\eps)}$, $c_{i,\eps}=c_{i} \ast \psi_{\omega(\eps)}$, $d_{\eps}=d \ast \psi_{\nu(\eps)}$, $e_{\eps}=e \ast \psi_{\nu(\eps)}$, $f_{\eps}=f \ast \varphi_\eps $, $g_{0,\eps}=g_{0} \ast \varphi_\eps$ and $g_{1,\eps}=g_{1} \ast \varphi_\eps$. Above, $\psi$ is a mollifier of type (1) and $\varphi$ a mollifier of type (2). Moreover, $\psi$ can be chosen in such a way that the regularised coefficients $(a_{i,\eps})_\eps$, $(b_{i,\eps})_\eps$ and $(c_{i,\eps})_\eps$ have compact support contained in $[0,T]$.

We now reduce \eqref{intro_eq_Dt_regularised} to a first order system using the transformation
\[	
u_{1,\eps}=\lara{D_x}u_{\eps}, \quad u_{2,\eps}=D_t u_{\eps},
\]
where  $\lara{D_x}$ is the pseudo-differential operator with symbol $\lara{\xi}=(1+|\xi|^2)^{\frac{1}{2}}$.
This gives us the following system
\beq \label{syst_Taylor_general}
D_t\left(
                             \begin{array}{c}
                               u_{1,\eps} \\                          
                               u_{2,\eps} \\
                             \end{array}
                           \right)
= M_{\eps}
  \left(\begin{array}{c}
                               u_{1,\eps} \\                             
                               u_{2,\eps} \\
                             \end{array}
                           \right)
                           -
\left(\begin{array}{c}
                               0 \\                               
                               f_{\eps} \\
                             \end{array}
                           \right) 
\eeq
where
\[
M_{\eps} =\left(
    \begin{array}{ccccc}
      0 & \lara{D_x} \\
      \left( \sum_{i=1}^n a_{i,\eps}D_{x_i}^2 + {\rm i}\sum_{i=1}^n c_{i,\eps} D_{x_i} + e_{\eps} \right)\lara{D_x}^{-1} & {\rm i}d_{\eps} \\
    \end{array}
  \right)
\]
The matrix  above can be written as $M_{\eps}=\mathbb{A}_{1,\eps}+B_{\eps}$ with
\begin{align*}
\mathbb{A}_{1,\eps}&=\left(
    \begin{array}{ccccc}
      0 & \lara{D_x} \\
      \sum_{i=1}^n a_{i,\eps}D_{x_i}^2\lara{D_x}^{-1} & 0 \\
    \end{array}
  \right) \quad \text{and}\\
B_{\eps}&=\left(
    \begin{array}{ccccc}
      0 & 0 \\
      \left({\rm i}\sum_{i=1}^n c_{i,\eps} D_{x_i} +e_{\eps} \right)\lara{D_x}^{-1} & {\rm i}d_{\eps} \\
    \end{array}
  \right).  
\end{align*}

By Fourier transforming both sides of \eqref{syst_Taylor_general} in $x$, we obtain the system
\beq
\label{system_new_general}
\begin{split}
D_t V&=\mathbb{A}_{1,\eps}(t,\xi)V+B_{\eps}(t,\xi)V+\widehat{F}_{\eps}(t,\xi),\\
V|_{t=0}(\xi)&=V_{0,\eps}(\xi),
\end{split}
\eeq
where $V$ is the $2$-column with entries $v_{j,\eps}=\widehat{u}_{j,\eps}$, $V_{0,\eps}$ is the $2$-column with entries
$v_{0,1,\eps}=\lara{\xi}\widehat{g}_{0,\eps}$, $v_{0,2,\eps}=-{\rm i}\widehat{g}_{1,\eps}$  and
\begin{align*}
\label{mA_general}
\mathbb{A}_{1,\eps}(t,\xi)&=\left(
    \begin{array}{ccccc}
      0 & \lara{\xi} \\
      \sum_{i=1}^n a_{i,\eps}\xi_i^2\lara{\xi}^{-1} & 0 \\
    \end{array}
  \right), \quad  \\
B_{\eps}(t,\xi)&=\left(
    \begin{array}{ccccc}
      0 & 0 \\
      \left({\rm i}\sum_{i=1}^n c_{i,\eps} \xi_i+e_{\eps} \right)\lara{\xi}^{-1} & {\rm i}d_{\eps} \\
    \end{array}
  \right), \quad	  
\widehat{F}_{\eps}(t,\xi)=\left(\begin{array}{c}
                               0 \\                              
                               -\widehat{f}_{\eps}(t,\cdot)(\xi) \\
                             \end{array}
                           \right). 
\end{align*}

Henceforth, we will focus on the system \eqref{system_new_general} and on the matrix
 $$A_{\eps}(t,\xi)\coloneqq\lara{\xi}^{-1}\mathbb{A}_{1,\eps}(t,\xi)$$
 for which we will construct a quasi-symmetriser. Note that the eigenvalues of the matrix $\mathbb{A}_{1,\eps}$
 are exactly the roots $\lambda_{1,\eps}(t,\xi),\,\lambda_{2,\eps}(t,\xi)$.
 Furthermore, the condition \eqref{KS_condition} holds for the eigenvalues
$\lara{\xi}^{-1}\lambda_{1,\eps}(t,\xi),\,\lara{\xi}^{-1}\lambda_{2,\eps}(t,\xi)$
of the $0$-order matrix $A_{\eps}(t,\xi)$ as well. We have that the eigenvalues of $A_{\eps}(t,\xi)$ are
\begin{align*}
\tilde{\lambda}_{1,2,\eps}\coloneqq \lara{\xi}^{-1}\lambda_{1,2,\eps}(t,\xi)= \pm\sqrt{\sum_{i=1}^n a_{i,\eps}\xi_i^2\lara{\xi}^{-2}},
\end{align*}
and hence the quasi-symmetriser 
\[
Q^{(2)}_{\delta}(\tilde{\lambda}_{1,\eps},\tilde{\lambda}_{2,\eps})=\left(
    \begin{array}{cc}
      \tilde{\lambda}_{1,\eps}^2+\tilde{\lambda}_{2,\eps}^2 & -(\tilde{\lambda}_{1,\eps}+\tilde{\lambda}_{2,\eps})\\
      -(\tilde{\lambda}_{1,\eps}+\tilde{\lambda}_{2,\eps}) & 2 \\
           \end{array}
  \right) + 2\delta^2 \left(
    \begin{array}{cc}
      1 & 0\\
      0 & 0 \\
           \end{array}
  \right),
\]
becomes
\[
Q^{(2)}_{\delta}(\tilde{\lambda}_{1,\eps},\tilde{\lambda}_{2,\eps})=\left(
    \begin{array}{cc}
     2\sum_{i=1}^n a_{i,\eps}\xi_i^2\lara{\xi}^{-2} & 0\\[0.2cm]
      0 & 2 \\
           \end{array}
  \right)
   + 2\delta^2 \left(
    \begin{array}{cc}
      1 & 0\\
      0 & 0 \\
           \end{array}
  \right).
\]
Before proceeding with the energy estimates we need to recall some important properties of the quasi-symmetriser $Q^{(2)}_\delta(\lambda_\eps)$ which have been proven in \cite{GR:14} and will be employed in the rest of the paper.
\subsection{Properties of $Q^{(2)}_{\delta}(\lambda_\eps)$}
We begin by noting that all the properties in the appendix  that we need for the quasi-symmetriser, can be stated in presence of the additional parameter $\eps$.  
\begin{proposition}[Proposition 3.5, \cite{GR:14}]
\label{prop_qs_2}
Let $Q^{(2)}_\delta(\lambda_\eps)$ as defined above. Then,
\beq
\label{est_diag_below}
(Q^{(2)}_\delta(\lambda_\eps)V,V)\ge \frac{1}{8} {\rm diag}(Q^{(2)}_{\delta}(\lambda_\eps)V,V),\nonumber
\eeq
where ${\rm diag}\,Q^{(2)}_{\delta}(\lambda_\eps)$ is the diagonal part of the matrix $Q^{(2)}_\delta(\lambda_\eps)$. In addition, there exists a constant $C_2>0$ such that
\begin{itemize}
\item[(i)]  $C_2^{-1}\omega(\eps)^{2L}\delta^2 I\le Q^{(2)}_\delta(\lambda_\eps(t,\xi))\le C_2\omega(\eps)^{-2L} I$,\\
\item[(ii)] $|((Q^{(2)}_\delta(\lambda_\eps)A_\eps(t,\xi)-A_\eps(t,\xi)^\ast Q^{(2)}_\delta(\lambda_\eps))V,V)|\le C_2\delta(Q^{(2)}_\delta(\lambda_\eps)V,V)$,\\
\end{itemize}
for all $\delta>0$, $\eps\in(0,1]$, $t\in[0,T]$, $\xi\in\R^n$ and $V\in\C^2$.  
\end{proposition}
Taking inspiration from the motivating example in the previous section, we can now formulate some Levi conditions and prove that they allow to bound the lower order terms by means of the quasi-symmetriser $Q^{(2)}_\delta(\lambda_\eps)$. More precisely, we have the following proposition.

\begin{proposition}\label{prop_lot_bound}
Let $Q^{(2)}_\delta(\lambda_\eps)$ as defined above. Under the Levi conditions 
\begin{itemize}
\item[(LC)] $\exists C_{1,\eps},\,C_{2,\eps}>0$ such that for $t\in [0,T]$ and $\xi$ away from $0$ (i.e., for $|\xi|\ge R$ for some $R>0$ independent of $\eps$),
\begin{align*} 
\left|{\rm i}\sum_{i=1}^n c_{i,\eps}(t)\xi_{i}+e_{\eps}(t) \right|^2 &\le C_{1,\eps} \left(  2\sum_{i=1}^n a_{i,\eps}(t)\xi_i^2 \right) \\
\left| d_{\eps}(t)\right|^2&\le C_{2,\eps}, 
\end{align*}
\end{itemize}
there exists a constant $C_{\eps}>0$ such that the inequality 
\[
|((Q^{(2)}_{\delta,\eps} B_{\eps}-B_{\eps}^\ast Q^{(2)}_{\delta,\eps})V,V)| \le C_{\eps}(Q^{(2)}_\delta(\lambda_\eps)V,V),
\]
holds for all values of $t$, $\xi$, $\eps$, $\delta$ as above and for all $V\in \C^2$. 
Furthermore, $$C_{\eps}=4\sqrt{C_{1,\eps}+\gamma C_{2,\eps}},$$ for some constant $\gamma>0$ independent of $t\in[0,T]$, $\xi\in\R^n$ and $\eps\in(0,1]$.
\end{proposition}

\begin{proof}
In this proof we combine the arguments of Section 4.3 in \cite{GR:12} at the $\eps$-level with Proposition \ref{prop_qs}(iv). For the quasi-symmetriser $Q^{(2)}_\delta(\lambda_\eps)$ we use the shorter notation $Q^{(2)}_{\delta,\eps}$. We have 
\begin{align*}
((Q^{(2)}_{\delta,\eps} B_{\eps}-B_{\eps}^\ast Q^{(2)}_{\delta,\eps})V,V)=&((Q^{(2)}_{0,\eps} B_{\eps}-B_{\eps}^\ast Q^{(2)}_{0,\eps})V,V)\\
&+\delta^2\sum_{i=1}^2((Q^{(1)}_{\delta}(\pi_i\lambda_\eps)^\sharp B_{\eps}-B_{\eps}^\ast Q^{(1)}_{\delta}(\pi_i\lambda_\eps)^\sharp)V,V),
\end{align*}
where for the definition of $Q^{(1)}_{\delta}(\pi_i\lambda_\eps)^\sharp$ we refer to Proposition \ref{prop_qs}(iv). From the structure of  $B_{\eps}$ and $Q^{(1)}_{\delta}(\pi_i\lambda_\eps)^\sharp$, we notice that
$(Q^{(1)}_{\delta}(\pi_i\lambda_\eps)^\sharp B_{\eps}-B_{\eps}^\ast Q^{(1)}_{\delta}(\pi_i\lambda_\eps)^\sharp)=0$. Hence
\[
|((Q^{(2)}_{\delta,\eps} B_{\eps}-B_{\eps}^\ast Q^{(2)}_{\delta,\eps})V,V)|=|((Q^{(2)}_{0,\eps} B_{\eps}-B_{\eps}^\ast Q^{(2)}_{0,\eps})V,V)|.
\]

We now note that from Proposition \ref{prop_qs}(i), we have that $(Q^{(2)}_{0,\eps} V,V)\le (Q^{(2)}_{\delta,\eps}V,V)$. Hence, if we can show that
\beq \label{cB}
|((Q^{(2)}_{0,\eps} B_{\eps}-B_{\eps}^\ast Q^{(2)}_{0,\eps})V,V)|\le C_{\eps} (Q^{(2)}_{0,\eps} V,V),
\eeq
for some constant $C_{\eps}>0$ independent of $t\in[0,T]$, $\xi\in\R^n$ and $V\in\C^2$, then the proof is concluded.

It therefore remains to show \eqref{cB}, under the Levi conditions (LC) on the lower order terms. Following Section 5 in \cite{GR:12} we first write $((Q^{(2)}_{0,\eps} B_{\eps}-B_{\eps}^\ast Q^{(2)}_{0,\eps})V,V)$ in terms of the matrix 
\[
\mathcal{W}_{\eps}=\left(
    \begin{array}{cc}
      -\tilde{\lambda}_{2,\eps} & 1\\
      -\tilde{\lambda}_{1,\eps} & 1\\
      \end{array}
  \right)  
 =\left(
    \begin{array}{cc}
      -\sqrt{\sum_{i=1}^n a_{i,\eps}\xi_i^2\lara{\xi}^{-2}} & 1\\
      +\sqrt{\sum_{i=1}^n a_{i,\eps}\xi_i^2\lara{\xi}^{-2}} & 1\\
      \end{array}
  \right).
\]
From Proposition \ref{prop_qs}(v) we have that, 

\[
((Q^{(2)}_{0,\eps} B_{\eps}-B_{\eps}^\ast Q^{(2)}_{0,\eps})V,V)=((\mathcal{W}_{\eps}B_{\eps}V,\mathcal{W}_{\eps}V)-(\mathcal{W}_{\eps}V,\mathcal{W}_{\eps}B_{\eps}V))
=2{\rm i}\Im (\mathcal{W}_{\eps}B_{\eps}V,\mathcal{W}_{\eps}V).
\]
It follows by the Cauchy-Schwartz inequality that
\[
|((Q^{(2)}_{0,\eps} B_{\eps}-B_{\eps}^\ast Q^{(2)}_{0,\eps})V,V)|\le 2|\mathcal{W}_{\eps}B_{\eps}V||\mathcal{W}_{\eps}V|.
\]
Again from Proposition \ref{prop_qs}(v) we get
\[
(Q^{(2)}_{0,\eps} V,V)=|\mathcal{W}_{\eps}V|^2.
\]
We therefore have that if
\beq
\label{cB2}
|\mathcal{W}_{\eps}B_{\eps}V|\le \frac{C_{\eps}}{2}|\mathcal{W}_{\eps}V| \nonumber
\eeq
for some constant $ C_{\eps}>0$ independent of $t$, $\xi$ and $V$, then the condition \eqref{cB} will hold. 


By definition of $\mathcal{W}_{\eps}$ and straightforward computations we have,
%
\begin{align*}
\mathcal{W}_{\eps}B_{\eps}V
  &=\mathcal{W}_{\eps}
 \left(
    \begin{array}{ccccc}
      0 & 0 \\
      \left({\rm i}\sum_{i=1}^n c_{i,\eps} \xi_i+e_{\eps} \right)\lara{\xi}^{-1} & {\rm i}d_{\eps} \\
    \end{array}
  \right) 
  \left(
    \begin{array}{c}
      V_{1}\\
      V_{2} \\
           \end{array}
  \right) \\
  &=\left(
    \begin{array}{c}
      \left({\rm i}\sum_{i=1}^n c_{i,\eps} \xi_i+e_{\eps} \right)\lara{\xi}^{-1}V_{1} +{\rm i}d_{\eps}V_{2}\\
      \left({\rm i}\sum_{i=1}^n c_{i,\eps} \xi_i+e_{\eps} \right)\lara{\xi}^{-1}V_{1} +{\rm i}d_{\eps}V_{2}\\
      \end{array}
  \right).
\end{align*}
Therefore,
\[
|\mathcal{W}_{\eps}B_{\eps}V|^2=2\left|\left(\sum_{i=1}^n{\rm i} c_{i,\eps} \xi_i+e_{\eps} \right)\lara{\xi}^{-1}V_{1} +{\rm i}d_{\eps}V_{2}\right|^2.
\]
We also have that
\[
\mathcal{W}_{\eps}V
  =\left(
    \begin{array}{cc}
      +\sqrt{\sum_{i=1}^n a_{i,\eps}\xi_i^2\lara{\xi}^{-2}}V_{1}+V_{2} \\
      -\sqrt{\sum_{i=1}^n a_{i,\eps}\xi_i^2\lara{\xi}^{-2}}V_{1}+V_{2}\\
      \end{array}
  \right)
\]
and hence
\begin{align*}
|\mathcal{W}_{\eps}V|^2&=\left|\sqrt{\sum_{i=1}^n a_{i,\eps}\xi_i^2\lara{\xi}^{-2}}V_{1}+V_{2} \right|^2
+ \left|-\sqrt{\sum_{i=1}^n a_{i,\eps}\xi_i^2\lara{\xi}^{-2}}V_{1}+V_{2}\right|^2.
\end{align*} 
It therefore suffices to show that 
\begin{align} \label{m21_general}
&|\mathcal{W}_\eps B_\eps V|^2 
=2\left|\left(\sum_{i=1}^n{\rm i} c_{i,\eps} \xi_i+e_{\eps} \right)\lara{\xi}^{-1}V_{1} +{\rm i}d_{\eps}V_{2}\right|^2\\ \nonumber
 &\le \frac{C_{\eps}^2}{4} \left( \left|\sqrt{\sum_{i=1}^n a_{i,\eps}\xi_i^2\lara{\xi}^{-2}}V_{1}+V_{2} \right|^2
+ \left|-\sqrt{\sum_{i=1}^n a_{i,\eps}\xi_i^2\lara{\xi}^{-2}}V_{1}+V_{2}\right|^2 \right).
\end{align}

Using the inequality $\frac{1}{2}|x-y|^2\le |x|^2+|y|^2$ we have that
\begin{align*}
 &\left|\sqrt{\sum_{i=1}^n a_{i,\eps}\xi_i^2\lara{\xi}^{-2}}V_{1}+V_{2} \right|^2
+ \left|-\sqrt{\sum_{i=1}^n a_{i,\eps}\xi_i^2\lara{\xi}^{-2}}V_{1}+V_{2}\right|^2 \\ 
\geq {} & \frac{1}{2}\left|2\sqrt{\sum_{i=1}^n a_{i,\eps}\xi_i^2\lara{\xi}^{-2}} V_{1}\right|^2 
= 2\left(\sum_{i=1}^n a_{i,\eps}\xi_i^2\lara{\xi}^{-2}\right)|V_1|^2.
\end{align*}

On the other hand
\[
\left|\left({\rm i}\sum_{i=1}^n c_{i,\eps} \xi_i+e_{\eps} \right)\lara{\xi}^{-1}V_{1} +{\rm i}d_{\eps}V_{2}\right|^2 \leq 2 \left(\left|\left({\rm i}\sum_{i=1}^n c_{i,\eps} \xi_i+e_{\eps} \right)\right|^2\lara{\xi}^{-2}|V_1|^2+|d_{\eps}|^2|V_2|^2\right)
\]
so we want to show the inequality
\begin{multline}
\label{m22_general}
4\left(\left|\left({\rm i}\sum_{i=1}^n c_{i,\eps} \xi_i+e_{\eps} \right)\right|^2\lara{\xi}^{-2}|V_1|^2+|d_{\eps}|^2|V_2|^2\right) \le \frac{C_{\eps}^2}{4} \left(2\sum_{i=1}^n a_{i,\eps}\xi_i^2\lara{\xi}^{-2}\right)|V_1|^2.
\end{multline}

We argue in two different areas as in the motivating example.

{\bf Area 1.} Let $|V_2|^2\le  \gamma  \left(2\sum_{i=1}^n a_{i,\eps}\xi_i^2\lara{\xi}^{-2}\right)|V_1|^2$ with $\gamma>0$. Note that if
\beq
\label{B1B2_general}
4\left(\left|\left({\rm i}\sum_{i=1}^n c_{i,\eps} \xi_i+e_{\eps} \right)\right|^2\lara{\xi}^{-2}
+ \gamma  \left(2\sum_{i=1}^n a_{i,\eps}\xi_i^2\lara{\xi}^{-2}\right) |d_{\eps}|^2 \right) 
\le \frac{C_{\eps}^2}{4} \left(2\sum_{i=1}^n a_{i,\eps}\xi_i^2\lara{\xi}^{-2}\right)
\eeq
then \eqref{m22_general} holds. Using the Levi conditions (LC) and that $a_{i,\eps}\ge0$, we have that \eqref{B1B2_general} holds independently of $\gamma>0$ with $\frac{C_{\eps}^2}{4}=4C_{1,\eps}+4\gamma C_{2,\eps}$.

{\bf Area 2.} Let  $|V_2|^2 >  \gamma  \left(2\sum_{i=1}^n a_{i,\eps}\xi_i^2\lara{\xi}^{-2}\right)|V_1|^2$. Using the Levi conditions (LC),
\begin{align*}
&4\left(\left|\left({\rm i}\sum_{i=1}^n c_{i,\eps} \xi_i+e_{\eps} \right)\right|^2\lara{\xi}^{-2}|V_1|^2+|d_{\eps}|^2|V_2|^2\right) \\
&\le 4\left( C_{1,\eps} \left( 2\sum_{i=1}^n a_{i,\eps}(t)\xi_i^2 \right)\lara{\xi}^{-2}|V_1|^2+C_{2,\eps}|V_2|^2 \right)\\
&\le 4(C_{2,\eps}+\frac{C_{1,\eps}}{\gamma})|V_2|^2 \le 4\max(C_{1,\eps},C_{2,\eps})(1+\frac{1}{\gamma})|V_2|^2.
\end{align*}
Note that 
\[
|V_1|\left|\sqrt{\sum_{i=1}^n a_{i,\eps}\xi_i^2\lara{\xi}^{-2}}\right|<\frac{1}{\sqrt{2\gamma}} |V_2|.
\] 

We therefore have that
\begin{align*}
&\frac{C_{\eps}^2}{4} \left( \left|\biggl(\sqrt{\sum_{i=1}^n a_{i,\eps}\xi_i^2\lara{\xi}^{-2}}\biggl) V_{1}+V_{2} \right|^2 + \left|\biggl(-\sqrt{\sum_{i=1}^n a_{i,\eps}\xi_i^2\lara{\xi}^{-2}}\biggl) V_{1}+V_{2} \right|^2 \right) \\ 
&\ge \frac{C_{\eps}^2}{4} \left( \left(|V_2|-\left|\biggl(\sqrt{\sum_{i=1}^n a_{i,\eps}\xi_i^2\lara{\xi}^{-2}}\biggl) V_1\right| \right)^2 + \left(|V_2|-\left|\biggl(-\sqrt{\sum_{i=1}^n a_{i,\eps}\xi_i^2\lara{\xi}^{-2}}\biggl) V_1\right| \right)^2 \right) \\
& \ge 2\frac{C_{\eps}^2}{4}(1-\frac{1}{\sqrt{2\gamma}} )^2|V_2|^2\ge \frac{C_{\eps}^2}{4}|V_2|^2 .
\end{align*}
We conclude that \eqref{m21_general} holds  if $\gamma$ is big enough with $\frac{C_{\eps}^2}{4}\geq 4 \max( C_{1,\eps},C_{2,\eps})$. Note that we can choose $C_{\eps}^2$ as in Area 1 and that $\gamma$ is independent of $\eps$ and hence it is a constant once it is fixed. Concluding, $\frac{C_{\eps}^2}{4}=4C_{1,\eps}+4\gamma C_{2,\eps}$ and hence $C_{\eps}=4\sqrt{C_{1,\eps}+\gamma C_{2,\eps}}$.

\end{proof}

\subsection{Energy estimates and very weak solutions}
In the rest of this section we will prove the under the hypotheses of Case 1 our Cauchy problem has a very weak solution.
\begin{theorem}
\label{theo_main_case_1}
Let us consider the Cauchy problem
\[
\begin{split}
\partial_t^2u(t,x)-\sum_{i=1}^n a_i(t)\partial_{x_i}^2u(t,x)+l(t,\partial_t,\partial_x)u(t,x)&=f(t,x),\\
u(0,x)&=g_0(x),\\
\partial_t u(0,x)&=g_1(x),\\
\end{split}
\]
where $t\in[0,T]$, $x\in\R^n$ and 
\[
l(t,\partial_t,\partial_x)=\sum_{i=1}^n c_i(t)\partial_{x_i}+d(t)\partial_t+e(t).
\]
Assume the following set of hypotheses for $s\ge1$ (Case 1):
\begin{itemize}
\item[(i)] the coefficients $a_i$  are real-valued distributions with compact support contained in $[0,T]$ with $a_i\ge 0$ for all $t\in[0,T]$,
 \item[(ii)] $f\in C([0,T], \gamma_c^s(\R^n))$,
\item[(iii)] $g_0,g_1\in \gamma_c^s(\R^n)$.
\end{itemize}
If the equation coefficients are regularised with a scale of logarithmic type and the following Levi conditions 
\begin{align*} 
\left|{\rm i}\sum_{i=1}^n c_{i,\eps}(t)\xi_{i}+e_{\eps}(t) \right|^2 &\le c_1(\ln(\eps^{-1}))^2 \left(  2\sum_{i=1}^n a_{i,\eps}(t)\xi_i^2 \right) \\
\left| d_{\eps}(t)\right|^2&\le c_2(\ln(\eps^{-1}))^2, 
\end{align*}
are fulfilled for $\eps\in(0,1]$ and for $|\xi|$ large enough, then the Cauchy problem has a very weak solution of order $s$.
\end{theorem}
\begin{remark}
Note that with scale of logarithmic type we mean $\omega^{-1}(\eps)=\ln(\eps^{-1})^{r}$ for some $r>0$. The regularisation is therefore given via convolution with $\psi_{\omega(\eps)}$, where $\psi$ is a mollifier of type $(1)$, for the equation coefficients and with $\varphi_{\omega(\eps)}$, where $\varphi$ is a mollifier of type (2) for initial data and right-hand side.
\end{remark}

\begin{proof}[Proof of Theorem \ref{theo_main_case_1}]
 
We work on the regularised problem 
\[
\begin{split}
\partial_t^2u_{\eps}(t,x)-\sum_{i=1}^n a_{i,\eps}(t)\partial_{x_i}^2u_{\eps}(t,x)+l_\eps(t,\partial_t,\partial_x)u_{\eps}(t,x)&=f_{\eps}(t,x),\\
u_{\eps}(0,x)&=g_{0,\eps}(x),\\
\partial_t u_{\eps}(0,x)&=g_{1,\eps}(x),\\
\end{split}
\]
where $t\in[0,T]$, $x\in\R^n$ and 
\[
l_\eps(t,\partial_t,\partial_x)=\sum_{i=1}^n c_{i,\eps}(t)\partial_{x_i}+d_\eps(t)\partial_t+e_\eps(t).
\]
Given mollifiers $\psi\ge 0$ in $C^\infty_c(\R)$  and $\varphi\in \S(\R^n)$ of type (1) and type (2), respectively, we set
\[
\begin{split}
&a_{i,\eps}= a_i\ast \psi_{\omega(\eps)},\,
c_{i,\eps}= c_i\ast \psi_{\nu(\eps)},\,
d_\eps= d\ast \psi_{\nu(\eps)},\,
e_\eps= e\ast \psi_{\nu(\eps)},\\
&f_\eps = f(t,\cdot)\ast\varphi_\eps,\,
g_{0,\eps}=g_0\ast\varphi_\eps,\,
g_{1,\eps}=g_1\ast\varphi_\eps.
\end{split}
\]
After regularisation, for fixed $\eps$ the Cauchy problem above fulfils the hypotheses of Theorem \ref{theo_GRi} so we know that there exists a net of solutions $(u_\eps)_\eps$ with $u_\eps\in C^2([0,T], \gamma^s(\R^n)$. This is true for every $s$ because after regularisation the coefficients are smooth. To prove that we have a very weak solution we need to show that, choosing the scales $\omega(\eps)$ and $\nu(\eps)$ suitably, then the net $(u_\eps)_\eps$ is moderate. We will achieve this by writing an $\eps$-version of the proof of Theorem \ref{theo_GRi} for the system \eqref{system_new_V}. We begin by defining the energy 
\[
E_{\delta,\eps}(t,\xi)\coloneqq(Q^{(2)}_{\delta,\eps}(t,\xi) V(t,\xi), V(t,\xi)).
\] 
We hence have
\begin{align} \label{EQ:energy_deriv_general}
\partial_t E_{\delta,\eps}(t,\xi)=&(\partial_tQ^{(2)}_{\delta,\eps} V,V)+ (Q^{(2)}_{\delta,\eps} \partial_tV,V)+(Q^{(2)}_{\delta,\eps} V,\partial_tV) \\ \nonumber
=&(\partial_tQ^{(2)}_{\delta,\eps} V,V)+ {\rm i}(Q^{(2)}_{\delta,\eps} D_tV,V)-{\rm i}(Q^{(2)}_{\delta,\eps} V,D_tV)\\ \nonumber
=&(\partial_tQ^{(2)}_{\delta,\eps} V,V)+{\rm i}(Q^{(2)}_{\delta,\eps}(\lara{\xi}A_{\eps}+B_{\eps})V-\widehat{F}_{\eps},V)\\ \nonumber
&-{\rm i}(Q^{(2)}_{\delta,\eps} V,(\lara{\xi}A_{\eps}+B_{\eps})V-\widehat{F}_{\eps})\\ \nonumber
=&(\partial_tQ^{(2)}_{\delta,\eps} V,V)+{\rm i}\lara{\xi}((Q^{(2)}_{\delta,\eps} A_\eps-A_\eps^\ast Q^{(2)}_{\delta,\eps})V,V)\\ \nonumber
&+{\rm i}((Q^{(2)}_{\delta,\eps} B_{\eps}-B_{\eps}^\ast Q^{(2)}_{\delta,\eps})V,V)-{\rm i}(Q^{(2)}_{\delta,\eps} \widehat{F}_{\eps},V)+ {\rm i}(Q^{(2)}_{\delta,\eps} V,\widehat{F}_{\eps}).
\end{align} 
Since our solution depends on the parameter $\eps$, we will use the notation $V_{\eps}$ from now on. From \eqref{EQ:energy_deriv_general} we have that 

\begin{align*} \label{EQ:energy_deriv_2}
\partial_t E_{\delta,\eps}(t,\xi) \le& |(\partial_tQ^{(2)}_{\delta,\eps} V_{\eps},V_{\eps})|+|\lara{\xi}((Q^{(2)}_{\delta,\eps} A_\eps-A_\eps^\ast Q^{(2)}_{\delta,\eps})V_{\eps},V_{\eps})|\\ \nonumber
&+|((Q^{(2)}_{\delta,\eps} B_{\eps}-B_{\eps}^\ast Q^{(2)}_{\delta,\eps})V_{\eps},V_{\eps})|+|(Q^{(2)}_{\delta,\eps} \widehat{F}_{\eps},V_{\eps})- (Q^{(2)}_{\delta,\eps} V_{\eps},\widehat{F}_{\eps})|.
\end{align*}
To estimate the first two terms in the right-hand side above, we follow the arguments of Section 4.2 in \cite{GR:14}. For the next two terms we employ Propositions \ref{prop_qs_2} and \ref{prop_lot_bound}. We get
\begin{equation}\label{est-rev_general}
\partial_t E_{\delta,\eps}(t,\xi) \le (K_{\delta,\eps}(t,\xi)+C_2\delta\lara{\xi}+C_{\eps})E_{\delta,\eps}(t,\xi)+|(Q^{(2)}_{\delta,\eps} \widehat{F}_{\eps},V_{\eps})- (Q^{(2)}_{\delta,\eps} V_{\eps},\widehat{F}_{\eps})|,
\end{equation}
where $K_{\delta,\eps}(t,\xi)$ has the property

\[
\int_0^T K_{\delta,\eps}(t,\xi)\, dt \le C_1\delta^{-\frac{2}{k}}\omega(\eps)^{-\frac{3L}{k}-1},
\]
with $k$ and $L$ depending on $a_i$ and $b_i$ and $C_1$, $C_2$ positive constants and the net $C_{\eps}$ depending on the Levi conditions on the lower order terms and defined as in Proposition \ref{prop_lot_bound}.
Note that 
\begin{align} 
\label{RHS_estimate}
|(Q^{(2)}_{\delta,\eps} \widehat{F}_{\eps},V_{\eps})- (Q^{(2)}_{\delta,\eps} V_{\eps},\widehat{F}_{\eps})| 
	&= |2\overline{V}_{2,\eps}\widehat{f_{\eps}}-2V_{2,\eps}\overline{\widehat{f_{\eps}}}| 
	\le 4|V_{2,\eps}||\widehat{f_{\eps}}| \nonumber\\
	&\le 2(|V_{2,\eps}|^2+|\widehat{F}_{\eps}|^2)
\end{align}
where in the last step we used that the quasi-symmetriser is a family of (nearly) diagonal Hermitian matrices. Hence 
\begin{align*}
E_{\delta,\eps}(t,\xi)&=(Q^{(2)}_{\delta,\eps} V_{\eps},V_{\eps}) =  (\text{diag }Q^{(2)}_{\delta,\eps} V_{\eps},V_{\eps}) \\
	&=\left(2\sum_{i=1}^n a_{i,\eps}\xi_i^2\lara{\xi}^{-2}+2\delta^2\right)|V_{1,\eps}|^2	+2|V_{2,\eps}|^2
	\ge 2|V_{2,\eps}|^2
\end{align*}
since $ a_i\ge0$. 
Using these inequalities in \eqref{RHS_estimate} we get that
\begin{align*}
|(Q^{(2)}_{\delta,\eps} \widehat{F}_{\eps},V_{\eps})- (Q^{(2)}_{\delta,\eps} V_{\eps},\widehat{F}_{\eps})| 	
	&\le 2|\widehat{F}_{\eps}|^2 + E_{\delta,\eps}(t,\xi).
\end{align*}

Therefore \eqref{est-rev_general} becomes
\begin{align*}
\partial_t E_{\delta,\eps}(t,\xi) \le (K_{\delta,\eps}(t,\xi)+C_2\delta\lara{\xi}+C_{\eps}+1)E_{\delta,\eps}(t,\xi)+2\underset{t \in [0,T]}{\sup}|\widehat{F}_{\eps}(t,\xi)|^2 .
\end{align*}

By Gronwall's lemma we obtain 
\begin{align*}
E_{\delta,\eps}(t,\xi)&\le \left( E_{\delta,\eps}(0,\xi)+2T\underset{t \in [0,T]}{\sup}|\widehat{F}_{\eps}(t,\xi)|^2 \right) \esp^{C_1\delta^{-\frac{2}{k}}\omega(\eps)^{-\frac{3L}{k}-1}+C_2T\delta \lara{\xi}+TC_{\eps}+T} \\
	&\le \left(  E_{\delta,\eps}(0,\xi)+2T\underset{t \in [0,T]}{\sup}|\widehat{F}_{\eps}(t,\xi)|^2 \right) \esp^{C_T(\delta^{-\frac{2}{k}}\omega(\eps)^{-\frac{3L}{k}-1}+\delta \lara{\xi}+C_{\eps}+1)}.
\end{align*}
As in \cite{GR:12}, we set $\delta^{-\frac{2}{k}}=\delta\lara{\xi}$. It follows that $\delta^{-\frac{2}{k}}=\lara{\xi}^{\frac{1}{\sigma}}$, where
\[
\sigma= 1+\frac{k}{2}.
\] 
Making use of Proposition \eqref{prop_qs_2}(i), of the definition of $Q^{(2)}_{\delta,\eps}$ and of the fact that $\omega(\eps)^{-1}\ge 1$, we obtain for some $D>0$ 
\begin{align*}
D^{-1}\omega(\eps)^{2L}\delta^2|V_{\eps}(t,\xi)|^2 \le& \left(  E_{\delta,\eps}(0,\xi)+2T\underset{t \in [0,T]}{\sup}|\widehat{F}_{\eps}(t,\xi)|^2 \right)\esp^{C_T(C_{\eps}+1)} \esp^{2C_T\lara{\xi}^{\frac{1}{\sigma}}\omega(\eps)^{-\frac{3L}{k}-1}}\\
	\le& \left( D\omega(\eps)^{-2L}|V_{\eps}(0,\xi)|^2+2T\underset{t \in [0,T]}{\sup}|\widehat{F}_{\eps}(t,\xi)|^2 \right)\\
	&\times\esp^{C_T(C_{\eps}+1)}\esp^{2C_T\lara{\xi}^{\frac{1}{\sigma}}\omega(\eps)^{-\frac{3L}{k}-1}}.
\end{align*}

This implies, for $M=(3L+k)/k$,
\begin{align*}
|V_{\eps}(t,\xi)|^2 \le {} &D\omega(\eps)^{-2L}\delta^{-2} \left( D\omega(\eps)^{-2L}|V_{\eps}(0,\xi)|^2+2T\underset{t \in [0,T]}{\sup}|\widehat{F}_{\eps}(t,\xi)|^2 \right) \\
	&\times \esp^{C_T(C_{\eps}+1)}\esp^{2C_T\lara{\xi}^{\frac{1}{\sigma}}\omega(\eps)^{-M}}\\
={} & D\omega(\eps)^{-2L}\lara{\xi}^{\frac{k}{\sigma}} \left( D\omega(\eps)^{-2L}|V_{\eps}(0,\xi)|^2+2T\underset{t \in [0,T]}{\sup}|\widehat{F}_{\eps}(t,\xi)|^2 \right) \\
	&\times \esp^{C_T(C_{\eps}+1)}\esp^{2C_T\lara{\xi}^{\frac{1}{\sigma}}\omega(\eps)^{-M}}.  
\end{align*}
Using the inequality $(|\alpha|^2+|\beta|^2)^{\frac{1}{2}}\le (|\alpha|+|\beta|)$ we have
\begin{align}\label{EEV}
|V_{\eps}(t,\xi)| \le& \sqrt{D}\omega(\eps)^{-L} \lara{\xi}^{\frac{k}{2\sigma}} \left( \sqrt{D}\omega(\eps)^{-L} |V_{\eps}(0,\xi)|+\sqrt{2T}\underset{t \in [0,T]}{\sup}|\widehat{F}_{\eps}(t,\xi)| \right) \\ \nonumber
&\times\esp^{ C_T(\frac{C_{\eps}}{2}+\frac{1}{2})} \esp^{C_T\lara{\xi}^{\frac{1}{\sigma}}\omega(\eps)^{-M}}  .
\end{align} 

Note that in this inequality we can clearly see the dependence of $V_{\eps}$ on the coefficients in the principal part as well as the dependence on the initial data and the right-hand side.
Since  $g_0,\,g_1,\,f(t,\cdot) \in \gamma_c^s(\R^n)$, we have from Proposition \ref{prop_reg_gevrey_sim}(iii) that
\[
|V_{\eps}(0,\xi)|\le c'_1 \esp^{-\kappa_1\lara{\xi}^{\frac{1}{s}}} \text{ and } \underset{t \in [0,T]}{\sup}|\widehat{F}_{\eps}(t,\xi)| \le c'_2\esp^{-\kappa_2\lara{\xi}^{\frac{1}{s}}} 
\]
where $c'_1,\,c'_2>0$.
Hence we have that 
\begin{align*}
|V_{\eps}(t,\xi)| \le& \sqrt{D} \omega(\eps)^{-L} \lara{\xi}^{\frac{k}{2\sigma}} \left( \sqrt{D}\omega(\eps)^{-L}  c'_1 \esp^{-\kappa_1\lara{\xi}^{\frac{1}{s}}}+\sqrt{2T} c'_2\esp^{-\kappa_2\lara{\xi}^{\frac{1}{s}}} \right) \\
&\times\esp^{C_T(\frac{C_{\eps}}{2}+\frac{1}{2})} \esp^{C_T\lara{\xi}^{\frac{1}{\sigma}}\omega(\eps)^{-M}} \\
			\le& \sqrt{D} \omega(\eps)^{-2L} \lara{\xi}^{\frac{k}{2\sigma}} \max{(\sqrt{D} c'_1, \sqrt{2T} c'_2)}  \esp^{-\min{(\kappa_1,\kappa_2)}\lara{\xi}^{\frac{1}{s}}} \\
&\times\esp^{C_T(\frac{C_{\eps}}{2}+\frac{1}{2})} \esp^{C_T\lara{\xi}^{\frac{1}{\sigma}}\omega(\eps)^{-M}}.
\end{align*} 

We now choose $\omega(\eps)^{-1}$ of logarithmic type, similarly to Section 4.2 of \cite{GR:14} and $\nu^{-1}(\eps)$ such that the Levi conditions hold, i.e., 
\[
C_{\eps}=4\sqrt{C_{1,\eps}+\gamma C_{2,\eps}})=O(\ln(\eps^{-1})).
\]
This proves that the net $V_\eps$ has the desired moderateness properties and therefore $(u_\eps)_\eps$ is a very weak solution of order $s\in [1, 1+\frac{k}{2}]$. Since $k$ can be chosen arbitrary then we can get a very weak solution for any order $s$.
\end{proof}

\section{Levi conditions and very weak solutions: Case 2}
We now pass to investigate {\bf Case 2} where the right-hand side and initial data are distributions with compact support.
\begin{theorem}
\label{theo_main_case_2}
Let us consider the Cauchy problem
\[
\begin{split}
\partial_t^2u(t,x)-\sum_{i=1}^n a_i(t)\partial_{x_i}^2u(t,x)+l(t,\partial_t,\partial_x)u(t,x)&=f(t,x),\\
u(0,x)&=g_0(x),\\
\partial_t u(0,x)&=g_1(x),\\
\end{split}
\]
where $t\in[0,T]$, $x\in\R^n$ and 
\[
l(t,\partial_t,\partial_x)=\sum_{i=1}^n c_i(t)\partial_{x_i}+d(t)\partial_t+e(t).
\]
Assume the following set of hypotheses for $s>1$ (Case 2):
\begin{itemize}
\item[(i)] the coefficients $a_i$ are real-valued distributions with compact support contained in $[0,T]$ with $a_i\ge 0$ for all $t\in[0,T]$,
 \item[(ii)] $f\in C([0,T], \E'(\R^n))$,
\item[(iii)] $g_0,\,g_1\in \E'(\R^n)$.
\end{itemize}
If the equation coefficients are regularised with a scale of logarithmic type and the following Levi conditions 
 
\begin{align*} 
\left|{\rm i}\sum_{i=1}^n c_{i,\eps}(t)\xi_{i}+e_{\eps}(t) \right|^2 &\le c_1(\ln(\eps^{-1}))^2 \left(  2\sum_{i=1}^n a_{i,\eps}(t)\xi_i^2 \right) \\
\left| d_{\eps}(t)\right|^2&\le c_2(\ln(\eps^{-1}))^2, 
\end{align*}
are fulfilled for $\eps\in(0,1]$ small enough and for $|\xi|$ large enough, then the Cauchy problem has a very weak solution of order $s$.
\end{theorem}
This is a straightforward extension of Theorem \ref{theo_main_case_1}. The main difference with respect to Case 1 is how we regularise $f$, $g_0$ and $g_1$. In detail, we use mollifiers of type (3) and order $s$ and we get that the nets
\[
g_{0,\eps}=g_0\ast\rho_\eps,\qquad g_{1,\eps}=g_1\ast\rho_\eps,
\]
are $\gamma^s$-moderate nets (see Proposition \ref{prop_reg_distr}). Concerning $f$, we have that
\[
 (f\ast\psi_\eps\rho_\eps)(t,x)=f_{s,y}(\psi_\eps(t-s)\rho_\eps(x-y))
 \]
 is $C^\infty([0,T], \gamma_s(\R^n))$-moderate, with $\psi$ mollifier of type (1). For the other equation coefficients we proceed with the regularisation as in Case 1.

\begin{proof}[Proof of Theorem \ref{theo_main_case_2}]
 
%

By repeating the transformation into a first order system and the energy estimates of Case 1 at the Fourier transform level we arrive at the inequality \eqref{EEV}, i.e.
\begin{align*}
|V_{\eps}(t,\xi)| \le& \sqrt{D}\omega(\eps)^{-L} \lara{\xi}^{\frac{k}{2\sigma}} \left( \sqrt{D}\omega(\eps)^{-L} |V_{\eps}(0,\xi)|+\sqrt{2T}\underset{t \in [0,T]}{\sup}|\widehat{F}_{\eps}(t,\xi)| \right) \\ \nonumber
&\times\esp^{ C_T(\frac{C_{\eps}}{2}+\frac{1}{2})} \esp^{C_T\lara{\xi}^{\frac{1}{\sigma}}\omega(\eps)^{-M}}  .
\end{align*} 
for $C_{\eps}=4\sqrt{C_{1,\eps}+\gamma C_{2,\eps}}$ with $M=(3L+k)/k$ and $C_{1,\eps},C_{2,\eps}>0$ coming from the Levi conditions as in the assumptions of Proposition \ref{prop_lot_bound}.
Since  $g_0,\,g_1,\,f \in \E'(\R^n)$, we have from Proposition \ref{prop_reg_distr} that $g_{0,\eps},\,g_{1,\eps},\,f_{\eps}(t,\cdot)$ are $\gamma_c^s(\R^n)$-moderate hence from Proposition \ref{prop_reg_gevrey_mod}(i) we have that
\[
|V_{\eps}(0,\xi)|\le  c'_1 \eps^{-N_1}\esp^{-\kappa_1\eps^{\frac{1}{s}} \lara{\xi}^{\frac{1}{s}}} \text{ and } \underset{t \in [0,T]}{\sup}|\widehat{F}_{\eps}(t,\xi)| \le  c'_2 \eps^{-N_2} \esp^{-\kappa_2 \eps^{\frac{1}{s}} \lara{\xi}^{\frac{1}{s}}} 
\]
for some constants $c'_1>0$ and $c'_2>0$ and $N_1$ and $N_2$ depending on the structure of initial data and right-hand side.
Hence for $\omega(\eps)^{-1} \ge 1$ and some $D>0$ we have that 
\begin{align*}
|V_{\eps}(t,\xi)| \le& \sqrt{D} \omega(\eps)^{-L} \lara{\xi}^{\frac{k}{2\sigma}} \left( \sqrt{D}\omega(\eps)^{-L}  c'_1\eps^{-N_1} \esp^{-\kappa_1\eps^{\frac{1}{s}} \lara{\xi}^{\frac{1}{s}}} +\sqrt{2T} c'_2\eps^{-N_2}\esp^{-\kappa_2 \eps^{\frac{1}{s}} \lara{\xi}^{\frac{1}{s}}}  \right) \\
&\times\esp^{C_T(\frac{C_{\eps}}{2}+\frac{1}{2})} \esp^{C_T\lara{\xi}^{\frac{1}{\sigma}}\omega(\eps)^{-M}} \\
\le& \sqrt{D} \omega(\eps)^{-2L} \lara{\xi}^{\frac{k}{2\sigma}} \max{(\sqrt{D} c'_1, \sqrt{2T} c'_2)}  \esp^{-\min{(\kappa_1,\kappa_2)}\eps^{\frac{1}{s}} \lara{\xi}^{\frac{1}{s}} } \eps^{-\max{(N_1, N_2)}}\\
&\times\esp^{C_T(\frac{C_{\eps}}{2}+\frac{1}{2})} \esp^{C_T\lara{\xi}^{\frac{1}{\sigma}}\omega(\eps)^{-M}}.
\end{align*} 

Choosing $\omega(\eps)^{-1}$ of logarithmic type as in Case 1 and making use of the Levi conditions on $C_\eps$ (i.e. $C_\eps=O(\ln(\eps^{-1}))$) we easily see that $V_\eps$ is $\gamma^s$-moderate, as desired.   
\end{proof}

\section{Uniqueness and consistency}
By analysing the proof of Theorem \ref{theo_main_case_1} and Theorem \ref{theo_main_case_2} is it straightforward to observe that if equation coefficients and initial data are given by negligible nets then the very weak solution $(u_\eps)_\eps$ will be negligible too. This allows us to conclude that the very weak solution is \emph{unique in the very weak sense}, i.e., negligible changes in coefficients and data will lead to negligible changes in the solution. This is equivalent to say that the Cauchy problem is well-posed in a suitable space of generalised functions of Colombeau type as already observed in \cite{GR:14}. In the rest of the section we will prove that when the coefficients are regular enough and the initial data are Gevrey, then the very weak solutions are consistent with the classical ones obtained in \cite{GR:12, KS}. This requires to preliminary show that the (classical) Levi conditions formulated in \cite{GR:12} imply the generalised ones introduced in this paper.  
\begin{proposition}
\label{prop_consistency}
Let 
\begin{equation*}
\label{CP_cons_prop}
\begin{split}
&\partial_t^2u(t,x)-\sum_{i=1}^n a_i(t)
\partial_{x_i}^2u(t,x)+l(t,\partial_t,\partial_x)u(t,x)=f(t,x),\\
&\text{where} \quad l(t,\partial_t,\partial_x)=\sum_{i=1}^n c_i(t)\partial_{x_i}+d(t)\partial_t+e(t),
\end{split}
\end{equation*}
and the real-valued coefficients $a_i$ are compactly supported, belong to $C^k([0,T])$ with $k\ge 2$ and $a_i\ge 0$ for all $i=1,\,\dots,\,n$. Suppose the Levi conditions on the lower order terms hold
\begin{align*}
\left|{\rm i}\sum_{i=1}^n c_i(t)\xi_{i}+e(t) \right|^2 &\le C_1 \left[  \left(\sum_{i=1}^n b_i(t)\xi_i\right)^2 +2\sum_{i=1}^n a_i(t)\xi_i^2 \right], \\ \nonumber
|d(t)|^2&\le C_2.
\end{align*}
Then the Levi conditions also hold for the regularised lower order terms
\begin{align*}
\left|{\rm i}\sum_{i=1}^n c_{i,\eps}(t)\xi_{i}+e_{\eps}(t) \right|^2 &\le \tilde{C_1} \left[  \left(\sum_{i=1}^n b_{i,\eps}(t)\xi_i\right)^2 +2\sum_{i=1}^n a_{i,\eps}(t)\xi_i^2 \right], \\ \nonumber
|d_{\eps}(t)|^2&\le \tilde{C_2},
\end{align*}
for some constants $\tilde{C_1}$, $\tilde{C_2}$ independent of $\eps$.
\end{proposition}

\begin{proof}

Starting with the first Levi condition,
\begin{align}
\left|{\rm i}\sum_{i=1}^n c_{i,\eps}(t)\xi_{i}+e_\eps(t) \right|^2 &=\left| {\rm i}\sum_{i=1}^n c_{i}\ast \varphi_\eps(t)\xi_{i}+e\ast \varphi_\eps(t) \right|^2 = \left| \left({\rm i}\sum_{i=1}^n c_{i}\xi_{i}+e  \right) \ast \varphi_\eps(t)\right|^2 \nonumber\\
& \le \left( \left|{\rm i}\sum_{i=1}^n c_{i}\xi_{i}+e \right| \ast \varphi_\eps(t)\right)^2\nonumber \\
&=  \left( \int_{\supp (\varphi_\eps)} \left|{\rm i}\sum_{i=1}^n c_{i}(t-\tau)\xi_{i}+e(t-\tau) \right| \varphi_\eps(\tau) d\tau \right)^2 \nonumber\\
&\le \left( \int_{\supp (\varphi_\eps)} C_1^\frac{1}{2}  \left[  2\sum_{i=1}^n a_i(t-\tau)\xi_i^2 \right]^\frac{1}{2} \varphi_\eps(\tau) d\tau \right)^2 \label{estimate1}
\end{align}
By Holder's inequality, the right-hand side of \eqref{estimate1} can be estimated by
\begin{align*}
&\le \mu(\supp (\varphi_\eps)) \int_{\supp (\varphi_\eps)} C_1  \left[   2\sum_{i=1}^n a_i(t-\tau)\xi_i^2  \right]\varphi_\eps(\tau)^2  d\tau\\
&= \eps CC_1 \left(   \int_{\supp (\varphi_\eps)} 2\sum_{i=1}^n a_i(t-\tau)\xi_i^2 \varphi_\eps(\tau)^2 d\tau\right)\\
&= \eps CC_1 \left( \int_{\supp (\varphi_\eps)} 2\sum_{i=1}^n a_i(t-\tau)\xi_i^2 \varphi_\eps(\tau)\frac{1}{\eps}\varphi\left(\frac{\tau}{\eps}\right) d\tau\right)\\
&\le  CC_1 \left( \int_{\supp (\varphi_\eps)} 2\sum_{i=1}^n a_i(t-\tau)\xi_i^2 \varphi_\eps(\tau)\tilde{C} d\tau\right)=\tilde{C_1} \left(    2\sum_{i=1}^n a_i \ast \varphi_\eps \xi_i^2 \right)\\
&=\tilde{C_1} \left(    2\sum_{i=1}^n a_{i,\eps} \xi_i^2 \right).
\end{align*}

For the second Levi condition,
\begin{align*}
|d_\eps(t)|^2&= |d\ast \varphi_\eps (t)|^2= \left| \int_\R d(t-\tau)\varphi_\eps(\tau) d\tau \right|^2 \le \left(\int_\R \left| d(t-\tau) \right|\varphi_\eps(\tau) d\tau \right)^2 \\
&\le \left(\int_\R C_2^\frac{1}{2} \varphi_\eps(\tau) d\tau \right)^2 = C_2 \left(\int_\R \varphi_\eps(\tau) d\tau \right)^2= C_2 .
\end{align*}
\end{proof}

We are now ready to prove that every very weak solution will converge to the classical solution as $\eps\to 0$ when the equation coefficients are regular enough. This result clearly holds independently of the choice of regularisation, i.e., mollifier and scale.

\begin{theorem}
\label{theo_consistency}
Let 
\beq
\label{CP_cons}
\begin{split}
\partial_t^2u(t,x)-\sum_{i=1}^n a_i(t)
\partial_{x_i}^2u(t,x)+l(t,\partial_t,\partial_x)u(t,x)&=f(t,x),\\
u(0,x)&=g_0(x),\\
\partial_t u(0,x)&=g_1(x).
\end{split}
\eeq
\vspace{0.2cm}
where
\[
l(t,\partial_t,\partial_x)=\sum_{i=1}^n c_i(t)\partial_{x_i}+d(t)\partial_t+e(t),
\]
and the real-valued coefficients $a_i$ are compactly supported, belong to $C^k([0,T])$ with $k\ge 2$ and $a_i\ge 0$ for all $i=1,\,\dots,\,n$. Let $g_0$ and $g_1$ belong to $\gamma^s_c(\R^n)$ and $f \in C([0,T];\gamma^s_c(\R^n))$ with $s \ge 1$. Suppose the Levi conditions on the lower order terms hold
\begin{align*}
\label{THM:lot2_general}
\left|{\rm i}\sum_{i=1}^n c_i(t)\xi_{i}+e(t) \right|^2 &\le C_1 \left[  \left(\sum_{i=1}^n b_i(t)\xi_i\right)^2 +2\sum_{i=1}^n a_i(t)\xi_i^2 \right], \\ \nonumber
|d(t)|^2&\le C_2.
\end{align*}
Then, any very weak solution $(u_\eps)_\eps$ of $u$ converges in $C([0,T];\gamma^s(\R^n))$ as $\eps\to0$ to the unique classical solution in $C^2([0,T], \gamma^s(\R^n))$ of the Cauchy problem \eqref{CP_cons};

\end{theorem}

\begin{proof}
 Let $\wt{u}$ be the classical solution. By definition we know that
\beq
\label{CP_class}
\begin{split}
D_t^2\wt{u}(t,x)-\sum_{i=1}^n a_i(t)
D_{x_i}^2\wt{u}(t,x)-l(t,{\rm i}D_t,{\rm i}D_x)\wt{u}(t,x)&=-f(t,x),\\
\wt{u}(0,x)&=g_0(x),\\
D_t \wt{u}(0,x)&=-{\rm i}g_1(x).
\end{split}
\eeq
Note that the initial data do not need to be regularised because they are already Gevrey functions. By Proposition \ref{prop_consistency}, the Levi conditions also hold for the regularised coefficients with constants independent of $\eps$. Hence, there exists a very weak solution $(u_\eps)_\eps$ of $u$ such that
\beq
\label{CP_class_2}
\begin{split}
D_t^2u_\eps(t,x)-\sum_{i=1}^n a_{i,\eps}(t)
D_{x_i}^2u_\eps(t,x)-l_\eps(t,{\rm i}D_t,{\rm i}D_x)u_\eps(t,x)&=-f_\eps(t,x),\\
u_\eps(0,x)&=g_0(x),\\
D_t u_\eps(0,x)&=-{\rm i}g_1(x),
\end{split}
\eeq
for suitable embeddings of the coefficients $a_i$. Noting that the nets $(a_{i,\eps}-a_i)_\eps$, $(c_{i,\eps}-c_i)_\eps$, $(d_{\eps}-d)_\eps$, $(e_{\eps}-e)_\eps$ and $(f_{\eps}-f)_\eps$ are converging to $0$ in $C([0,T]\times\R^n)$ for $i=1,\,\dots,\,n$ we can rewrite \eqref{CP_class} as
\beq
\label{CP_class_3}
\begin{split}
D_t^2\wt{u}(t,x)-\sum_{i=1}^n a_{i,\eps}(t)
D_{x_i}^2\wt{u}(t,x)-l_\eps(t,{\rm i}D_t,{\rm i}D_x)\wt{u}(t,x)&=- f_\eps(t,x)+n_\eps(t,x),\\
\wt{u}(0,x)&=g_0(x),\\
D_t \wt{u}(0,x)&=-{\rm i}g_1(x).
\end{split}
\eeq
where $n_\eps\in C([0,T];\gamma^s(\R^n))$ and converges to $0$ in this space. From \eqref{CP_class_3} and \eqref{CP_class_2} we get that $\wt{u}-u_\eps$ solves the Cauchy problem
\[
\begin{split}
D_t^2(\wt{u}-u_\eps)(t,x)-\sum_{i=1}^n a_{i,\eps}(t)D_{x_i}^2(\wt{u}-u_\eps)(t,x)-l_\eps(t,{\rm i}D_t,{\rm i}D_x)(\wt{u}-u_\eps)&=n_\eps(t,x),\\
(\wt{u}-u_\eps)(0,x)&=0,\\
(D_t \wt{u}-D_t u_\eps)(0,x)&=0,\\
\end{split}
\]
Following the energy estimates of Case 1 and arguing as in the proof of Theorem \ref{theo_main_case_1}, after reduction to a system and by applying the Fourier transform, we have an estimate of $|(\wt{V}-V_\eps)(t,\xi)|$ as in \eqref{EEV}, in terms of  $(\wt{V}-V_\eps)(0,\xi)$ and the right-hand side $n_\eps(t,x)$. In particular, since the coefficients are of class $C^k$, $k\ge 2$, we note that $\Vert Q^{(2)}_{\delta,\eps}(\cdot,\xi)\Vert_{{C}^k([0,T])}$ is uniformly bounded with respect to  $\eps$, because we can differentiate the entries up to order $k$, without putting any derivatives on the mollifier. Therefore the estimate (27) in \cite{GR:14} becomes
\[
\int_0^T K_{\delta,\eps}(t,\xi)\, dt \le C_1\delta^{-\frac{2}{k}}.
\]
Hence the terms  $\omega(\eps)^{-L}$ and $\omega(\eps)^{-M}$ disappear in \eqref{EEV} and we simply get
\begin{align*}
\label{impor_est_const}
|(\wt{V}-V_\eps)(t,\xi)| \le& \sqrt{D} \lara{\xi}^{\frac{k}{2\sigma}} \left( \sqrt{D}|(\wt{V}-V_\eps)(0,\xi)|+\sqrt{2T}\underset{t \in [0,T]}{\sup}|\widehat{n_\eps}(t,\xi)| \right) \\ \nonumber
&\times\esp^{ C_T(\frac{C}{2}+\frac{1}{2})} \esp^{C_T\lara{\xi}^{\frac{1}{\sigma}}}  .
\end{align*} 
Note that $\frac{C}{2}$ comes from the Levi conditions and from Proposition \ref{prop_consistency} it can be chosen independent of $\eps$. Since  $(\wt{V}-V_\eps)(0,\xi)=0$ and $n_\eps\to 0$ in $C([0,T];\gamma^s(\R^n))$, we conclude that $u_\eps\to \wt{u}$ in $C([0,T];\gamma^s(\R^n))$. Note that our argument is independent of the choice of the regularisation of the coefficients and the right-hand side. 

\end{proof}
\section{Examples and numerical models}
In this final section we will study the Cauchy problem 
\begin{equation*}
\label{CP_gen_numerical}
\begin{split}
\partial_t^2u(t,x)-a(t)\partial_x u&=f(t,x),\qquad t\in[0,T], x\in\R,\\
u(0,x)&=g_0(x),\\
\partial_t u(0,x)&=g_1(x), 
\end{split}
\end{equation*}
under different set of hypotheses on the coefficients, right-hand side and initial data. For the sake of the reader, we start by recalling some classical results for $C^\infty$ well-posedness for weakly hyperbolic Cauchy problems that we will employ in the sequel.

\subsection{$C^\infty$ well-posedness for weakly hyperbolic equations} \label{section_Oleinik}
The next theorem is due to the pioneer work of Oleinik on second order hyperbolic equations. In \cite{O70},  Oleinik considers the following Cauchy problem in the domain $G=\{0 \le t \le T, x \in \R^m\}$ 
\beq \label{Oleinik_CP}
u_{tt}(t,x)-\sum_{i,j=1}^{m} (a^{ij}(t,x)u_{x_i})_{x_j}+ \sum_{i=1}^{m} b^i(t,x)u_{x_i} + b^0(t,x)u_t+c(t,x)=f(t,x),
\eeq
\[
u(0,x)=g_0 (x), \quad u_t(0,x)=g_1 (x),
\]
where $x \in \R^m $ and $a^{ij}(t,x)\xi_i\xi_j \ge 0$ in $G$ for all $\xi \in \R^m$, obtaining the $C^\infty$ well-posedness result below.

\begin{theorem} \label{Oleinik_theorem}
Assume that there exists a constant $D$ such that for the coefficients of \eqref{Oleinik_CP}, the inequality
\beq	\label{Oleinik_inequality}
\sum_{i=1}^{m}\alpha t(b^i \xi_i)^2 \le  \sum_{i,j=1}^{m} B a^{ij} \xi_i \xi_j + a_t^{ij} \xi_i \xi_j 
\eeq
holds in $G=\{0 \le t \le T, x \in \R^m\}$,  for any $\xi \in \R^m$. Here $\alpha > (2p+6)^{-1}$ (p being an integer greater than or equal to $-1$), $B$ is some constant for $t \in [0,t_0]$, $t_0=const.>0$ and $\alpha$, $B$ are some positive constants for  $t \in [t_0, T]$. Let $a^{ij},$ $a_{x_i}^{ij}$, $b^i$, $b^0$, $b_t^0$, $c$ and their derivatives up to the order $k$, $k \ge 2$, with respect to x and up to the order $k-2$ with respect to $x$ and $t$ be bounded in $G$. Let the derivatives of the form $\partial^\rho_t \partial^\beta_x$  of $a^{ij}$, $a_{x_i}^{ij}$, $b^i$, $b^0$, $b_t^0$, $c$, where $\rho \in \N$, $\rho \le p+1$ and $\beta \in \N^{\N}$, $|\beta|\le k$, be bounded for $t \in [0,t_0]$  and the derivatives of these functions of the form $\partial^\rho_t \partial^\beta_x$, $\rho \le p$, $|\beta|\le p+ k$, be bounded for $t=0$. Suppose that the functions $f$, $g_0$, $g_1$ have compact support.
Then there exists a unique solution u of \eqref{Oleinik_CP} and the estimate 
\beq \label{Oleinik_estimate}
\begin{split}
&\forall \tau \in [0,T], \quad \Vert u(\tau, \cdot) \Vert ^2_{H^k(\R)} \le  C \biggl(\Vert g_0 \Vert ^2_{H^{k+p+4}(\R)}+\Vert g_1 \Vert ^2_{H^{k+p+3}(\R)} +\int_0^{\tau} \Vert f(\sigma, \cdot )\Vert ^2_{H^k(\R)} d\sigma \\
&+ \Vert f(\tau, \cdot) \Vert ^2_{H^{k-2}(\R)} 
+ \sum_{\substack{\rho \le p \\ |\beta| \le p+k+2}} \int_{\R^m} |\partial^\rho_t \partial^\beta_x f(0,x)|^2 dx 
+ \max_{\sigma \in [0,t_0]} \sum_{\substack{\rho \le p+1 \\ |\beta| \le k}} \int_{\R^m} |\partial^\rho_t \partial^\beta_x f(\sigma,x)|^2 dx \biggr)
\end{split}
\eeq
holds provided that the norms of $f$, $g_0$ and $g_1$ on the right of \eqref{Oleinik_estimate} are finite, where the constant $C>0$, depends on $\Vert \cdot \Vert_\infty$ in $G$ of the derivatives of the coefficients $a^{ij}$,  $a_{x_i}^{ij}$, $b^i$, $b^0$, $b_t^0$, $c$, as stated above. If $2(k-2)>m+1$, then the classical solution to \eqref{Oleinik_CP} exists.
\end{theorem}


We now consider the Cauchy problem 
\beq \label{Oleinik_CP_simple}
u_{tt}(t,x)-a(t)u_{xx}=0,
\eeq
\[
u(0,x)=g_0 (x), \quad u_t(0,x)=g_1 (x),
\]
where $x \in \R $, $t \in [0,T]$ and $a(t) \ge 0$ on $[0,T]$. By reformulating Theorem \ref{Oleinik_theorem} in this context we obtain the next statement.

\begin{theorem} \label{Oleinik_theorem_simple}
Assume that there exists a constant $D$ such that for the coefficient of \eqref{Oleinik_CP_simple}, the inequality
\beq	\label{Oleinik_inequality_simple}
0 \le  D a(t) + a_t(t)
\eeq
holds for any $t \in [0,T]$. Here p is an integer greater than or equal to $-1$, $D$ is some constant for $t \in [0,t_0]$, $t_0=const.>0$ and some positive constant for  $t \in [t_0, T]$. Let $a$ and its derivatives up to the order $k-2$, $k \ge 2$ with respect to  $t$ be bounded on $[0,T]$. Let the derivatives of the form $d^{\rho}/dt^{\rho}$ of $a$, where $\rho \in \N$, $\rho \le p+1$, be bounded for $t \in [0,t_0]$  and the derivatives  of the form $d^{\rho}/dt^{\rho}$, $\rho \le p$, be bounded for $t=0$. Suppose that the functions $g_0$, $g_1$ have compact support.
Then there exists a unique solution u of \eqref{Oleinik_CP_simple} and the estimate 
\beq \label{Oleinik_estimate_simple}
\begin{split}
\forall \tau \in [0,T], \quad \Vert u(\tau, \cdot) \Vert ^2_{H^k(\R)} \le & C (\Vert g_0 \Vert ^2_{H^{k+p+4}(\R)}+\Vert g_1 \Vert ^2_{H^{k+p+3}(\R)})
\end{split}
\eeq
holds provided that the norms of $g_0$ and $g_1$ on the right of \eqref{Oleinik_estimate_simple} are finite, where the constant $C>0$, depends on $\Vert \cdot \Vert_\infty$ on $[0,T]$ of the derivatives of the coefficient $a$, as stated above. If $2(k-2)>m+1$, then the classical solution to \eqref{Oleinik_CP_simple} exists.
\end{theorem}

\begin{remark}
Note that if $a,\,a_t \ge 0 $ then the inequality \eqref{Oleinik_inequality_simple} is automatically satisfied for all $D>0$ .
\end{remark}

Choosing $p=-1$ and $k=2$ in the above estimate, we obtain the following.
\begin{corollary}\label{Oleinik_corollary_k,p}
Assume that there exists a constant $D$ such that for the coefficient of \eqref{Oleinik_CP_simple}, the inequality
\beq	\label{Oleinik_inequality_simple_k,p}
0 \le  D a(t)+ a_t(t)
\eeq
holds for any $t \in [0,T]$. Here $D$ is some constant for $t \in [0,t_0]$, $t_0=const.>0$ and some positive constant for  $t \in [t_0, T]$. Let $a$  be bounded on $[0,T]$. Suppose that the functions $g_0,\, g_1$ have compact support.
Then there exists a unique solution u of \eqref{Oleinik_CP_simple} and the estimate 
\beq \label{Oleinik_estimate_simple_k,p}
\begin{split}
\forall \tau \in [0,T], \quad \Vert u(\tau, \cdot) \Vert ^2_{H^2(\R)} \le & C (\Vert g_0 \Vert ^2_{H^{5}(\R)}+\Vert g_1 \Vert ^2_{H^{4}(\R)})
\end{split}
\eeq
holds provided that the norms of $g_0$ and $g_1$ on the right of \eqref{Oleinik_estimate_simple_k,p} are finite, where the constant $C>0$, depends on $\Vert \cdot \Vert_\infty$ on $[0,T]$ of the coefficient $a$.   
\end{corollary}

We conclude this section by focusing on two toy models where the coefficient $a(t)$ in the principal is a Heaviside function or a delta distribution. For simplicity, we set $f=0$.

\subsection{First model: $a(t)=H(t)$}
Let
\beq
\label{CP_gen_Oleinik}
\begin{split}
\partial_t^2u(t,x)-H(t-1)\partial_x^2 u(t,x)&=0,\qquad t\in[0,T], x\in\R,\\
u(0,x)&=g_0(x),\\
\partial_t u(0,x)&=g_1(x),
\end{split}
\eeq
where $g_0,\, g_1 \in \mathcal{C}_c^{\infty}(\R)$ and $H$ is the Heaviside function 
\begin{equation*}
H(t)=
\begin{cases}
1, \quad \text{for } t \ge 0 \\
0, \quad \text{for } t <0.
\end{cases}
\end{equation*}
By applying the regularisation methods described before, one can transform the Cauchy problem \eqref{CP_gen_Oleinik} to 
\beq
\label{CP_gen_reg}
\begin{split}
\partial_t^2u_\eps(t,x)-H_\eps(t-1)\partial_x^2 u_\eps(t,x)&=0,\qquad t\in[0,T], x\in\R,\\
u_\eps(0,x)&=g_{0}(x),\\
\partial_t u_\eps(0,x)&=g_{1}(x),
\end{split}
\eeq
where $H_{\eps}=H*\varphi_\eps$, $\varphi_\eps(t)=\frac{1}{\eps}\varphi(\frac{t}{\eps})$ and $\varphi \in \mathcal{C}_c^{\infty}(\R)$ with $\varphi \ge 0$ and $\int \varphi =1$. Note that since $g_0,\, g_1 \in \mathcal{C}_c^{\infty}(\R)$, we do not need to mollify them.

We now apply the results of Section \ref{section_Oleinik} to the Cauchy problem \eqref{CP_gen_reg}.

\begin{proposition} \label{Oleinik_estimate_Heaviside_proposition}
The solution net, $u_\eps$, of the Cauchy problem \eqref{CP_gen_reg} fulfils
\beq \label{Oleinik_estimate_Heaviside}
\begin{split}
\forall \tau \in [0,T], \quad \Vert u_{\eps}(\tau, \cdot) \Vert ^2_{H^2(\R)} \le & C (\Vert g_{0} \Vert ^2_{H^{5}(\R)}+\Vert g_{1}\Vert ^2_{H^{4}(\R)}),
\end{split}
\eeq
where $C$ is independent of $\eps$.
\end{proposition}

\begin{proof}
We note that $H_{\eps},\,H'_{\eps}\ge 0$ and hence for any $D>0$, the condition \eqref{Oleinik_inequality_simple_k,p} is automatically satisfied. Hence from Corollary \ref{Oleinik_corollary_k,p}, we get the estimate \eqref{Oleinik_estimate_Heaviside} where $C$ depends on $\Vert H_\eps(\cdot -1) \Vert_\infty$ on $[0,T]$. We now calculate $\Vert H_\eps(\cdot -1) \Vert_\infty$ to show that it is independent of $\eps$.
\begin{align*}
\Vert H_\eps (\cdot -1) \Vert_\infty &= \Vert H* \varphi_\eps (\cdot -1) \Vert_\infty= \sup_{t \in [0,T]} \left|\int_\R H(t-1 -\tau) \varphi_\eps(\tau) d\tau \right| \\
&\le  \int_\R \sup_{t \in [0,T]}  \left| H(t-1 -\tau) \varphi_\eps(\tau) \right| d\tau 
 = \int_\R 1 \times \left| \varphi_\eps(\tau) \right| d\tau \\
&=  \int_\R  \varphi_\eps(\tau)  d\tau=1.
\end{align*}
Therefore $\Vert H_\eps (\cdot-1) \Vert_\infty \le 1$ and hence $C$ can be chosen independent of $\eps$.
\end{proof}

\begin{remark}
The proof of Proposition \ref{Oleinik_estimate_Heaviside_proposition}  holds for more general initial data $g_0,\,g_1$. For instance one can take $g_0 \in H^{5}(\R)$, $g_1 \in H^{4}(\R)$.
\end{remark}

\begin{remark}\label{rmk:delta}
The argument of Proposition \ref{Oleinik_estimate_Heaviside_proposition} fails if the regularisation of the dirac delta distribution $\delta_\eps$ is considered in the Cauchy problem \eqref{CP_gen_reg} instead of $H_{\eps}$. Indeed, in this case we cannot get a constant $C$ independent of $\eps$. In fact,
\begin{align*}
\Vert \delta_\eps (\cdot -1) \Vert_\infty 
= \Vert \delta * \varphi_\eps (\cdot -1) \Vert_\infty 
 = \Vert \varphi_\eps (\cdot -1) \Vert_\infty 
= \left \Vert\frac{1}{\eps}\varphi\left(\frac{\cdot -1}{\eps}\right) \right \Vert_\infty 
 \le\frac{1}{\eps}M_{\varphi},
\end{align*}
where we have used that $\varphi \in C_c^{\infty}(\R)$ and hence it is bounded by $M_{\varphi}$ on $[0,T]$ and that $\delta * \varphi_\eps= \varphi_\eps$, from distribution theory. Therefore, when the coefficient is $\delta$ we cannot achieve boundedness.
\end{remark}

\subsubsection{Heaviside function}\label{sect:heavisideFunction}
We consider again the Cauchy problem \eqref{CP_gen_Oleinik}.
Classically, this Cauchy problem can be solved piecewise. First for $t<1$ and then for $t>1$ and taking the final values at $t=1$ as initial values for $t>1$.
For $t<1$, the equation becomes $u_{tt}=0$. Using the initial conditions we obtain that 
\[u(x,t)=tg_1(x)+g_0(x).\]
Therefore, $u(x,1)=g_1(x)+g_0(x)$ and $u_t(x,1)=g_1(x)$ which are used as initial conditions for $t>1$. For $t>1$, using  d'Alembert's formula we obtain that  
\begin{align*}
u(x,t)=&\frac{g_1(x+t-1)+g_1(x-t+1)+g_0(x+t-1)+g_0(x-t+1)}{2} +\frac{1}{2}\int_{x-t+1}^{x+t-1} g_1(s)ds.
\end{align*}
Combining these two solutions together we obtain the unique piecewise distributional solution $\bar{u}$, to the Cauchy problem \eqref{CP_gen_Oleinik}.

This proof is modelled on the proof in \cite{DHO13}. The new ingredient is that we are dealing now with the Heaviside function where before the jump it is equal to zero, whereas in \cite{DHO13} it was positive.
\begin{theorem} \label{thm_convergence}
For the Cauchy problem \eqref{CP_gen_Oleinik}, every very weak solution converges to the piecewise distributional solution $\bar{u}$, as $\eps \to 0$. 
\end{theorem}

\begin{proof}
From assumption, $g_0,\, g_1 \in \mathcal{C}_c^{\infty}(\R)$. Let $(u_\eps)_\eps$ be a very weak solution of the Cauchy problem \eqref{CP_gen_Oleinik} that vanishes for $x$ outside some compact set, independently of $t$, $t \in [0,T]$. By construction we have that $u_\eps\in\mathcal{C}^{\infty}([0,T], H^{\infty}(\R))$. Therefore we can use energy estimates. We first note that $H(t-1)= H(t-1)^2$, therefore our equation can be written as $u_{tt}-H(t-1)^2 u_{xx}=0$.
Following the argument of the proof of Theorem 6.1 in \cite{DHO13}, we now multiply the regularised equation $u_{\eps,tt}-H_{\eps}(t-1)^2 u_{\eps,xx}=0$ by $u_{\eps,t}$ and integrating by parts we have
\[
\int_\R \left( u_{\eps,tt}u_{\eps,t}+H_{\eps}(t-1)^2 u_{\eps,x}u_{\eps,xt}\right)dx=0.
\]
Noting that 
\[
\frac{1}{2}\frac{\partial}{\partial t} \left( H_{\eps}(t-1)u_{\eps,x} \right)^2 = H_{\eps}(t-1)^2 u_{\eps,x}u_{\eps,xt} + H_{\eps}(t-1)H'_{\eps}(t-1) u_{\eps,x}^2,
\]
we obtain that
\[
\frac{1}{2}\frac{\partial} {\partial t} \int_\R\left( |u_{\eps,t}|^2 + H_{\eps}(t-1)^2 |u_{\eps,x}|^2\right)dx=\int_\R H_{\eps}(t-1)H'_{\eps}(t-1) |u_{\eps,x}|^2 dx.
\]
We therefore have the following energy estimate
\begin{align}\label{energy_estimates_Deguchi}
 \int_\R\left( |u_{\eps,t}|^2 + H_{\eps}(t-1)^2 |u_{\eps,x}|^2\right)dx \le& \int_\R \left( |g_1|^2 + H_{\eps}(t-1)^2|g_{0,x}|^2 \right) dx\\
 &+2\int^t_0 |H_{\eps}(t-1)H'_{\eps}(t-1)|\int_\R |u_{\eps,x}|^2 dxdt. \nonumber
\end{align}
From the estimates obtained by employing Oleinik's result in \eqref{Oleinik_estimate_Heaviside}, we have that $\int_\R|u_{\eps,x}(t,x)|^2dx$ is bounded on $[0,T]$. We also note that,
\begin{align*}
&\int_\R | H_\eps(t-1) H'_\eps(t-1) |dt =\int_\R | H_\eps(t-1) (\delta \ast \varphi_\eps)(t-1)| dt \\
&=\int_\R \left|H_\eps(t-1)\, \frac{1}{\eps} \varphi\left(\frac{t-1}{\eps}\right)\right| dt =\int_\R |H_\eps(\eps y) \varphi(y)| dy\le \int_\R |\varphi(y)| dy < \infty .
\end{align*}
Noting that
\begin{align*}
\int_\R |u_{\eps,t}|^2 dx \le \int_\R\left( |u_{\eps,t}|^2 + H_{\eps}(t-1)^2 |u_{\eps,x}|^2\right)dx
\end{align*}
and applying the energy estimates \eqref{energy_estimates_Deguchi}, we conclude that $\int_\R|u_{\eps,t}(t,x)|^2dx$ is bounded on $[0,T]$. We therefore have that $(u_\eps)_\eps$ is bounded in $\mathcal{C}^1([0,T], L^2(\R))$. By the Mean Value Theorem, $(u_\eps)_\eps$ is equicontinuous in  $\mathcal{C}([0,T], L^2(\R))$. Since $(u_\eps)_\eps$ is also bounded in $\mathcal{C}([0,T], L^2(\R))$, by Arzela-Ascoli, $(u_\eps)_\eps$ is relatively compact in $\mathcal{C}([0,T], L^2(\R))$. Furthermore, by \eqref{Oleinik_estimate_Heaviside} we have that  $\int_\R|u_{\eps,xx}(t,x)|^2dx$ is bounded on $[0,T]$. By the differential equation we get that $\int_\R|u_{\eps,tt}(t,x)|^2dx$ is also bounded on $[0,T]$. By the same argument as above, we obtain that $(u_\eps)_\eps$ is relatively compact in $\mathcal{C}^1([0,T], L^2(\R))$.
Therefore, there exists a subsequence $(u_{\eps_k})_k$ such that
\[
\lim_{k \to \infty} (u_{\eps_k})_k = \tilde{u} \in  \mathcal{C}^1([0,T], L^2(\R)) \subset \mathcal{C}^1([0,T], \mathcal{D}'(\R)) .
\]
But on every compact subinterval of $(0,1)$ and of $(1,T)$, $H_\eps$ is identically equal to 0 or 1, when $\eps$ is small enough. Therefore $\tilde{u}$ is a distributional solution of the Cauchy problem \eqref{CP_gen_reg} on both strips. Since $\tilde{u} \in \mathcal{C}([0,T],\mathcal{D}'(\R))$, so is $\tilde{u}_{xx}$ and hence from the equation, so is also $\tilde{u}_{tt}$. We therefore have,
\[
\tilde{u} \in \left( \mathcal{C}^2([0,1], \mathcal{D}'(\R))\oplus \mathcal{C}^2([1,T], \mathcal{D}'(\R)) \right).
\]
That is, $\tilde{u}$ is the unique piecewise distributional solution to the Cauchy problem \eqref{CP_gen_reg}. Therefore $\tilde{u}=\bar{u}$. Since $(u_\eps)_\eps$ is equicontinuous in  $\mathcal{C}^1([0,T], L^2(\R))$, we have that the whole net $(u_\eps)_\eps$ converges to $\tilde{u}=\bar{u}$.
\end{proof}

\subsubsection{Numerical model}
We solve the Cauchy problem \eqref{CP_gen_reg} numerically using the Lax–Friedrichs method (see, e.g., \cite{LV92}) after reformulating it as an equivalent first-order system. We consider $t \in [0.8,2]$, $x \in [0,2]$, and periodic initial conditions $g_0(x) = \sin(2\pi x)$ and $g_1(x) = \cos(2\pi x)$. For the space discretisation, we set the discretisation step $\Delta x= 0.0007$, and for the time discretisation we choose $\Delta t= \frac{\eps}{10}$ when $H_\eps (t-1)=0$ and we take $\Delta t$ adaptively when $H_\eps (t-1)\not=0$ to ensure that the Courant number remains equal to $1$. We compute numerical solutions $u_\eps$ for various values of $\eps$ and we compare them with the piecewise distributional solution $\bar{u}$ obtained as in section \ref{sect:heavisideFunction}. More precisely, we compute the norm $\|u_\eps - \bar{u} \|_{L^2([0,2])}$ at $t=2$ and show that it tends to $0$ for $\eps \to 0$ according to Theorem \ref{thm_convergence}.

First, we consider the mollifier $\varphi_{1,\eps}(t)=\frac{1}{\eps}\,\varphi_1(\frac{t}{\eps})$ with
\[
\varphi_{1}(t) = 
\begin{cases*}
\frac{1}{0.443994} e^{\frac{1}{(t^2 -1)}}, & if $\lvert t \rvert < 1$  \\
                                        0, & if $\lvert t \rvert \ge 1.$
\end{cases*}
\]
Figure \ref{fig:numerics1} shows the exact solution $\bar{u}$ and the solutions $u_\eps$ at $t=2$ for $\eps = 0.1$ and $\eps=0.05$ (left), and the computed $L^2$ error norm for various values of $\eps$ (right). We can see that the solutions $u_\eps$ better approach the exact solution $\bar{u}$ as $\eps$ is reduced and that the error norm decreases for $\eps \to 0$ as expected.

\begin{figure}[bht]
\centerline{%
\includegraphics[width=0.49\textwidth]{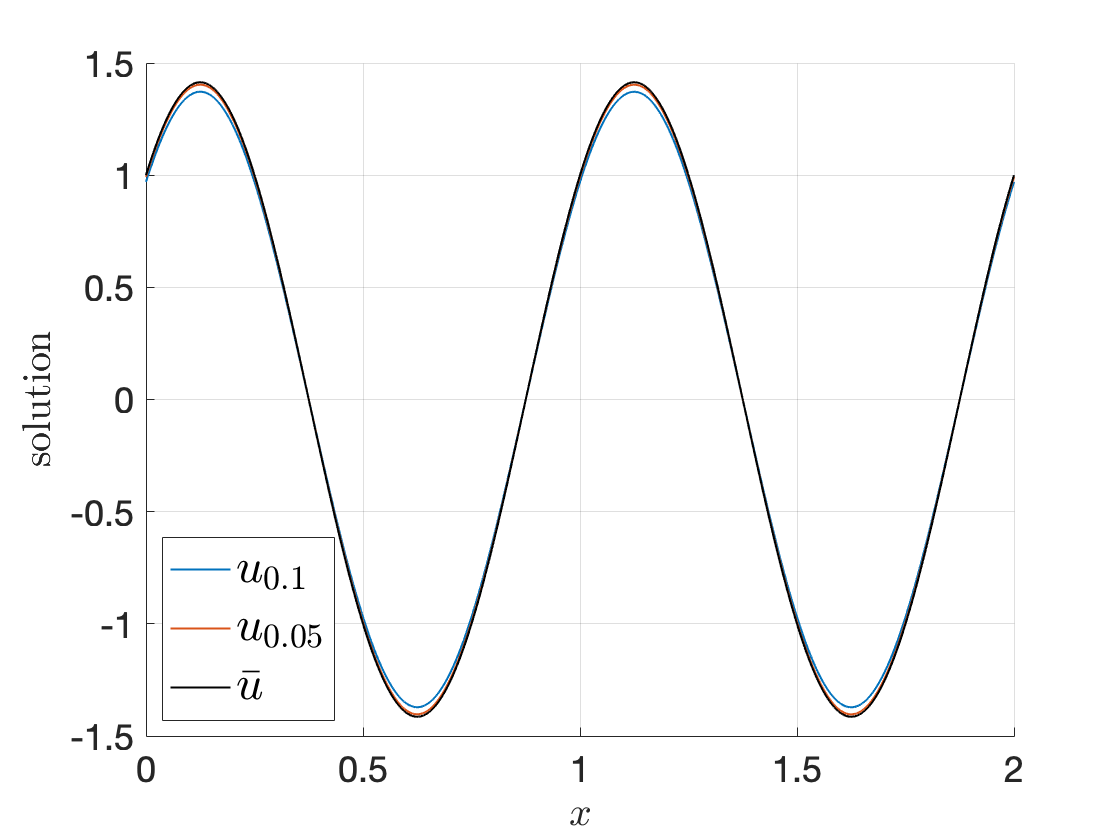}
\hfill
\includegraphics[width=0.49\textwidth]{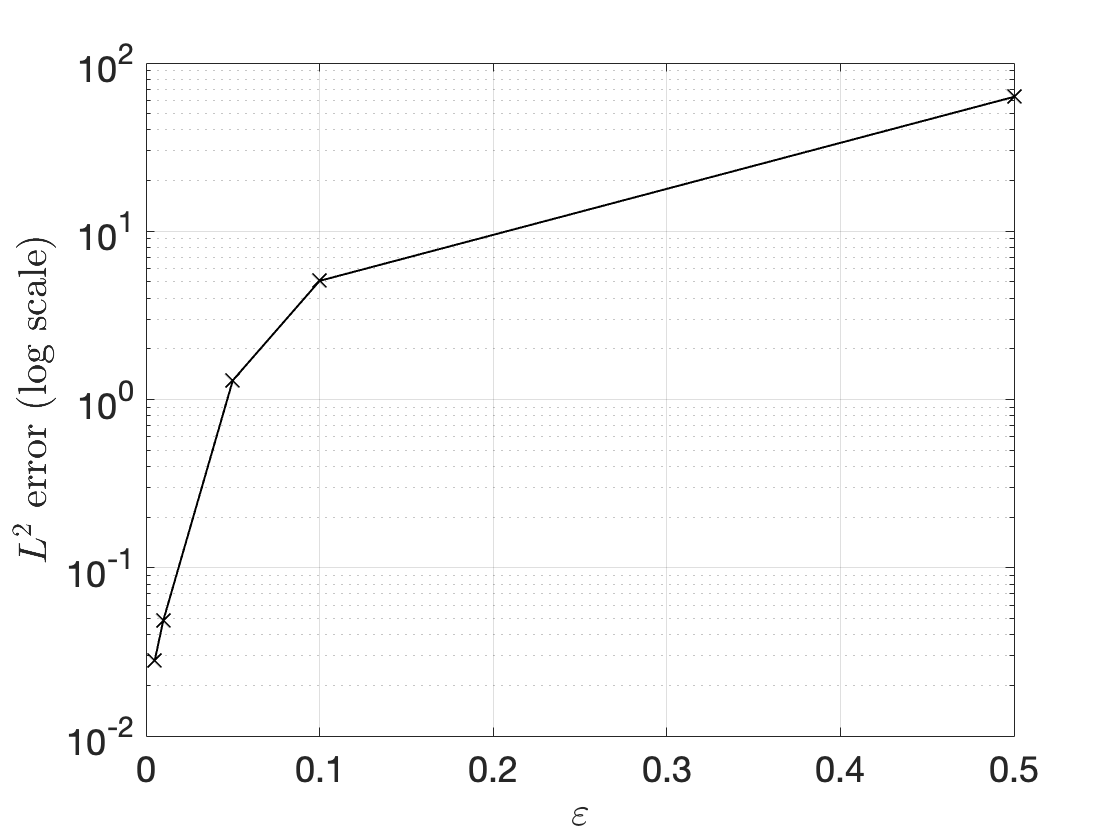}
}
\caption{Numerical tests with mollifier $\varphi_{1,\eps}(t)$: exact solution $\bar{u}$ and numerical approximations $u_\eps$ for $\eps=0.1,0.05$ (left); norm $\|u_\eps - \bar{u} \|_{L^2([0,2])}$ at $t=2$ versus $\eps$ (right).}
\label{fig:numerics1}
\end{figure}

To verify that results are independent of the choice of the mollifier, we repeated the computations taking now the mollifier $\varphi_{2,\eps}(t)=\frac{1}{\eps}\,\varphi_2(\frac{t}{\eps})$ with
\[
\varphi_{2}(t) = 
\begin{cases*}
\frac{1}{0.887988} e^{\frac{1}{(t^2 -1)}}, & if $\lvert t \rvert < 2$, \\
                                        0, & if $\lvert t \rvert\ge 2.$
\end{cases*}
\]

Also in this case, the solutions $u_\eps$ better approach the exact solution $\bar{u}$ when $\eps \to 0$ and the norm $\|u_\eps - \bar{u} \|_{L^2([0,2])}$ tends to $0$ for $\eps \to 0$ as shown in Figure \ref{fig:numerics2}.

\begin{figure}[bht]
\centerline{%
\includegraphics[width=0.49\textwidth]{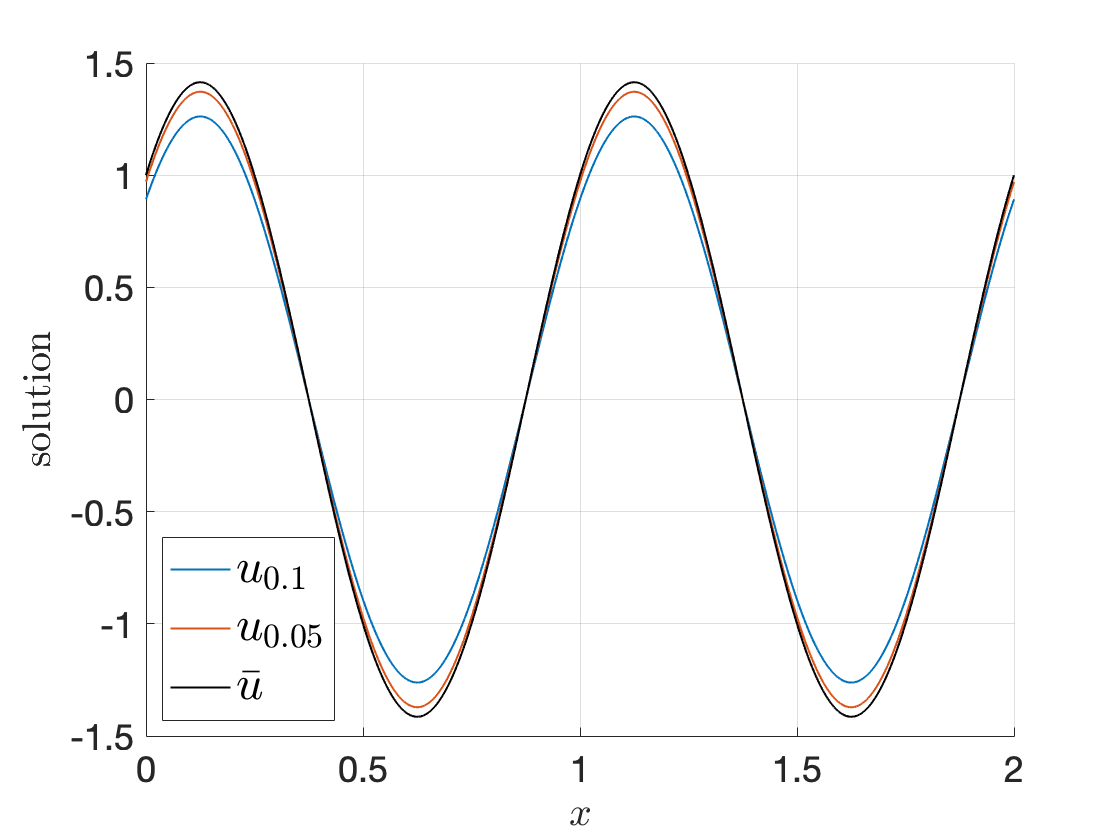}
\hfill
\includegraphics[width=0.49\textwidth]{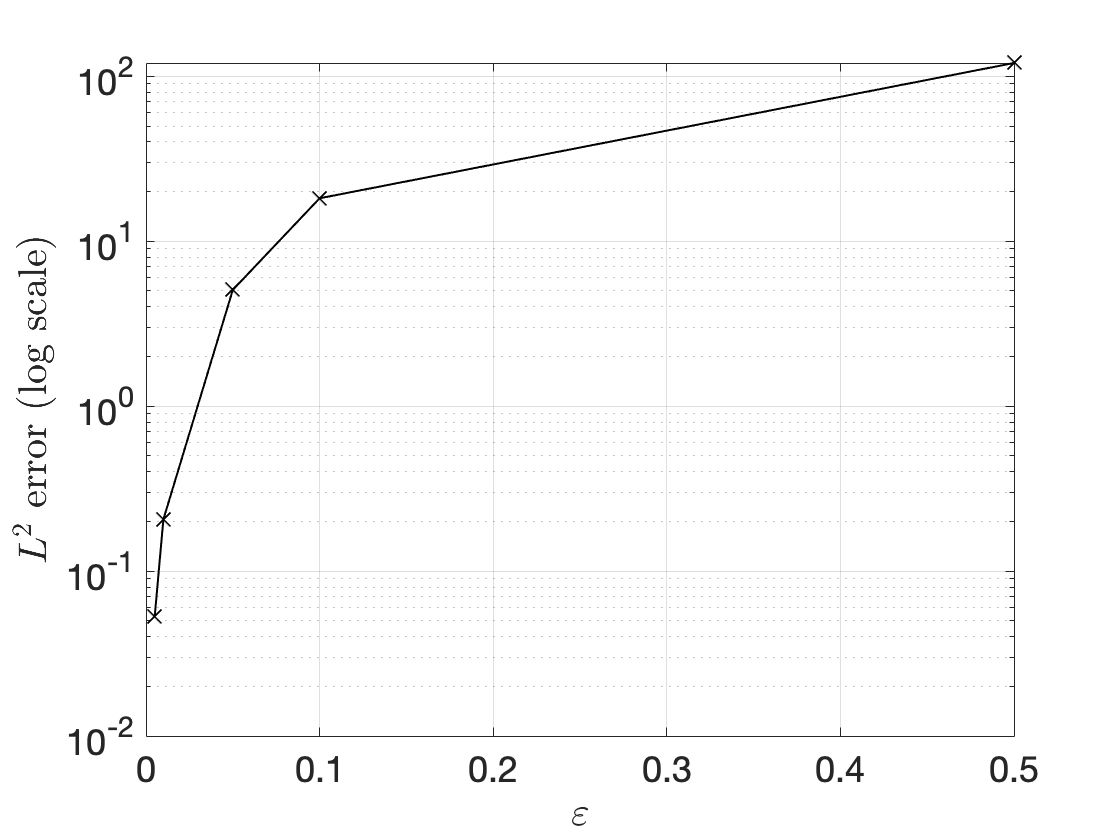}
}
\caption{Numerical tests with mollifier $\varphi_{2,\eps}(t)$: exact solution $\bar{u}$ and numerical approximations $u_\eps$ for $\eps=0.1,\,0.05$ (left); norm $\|u_\eps - \bar{u} \|_{L^2([0,2])}$ at $t=2$ versus $\eps$ (right).}
\label{fig:numerics2}
\end{figure}

\subsection{Second model: $a(t)=\delta$}

Finally, we numerically study the problem \eqref{CP_gen_reg} where the Heaviside function $H_\eps$ is replaced by a regularised delta distribution $\delta_\eps$. In this case, according to Remark \ref{rmk:delta}, the argument of Proposition \ref{Oleinik_estimate_Heaviside_proposition} does not hold so that, in particular, it is not possible to identify a constant $C$ independent of $\eps$ such that inequality \eqref{Oleinik_estimate_Heaviside} is satisfied. Considering the same setting and discretisation as in the previous tests, we solve \eqref{CP_gen_reg} with a delta distribution $\delta_\eps$ regularised using the mollifier $\varphi_{1,\eps}$ to obtain numerical solutions $u_\eps$ for various values of $\eps$. Then, we compute the ratio
\begin{equation}\label{eq:ratioDelta}
\frac{\|u_\eps(\tau,\cdot)\|_{{L^2}([0,2])}^2}{\|g_0\|_{H^5([0,2])}^2 + \|g_1\|_{H^4([0,2])}^2}
\end{equation}
at $\tau=2$ and we show that this ratio is not bounded independently of $\eps$ when $\eps \to 0$. The ratios computed for several values of $\eps$ are plotted in Fig. \ref{fig:numerics3}, and they seem to confirm that boundedness cannot be achieved in the case of coefficients as singular as delta.  

\begin{figure}[!htb]
\centerline{%
\includegraphics[width=0.49\textwidth]{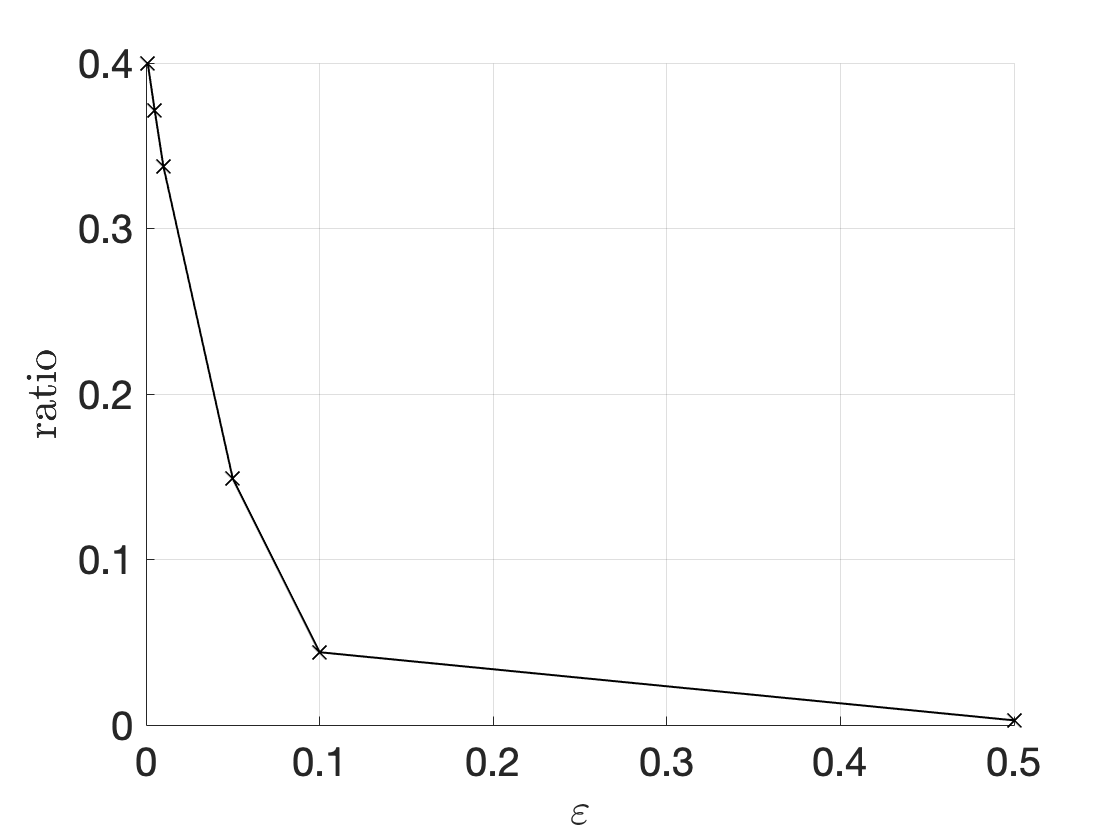}
}
\caption{Ratios \eqref{eq:ratioDelta} computed for different values of $\eps$ with $u_\eps$ the solution of \eqref{CP_gen_reg} with coefficient $\delta_\eps$.}
\label{fig:numerics3}
\end{figure}

\section{appendix: Quasi-symmetriser} \label{section _quasi}

This appendix is devoted to the general definition of quasi-symmetriser for a matrix in Sylvester form. We refer the reader to \cite{KS, GR:12} for more details.

In the sequel, $A(\lambda)$ is an $m\times m$ Sylvester matrix with real eigenvalues $\lambda_l$, i.e.,
\[
A(\lambda)=\left(
    \begin{array}{ccccc}
      0 & 1 & 0 & \dots & 0\\
      0 & 0 & 1 & \dots & 0 \\
      \dots & \dots & \dots & \dots & 1 \\
      -\sigma_m^{(m)}(\lambda) & -\sigma_{m-1}^{(m)}(\lambda) & \dots & \dots & -\sigma_1^{(m)}(\lambda) \\
    \end{array}
  \right),
\]
where
\[
\sigma_h^{(m)}(\lambda)=(-1)^h\sum_{1\le i_1<...<i_h\le m}\lambda_{i_1}...\lambda_{i_h}
\]
for all $1\le h\le m$. In the sequel $\mathcal{P}_m$ is the class of permutations of $\{1,\,\dots,\,m\}$, $\lambda_\rho=(\lambda_{\rho_1},\,\dots,\,\lambda_{\rho_m})$ with $\lambda\in\R^m$ and $\rho\in\mathcal{P}_m$, $\pi_i\lambda=(\lambda_1,\,\dots,\,\lambda_{i-1},\,\lambda_{i+1},\,\dots,\,\lambda_m)$ and $\lambda'=\pi_m\lambda=(\lambda_1,\,\dots,\,\lambda_{m-1})$. Following Section 4 in \cite{KS}, we have the following definition.

\begin{definition}
The quasi-symmetriser of $A(\lambda)$ is the Hermitian matrix
\[
Q^{(m)}_\delta(\lambda)=\sum_{\rho\in\mathcal{P}_m} P_\delta^{(m)}(\lambda_\rho)^\ast P_\delta^{(m)}(\lambda_\rho),
\]
where $\delta\in(0,1]$, $P_\delta^{(m)}(\lambda)=H^{(m)}_\delta P^{(m)}(\lambda)$, $H_\delta^{(m)}={\rm diag}\{\delta^{m-1},\,\dots,\,\delta,1\}$ and the matrix $P^{(m)}(\lambda)$ is defined inductively by $P^{(1)}(\lambda)=1$ and
\[
P^{(m)}(\lambda)=\left(
    \begin{array}{ccccc}
      \, & \, & \, & \, & 0\\
      \, & \, & P^{(m-1)}(\lambda') & \, & \vdots \\
      \, & \, & \, & \, & 0 \\
      \sigma_{m-1}^{(m-1)}(\lambda') & \dots & \dots & \sigma_1^{(m-1)}(\lambda') & 1 \\
    \end{array}
  \right).
\]
\end{definition}

Note that $P^{(m)}(\lambda)$ is depending only on $\lambda'$. 

{\bf Notations:} $W^{(m)}_i(\lambda)$ denotes the row vector
\[
\big(\sigma_{m-1}^{(m-1)}(\pi_i\lambda),\, \dots, \,\sigma_1^{(m-1)}(\pi_i\lambda),\,1\big),\quad 1\le i\le m,
\]
and $\mathcal{W}^{(m)}(\lambda)$ the matrix with row vectors $W^{(m)}_i$.

In the next proposition we collect the main properties of the quasi-symmetriser $Q^{(m)}_\delta(\lambda)$. For a detailed proof we refer the reader to Propositions 1 and 2 in \cite{KS} and to Proposition 1 in \cite{dASpa:98}. Note that for $m\times m$ matrices $A_1$ and $A_2$ the notation $A_1\le A_2$ means $(A_1v,v)\le (A_2v,v)$ for all $v\in\C^m$ with $(\cdot,\cdot)$ the scalar product in $\C^m$.

\begin{proposition}
\label{prop_qs}
\leavevmode
\begin{itemize}
\item[(i)] The quasi-symmetriser $Q_\delta^{(m)}(\lambda)$ can be written as
\[
Q_0^{(m)}(\lambda)+\delta^2 Q_1^{(m)}(\lambda)+...+\delta^{2(m-1)}Q_{m-1}^{(m)}(\lambda),
\]
where the matrices $Q^{(m)}_i(\lambda)$, $i=1,\,\dots,\,m-1,$ are non-negative and Hermitian with
entries being symmetric polynomials in $\lambda_1,\,\dots,\,\lambda_m$.
\item[(ii)] There exists a function $C_m(\lambda)$ bounded for
bounded $|\lambda|$ such that
\[
C_m(\lambda)^{-1}\delta^{2(m-1)}I\le Q^{(m)}_\delta(\lambda)\le C_m(\lambda)I.
\]
\item[(iii)] We have
\[
|Q_\delta^{(m)}(\lambda)A(\lambda)-A(\lambda)^\ast Q_\delta^{(m)}(\lambda)|\le C_m(\lambda)\delta Q_\delta^{(m)}(\lambda).
\]
\item[(iv)] For any $(m-1)\times(m-1)$ matrix $T$ let $T^\sharp$ denote the $m\times m$ matrix
\[
\left(
    \begin{array}{cc}
    T & 0\\
    0 & 0 \\
    \end{array}
  \right).
\]
Then, $Q_\delta^{(m)}(\lambda)=Q_0^{(m)}(\lambda)+\delta^2\sum_{i=1}^m Q_\delta^{(m-1)}(\pi_i\lambda)^\sharp$.
\item[(v)] We have
\[
Q_0^{(m)}(\lambda)=(m-1)!\mathcal{W}^{(m)}(\lambda)^\ast \mathcal{W}^{(m)}(\lambda).
\]
\item[(vi)] We have
\[
\det Q_0^{(m)}(\lambda)=(m-1)!\prod_{1\le i<j\le m}(\lambda_i-\lambda_j)^2.
\]
\item[(vii)] There exists a constant $C_m$ such that
\[
q_{0,11}^{(m)}(\lambda)\cdots q_{0,mm}^{(m)}(\lambda)\le C_m\prod_{1\le i<j\le m}(\lambda^2_i+\lambda^2_j).
\]
\end{itemize}
\end{proposition}
We finally recall the definition of \emph{nearly diagonal} matrices.
\begin{definition}
A family $\{Q_\alpha\}$ of nonnegative Hermitian matrices is called \emph{nearly diagonal} if there exists a positive constant $c_0$ such that
\[
Q_\alpha\ge c_0\,{\rm diag}\,Q_\alpha
\]
for all $\alpha$, with $${\rm diag}\,Q_\alpha ={\rm diag}\{q_{\alpha,11},\,\dots,\,q_{\alpha, mm}\}.$$
\end{definition}

The following linear algebra result is proven in \cite[Lemma 1]{KS}.
\begin{lemma}
\label{lem_old}
Let $\{Q_\alpha\}$ be a family of nonnegative Hermitian $m\times m$ matrices such that $\det Q_\alpha>0$ and
\[
\det Q_\alpha \ge c\, q_{\alpha,11}q_{\alpha,22}\cdots q_{\alpha,mm}
\]
for a certain constant $c>0$ independent of $\alpha$. Then,
\[
Q_\alpha\ge c\, m^{1-m}\,{\rm diag}\,Q_\alpha
\]
for all $\alpha$, i.e., the family $\{Q_\alpha\}$ is nearly diagonal.
\end{lemma}
Lemma \ref{lem_old} is employed to prove that the family  $Q_\delta^{(m)}(\lambda)$ of quasi-sym\-me\-tri\-sers defined above is nearly diagonal when $\lambda$ belongs to a suitable set. The following statement is proven in \cite[Proposition 3]{KS}.
\begin{proposition}
\label{prop_SM}
For any $M>0$ define the set
\[
\mathcal{S}_M=\{\lambda\in\R^m:\, \lambda_i^2+\lambda_j^2\le M (\lambda_i-\lambda_j)^2,\quad 1\le i<j\le m\}.
\]
Then the family of matrices $\{Q_\delta^{(m)}(\lambda):\, 0<\delta\le 1, \lambda\in\mathcal{S}_M\}$ is nearly diagonal.
\end{proposition}
Using this we get the following corollary.
\begin{corollary} \label{coroll_nearly_diagonal}
Under the Kinoshota-Spagnolo condition 
\begin{align*}
\label{LC}
&\exists M>0:\qquad 
\lambda_i(t,\xi)^2+\lambda_j(t,\xi)^2\le M(\lambda_i(t,\xi)-\lambda_j(t,\xi))^2,\\
&\text{for}\ 1\le i,j\le m, t\in[0,T],  \textrm{ for all }\xi,\nonumber
\end{align*}
on the roots of the equation, the quasi-symmetriser is nearly diagonal.
\end{corollary}

We conclude this appendix with a result on nearly diagonal matrices depending on three parameters (i.e. $\delta, t, \xi$)
which will is crucial when dealing with regularised coefficients. Note that this is a straightforward
extension of Lemma 2 in \cite{KS} valid for two parameter
(i.e. $\delta, t$) dependent matrices.
\begin{lemma}
\label{lem_new}
Let $\{ Q^{(m)}_\delta(t,\xi): 0<\delta\le 1, 0\le t\le T, \xi\in\R^n\}$ be a nearly diagonal family of coercive Hermitian matrices of class ${C}^k$ in $t$, $k\ge 1$. Then, there exists a constant $C_T>0$ such that for any non-zero
continuous function $V:[0,T]\times\R^n\to \C^m$
we have
\[
\int_{0}^T \frac{|(\partial_t Q^{(m)}_\delta(t,\xi)
V(t,\xi),V(t,\xi))|}{(Q^{(m)}_\delta(t,\xi)V(t,\xi),V(t,\xi))^{1-1/k}
|V(t,\xi)|^{2/k}}\, dt\le C_T
\Vert Q^{(m)}_\delta(\cdot,\xi)\Vert^{1/k}_{{C}^k([0,T])}
\]
for all $\xi\in\Rn.$
\end{lemma}

\end{document}